\documentclass{amsart}
\usepackage[T1]{fontenc}
\usepackage[utf8]{inputenc}
\usepackage{lmodern}
\usepackage{amsmath}
\usepackage{amssymb}
\usepackage{stmaryrd}
\usepackage{mathrsfs}
\usepackage{tikz}
\usetikzlibrary{patterns}
\usepackage{empheq}
\usepackage{bbm}
\usepackage{hyperref}
\hypersetup{colorlinks=true,
            linkcolor=green!30!black,
            citecolor=cyan!50!black}

\usepackage{amsthm}
\newtheorem{thm}{Theorem}[section]
\newtheorem{prop}[thm]{Proposition}
\newtheorem{lemma}[thm]{Lemma}
\newtheorem{cor}[thm]{Corollary}
\newtheorem{fact}[thm]{Fact}

\numberwithin{equation}{section}

\usepackage{vmargin}
\newcommand{\N}{\mathbb{N}}
\newcommand{\Z}{\mathbb{Z}}
\providecommand{\root}{}
\renewcommand{\root}{\ensuremath{\varnothing}}
\newcommand{\bst}[1]{T\langle#1\rangle}
\newcommand{\p}{\ensuremath{\mathbb{P}}}
\newcommand{\mt}{\ensuremath{\mathrm{MT}}}
\usepackage[normalem]{ulem}
\newcommand{\R}{\ensuremath{\mathbb{R}}}

\DeclareRobustCommand{\SkipTocEntry}[5]{} 

\title{The Height of Mallows Trees}
\author{Louigi Addario-Berry \& Benoît Corsini}
\address{Department of Mathematics and Statistics, McGill University, Montr\'eal, Canada}
\email{louigi.addario@mcgill.ca,benoit.corsini@mail.mcgill.ca}
\date{July 27, 2020}

\begin{document}
\maketitle

\begin{abstract}
    Random binary search trees are obtained by recursively inserting the elements $\sigma(1),\sigma(2),\ldots,\sigma(n)$ of a uniformly random permutation $\sigma$ of $[n]=\{1,\dots,n\}$ into a binary search tree data structure. Devroye (1986) proved that the height of such trees is asymptotically of order $c^*\log n$, where $c^*=4.311\ldots$ is the unique solution of $c \log((2e)/c)=1$ with $c \geq 2$. In this paper, we study the structure of binary search trees $T_{n,q}$ built from Mallows permutations. A $\textrm{Mallows}(q)$ permutation is a random permutation of $[n]=\{1,\ldots,n\}$ whose probability is proportional to $q^{\textrm{Inv}(\sigma)}$, where $\textrm{Inv}(\sigma) = \#\{i < j: \sigma(i) > \sigma(j)\}$. This model generalizes random binary search trees, since $\textrm{Mallows}(q)$ permutations with $q=1$ are uniformly distributed. The laws of $T_{n,q}$ and $T_{n,q^{-1}}$ are related by a simple symmetry (switching the roles of the left and right children), so it suffices to restrict our attention to $q\leq1$.
    
    We show that, for $q\in[0,1]$, the height of $T_{n,q}$ is asymptotically $(1+o(1))(c^* \log n + n(1-q))$ in probability. This yields three regimes of behaviour for the height of $T_{n,q}$, depending on whether $n(1-q)/\log n$ tends to zero, tends to infinity, or remains bounded away from zero and infinity. In particular, when $n(1-q)/\log n$ tends to zero, the height of $T_{n,q}$ is asymptotically of order $c^*\log n$, like it is for random binary search trees. Finally, when $n(1-q)/\log n$ tends to infinity, we prove stronger tail bounds and distributional limit theorems for the height of $T_{n,q}$.
\end{abstract}
\tableofcontents

\section{Introduction}\label{sec:Intro}

Let $T_\infty = \{\root\} \cup \bigcup_{k \ge 0} \{\overline{0},\overline{1}\}^k$ be the complete infinite rooted binary tree, with nodes at depth $n\ge 1$ indexed by strings  $u=u_1,\ldots,u_k \in \{\overline{0},\overline{1}\}^k$, so $u$ has parent $u_1,\ldots,u_{k-1}$ and children $u\overline{0}$ and $u\overline{1}$. For a set $V \subset  T_\infty$ and node $u \in T_\infty$, we write $uV = \{uv,v \in V\}$. 

For $u \in \{\overline{0},\overline{1}\}^k$, we write $|u|=k$ and say that $u$ has depth $k$. A subtree of $T_\infty$ (or just ``a tree'', for short) is a set $T \subset T_\infty$ which is connected when viewed as a subgraph of $T_\infty$. For any subtree $T$ of $T_\infty$, the {\em root} of $T$ is defined to be the unique element of $T$ of minimum depth. 
For a tree $T$ and a node $u \in T_{\infty}$, we write $T(u) = (uT_\infty) \cap T$ for the subtree of $T$ rooted at $u$; when $\root\in T$, then $T(u)=\emptyset$ if and only if $u\notin T$. 
Finally, for $T \subset T_\infty$, we write $h(T) = \sup(|u|,u \in T)-\inf(|u|,u \in T)$; if $T$ is a tree then $h(T)$ is the greatest distance of any node of $T$ from the root of $T$. 

For $n \ge 1$ we write $[n]=\{1,2,\ldots,n\}$. Given an injective function $f:[n] \to \Z_+:=\{1,2,\ldots\}$, the {\em binary search tree} $\bst{f}$ is the subtree of $T_\infty$ defined inductively as follows (see Figure~\ref{fig:BST} for an example). If $n=0$, then $\bst{f}:=\emptyset$ is the empty tree. 
Otherwise, view $f$ as a sequence of distinct integers $f=\big(f(1),f(2),...,f(n)\big)$, and write $f^-$ (respectively $f^+$) for the subsequence of $f$ consisting of terms $f(i)$ such that $f(i)<f(1)$ (respectively $f(i)>f(1)$), listed in the same order as in $f$.
Then set $\bst{f}:=\{\root\}\cup\big(\overline{0}\bst{f^-}\big)\cup\big(\overline{1}\bst{f^+}\big)$. 

We label the nodes of $\bst{f}$ by the elements of $\{f(1),f(2),...,f(n)\}$ as follows.  Set $\tau\langle f\rangle(\root)=f(1)$; then, inductively, for nodes $u\in \bst{f}$ with $|u| \ge 1$, set 
\begin{align*}
    \tau\langle f\rangle(u)&=
    \begin{cases}
        \tau\langle f^-\rangle(v) & \textrm{if $u=\overline{0}v$} \\
        \tau\langle f^+\rangle(v) & \textrm{if $u=\overline{1}v$}
    \end{cases}
\end{align*}
The definitions of $\bst{f}$ and $\tau\langle f\rangle$ easily extend to injective functions $f:\Z^+\mapsto\Z^+$, by considering the sequence $\big(f(i),i \ge 1\big)$. 

\begin{figure}[ht]
    \centering
    \begin{tikzpicture}[scale=0.6]
        \draw[very thin, gray!50!white] (0,0) grid (13,9);
        \node at (-0.3,9.3){$0$};
        \node at (-0.5,0){$i$};
        \node at (13,9.5){$f_i$};
        \draw[->] (0,9)--(13.5,9);
        \draw[->] (0,9)--(0,-0.5);
        \node(1) at (4,8){};
        \node(2) at (1,7){};
        \node(3) at (9,6){};
        \node(4) at (7,5){};
        \node(5) at (2,4){};
        \node(6) at (6,3){};
        \node(7) at (12,2){};
        \node(8) at (8,1){};
        \draw[ultra thick] (1.center)--(2.center);
        \draw[ultra thick] (1.center)--(3.center);
        \draw[ultra thick,blue!70!black] (2.center)--(5.center);
        \draw[ultra thick,red!70!black] (3.center)--(4.center);
        \draw[ultra thick,red!70!black] (3.center)--(7.center);
        \draw[ultra thick,red!70!black] (4.center)--(6.center);
        \draw[ultra thick,red!70!black] (4.center)--(8.center);
        \node(a) at (-0.2,8){};
        \node(b) at (-0.2,7){};
        \node(c) at (-0.2,6){};
        \node(d) at (-0.2,5){};
        \node(e) at (-0.2,4){};
        \node(f) at (-0.2,3){};
        \node(g) at (-0.2,2){};
        \node(h) at (-0.2,1){};
        \node[scale=0.8] at (-0.4,8){$4$};
        \node[scale=0.8,blue!70!black] at (-0.4,7){$1$};
        \node[scale=0.8,red!70!black] at (-0.4,6){$9$};
        \node[scale=0.8,red!70!black] at (-0.4,5){$7$};
        \node[scale=0.8,blue!70!black] at (-0.4,4){$2$};
        \node[scale=0.8,red!70!black] at (-0.4,3){$6$};
        \node[scale=0.8,red!70!black] at (-0.4,2){$12$};
        \node[scale=0.8,red!70!black] at (-0.4,1){$8$};
        \draw[dashed,thick] (a.east)--(1.center);
        \draw[dashed,blue!70!black,thick] (b.east)--(2.center);
        \draw[dashed,red!70!black,thick] (c.east)--(3.center);
        \draw[dashed,red!70!black,thick] (d.east)--(4.center);
        \draw[dashed,blue!70!black,thick] (e.east)--(5.center);
        \draw[dashed,red!70!black,thick] (f.east)--(6.center);
        \draw[dashed,red!70!black,thick] (g.east)--(7.center);
        \draw[dashed,red!70!black,thick] (h.east)--(8.center);
        \node[draw,circle,ultra thick,scale=1.5,fill=white] at (4,8){};
        \node[draw,circle,ultra thick,scale=1.5,blue!70!black,fill=white] at (1,7){};
        \node[draw,circle,ultra thick,scale=1.5,red!70!black,fill=white] at (9,6){};
        \node[draw,circle,ultra thick,scale=1.5,red!70!black,fill=white] at (7,5){};
        \node[draw,circle,ultra thick,scale=1.5,blue!70!black,fill=white] at (2,4){};
        \node[draw,circle,ultra thick,scale=1.5,red!70!black,fill=white] at (6,3){};
        \node[draw,circle,ultra thick,scale=1.5,red!70!black,fill=white] at (12,2){};
        \node[draw,circle,ultra thick,scale=1.5,red!70!black,fill=white] at (8,1){};
        \node at (4,8){$\mathbf{4}$};
        \node at (1,7){$\mathbf{1}$};
        \node at (9,6){$\mathbf{9}$};
        \node at (7,5){$\mathbf{7}$};
        \node at (2,4){$\mathbf{2}$};
        \node at (6,3){$\mathbf{6}$};
        \node at (12,2){$\mathbf{12}$};
        \node at (8,1){$\mathbf{8}$};
        \node[scale=1.5] at (11,8){$\boldsymbol{T_f}$};
        \node[scale=1.2,red!70!black] at (9,3){$\boldsymbol{T_{f^+}}$};
        \node[scale=1.2,blue!70!black] at (2.5,5.5){$\boldsymbol{T_{f^-}}$};
    \end{tikzpicture}
    \caption{The labelled tree $\big(\bst{f},\tau\langle f\rangle\big)$ for $f=(4,\textcolor{blue!70!black}{1},\textcolor{red!70!black}{9},\textcolor{red!70!black}{7},\textcolor{blue!70!black}{2},\textcolor{red!70!black}{6},\textcolor{red!70!black}{12},\textcolor{red!70!black}{8})$. The sequences $f^-$ and $f^+$ are $(\textcolor{blue!70!black}{1},\textcolor{blue!70!black}{2})$ and $(\textcolor{red!70!black}{9},\textcolor{red!70!black}{7},\textcolor{red!70!black}{6},\textcolor{red!70!black}{12},\textcolor{red!70!black}{8})$ respectively. The subtree in blue corresponds to $\bst{f^-}$ and the one in red corresponds to $\bst{f^+}$. The corresponding labels given by $\tau\langle f\rangle$ are written on the nodes; so, for example, $\tau\langle f\rangle(\overline{0}\overline{1})=2$ and $\tau\langle f\rangle(\overline{1}\overline{1})=12$.}
\end{figure}\label{fig:BST}

In this article, we study the heights of binary search trees built from random, Mallows-distributed permutations. 
For $n\geq0$ and $q\in[0,\infty)$, the Mallows distribution with parameters $n$ and $q$ (introduced in \cite{mallows1957non}) is the probability measure $\pi_{n,q}$ on the symmetric group $\mathcal{S}_n$ given by
\begin{align*}
    \pi_{n,q}(\sigma)&:=Z_{n,q}^{-1}\cdot q^{\textrm{Inv}(\sigma)}\, .
\end{align*}
Here $\textrm{Inv}(\sigma):=\big|\big\{1\leq i<j\leq n:\sigma(i)>\sigma(j)\big\}\big|$ is the number of inversions of $\sigma$ and $Z_{n,q}:=\sum_{\sigma\in\mathcal{S}_n}q^{\textrm{Inv}(\sigma)}$ is a normalizing constant. 

For a permutation $\sigma=(\sigma(1),\ldots,\sigma(n))$, the reversed permutation $\sigma'=(n+1-\sigma(1),\ldots,n+1-\sigma(n))$ has 
$\textrm{Inv}(\sigma') = \binom{n}{2}-\textrm{Inv}(\sigma)$. This implies that, if $\sigma$ is a $\pi_{n,q}$-distributed random permutation, then its reversal is $\pi_{n,\frac{1}{q}}$-distributed. The effect of this reversal on the associated binary search trees is also easy to understand: $T_{\sigma'}$ is obtained from $T_\sigma$ by swapping all left and right subtrees.
Since the map $q\mapsto\frac{1}{q}$ bijectively sends $(1,\infty)$ to $(0,1)$, it follows from these observations that we may as well restrict our attention to $q\in[0,1]$. Note that when $q=0$, $\pi_{n,q}$ assigns weight $1$ to the identity permutation and when $q=1$, $\pi_{n,q}$ is the uniform distribution on $\mathcal{S}_n$.

We prove the following results. In what follows, we write $T_{n,q}$ for a random tree with the distribution of $\bst{\sigma}$ for $\sigma$ a $\pi_{n,q}$-distributed random permutation, and we write $\mt(n,q)$ for the law of such a tree; we call $T_{n,q}$ a {\em Mallows tree} (with parameters $n$ and $q$). Also, we let $c^*=4.311\ldots$ be the unique solution of $c\log\left(\frac{2e}{c}\right)=1$ with $c\geq2$. 
\begin{thm}\label{thm:Prob}
    For any $[0,1]$-valued sequence $(q_n)_{n\geq0}$, 
    \begin{align*}
        \frac{h(T_{n,q_n})}{n(1-q_n)+c^*\log n}\to 1
    \end{align*}
   in probability and in $L_p$ for any $p > 0$. 
\end{thm}
When $q_n=1$ for all $n$, the trees $T_{n,q_n}$ are {\em random binary search trees} - the binary search trees corresponding to uniformly random permutations. This case of Theorem~\ref{thm:Prob} implies that $\frac{h(T_{n,1})}{c^*\log n}\rightarrow1$ in probability, which is a well-known result of Devroye~\cite{devroye1986note}.  

On the other hand, when $q_n=q\in[0,1)$ for all $n$, Theorem~\ref{thm:Prob} implies that $h(T_{n,q_n})=\big(1-q+o_\mathbb{P}(1)\big)n$. In this case, $T_{n,q}$ consists of a ``rightward'' path of length $\big(1-q+o_\mathbb{P}(1)\big)n$, with left subtrees of height $O_\mathbb{P}\big(\log\frac{1}{1-q}\big)$ hanging from each of its nodes. (The notation $O_{\mathbb{P}}$ and $o_{\mathbb{P}}$ is defined in Section~\ref{sec:notation}, below.)

When $(q_n)_{n\geq0}$ is small enough that the first term in the denominator overwhelms the second, we are able to strengthen the above result, obtaining strong bounds on the rate of convergence.
\begin{thm}\label{thm:AS}
    Fix any $[0,1]$-valued sequence $(q_n)_{n\geq0}$ such that $n(1-q_n)/\log n \to \infty$. Then for any $\varepsilon>0$ and $\lambda>0$,
    \begin{align*}
        \mathbb{P}\left(\left|\frac{h(T_{n,q_n})}{n(1-q_n)}-1\right|>\varepsilon\right)=O\left(\frac{1}{n^\lambda}\right).
    \end{align*}
\end{thm}
When the first term in the denominator is dominant and also $nq_n\rightarrow\infty$, we prove a central limit theorem for the height.

\begin{thm}\label{thm:CLT}
    Fix any $[0,1]$-valued sequence $(q_n)_{n\geq0}$ such that $n(1-q_n)/\log n \to \infty$ and $nq_n\to \infty$. Then
    \begin{align*}
        \frac{h(T_{n,q_n})-n(1-q_n)-c^*\log\big((1-q_n)^{-1}\big)}{\sqrt{n(1-q_n)q_n}}\overset{\textrm{d}}{\longrightarrow}\textsc{Normal}(0,1)\, .
    \end{align*}
\end{thm}

Finally when $nq_n=O(1)$, we prove a Poisson limit theorem for the height (after re-centering but without re-scaling).

\begin{thm}\label{thm:Poisson}
    Let $(q_n)_{n\geq0}$ be any $[0,1]$-valued sequence such that $nq_n\rightarrow\lambda\in[0,\infty)$. Then
    \begin{align*}
        n-1-h(T_{n,q_n})\overset{d}{\longrightarrow}\textsc{Poisson}(\lambda)\,.
    \end{align*}
\end{thm}

The next subsection briefly discusses related literature on random trees and Mallows permutations. Section~\ref{sec:notation} then introduces some notation we need. The remainder of Section~\ref{sec:Intro} describes some of the key tools used in proving Theorem~\ref{thm:Prob}~-~\ref{thm:Poisson}, and, while doing so, provides an overview of our approach to their proofs. Theorem~\ref{thm:Prob} is proved in three parts, depending on whether $n(1-q_n)$ is much smaller than, much larger than, or of the same order as $\log n$. The arguments for these cases are sketched in Sections~\ref{sec:MTandRBST}~-~\ref{sec:intValues}, respectively. Since the proof of Theorem~\ref{thm:AS} essentially consists in extracting quantitative estimates from the proof of Theorem~\ref{thm:Prob}, we do not spend much space on it in the introduction. Finally, Section~\ref{sec:DistrLimits} describes our arguments for our distributional limit results, Theorem~\ref{thm:CLT} and \ref{thm:Poisson}.

\subsection{Related work}
 
The Mallows permutation model was first introduced by C.L.~Mallows~\cite{mallows1957non} 
in the context of ranking theory. The study of its probabilistic properties has taken off in the past decade; 
previously studied properties of Mallows permutations include the length of the longest increasing subsequence~\cite{basu2017limit,bhatnagar2015lengths,mueller2013length}, the cycle structure~\cite{crane2018probability,gladkich2018cycle,he2020central,mukherjee2016fixed,pinsky2019permutations}, relations to exchangeability~\cite{gnedin2012two,gnedin2010q} and to random matchings~\cite{angel2018mallows}, random dynamics with Mallows permutations as stationary distribution~\cite{benjamini2005mixing,diaconis2004analysis}, and thermodynamic properties of Mallows measures~\cite{starr2009thermodynamic,starr2018phase}.

Since our work is focused on random trees built from Mallows permutations, it is also natural to situate it in the context of the literature on random trees. This is a vast literature and we only discuss a smattering of it. As mentioned above, Devroye~\cite{devroye1986note} proved that the height $h_n$ of a random binary search tree of size $n$ is asymptotically $\big(c^*+o_\mathbb{P}(1)\big)\log n$; this built on previous work of Pittel~\cite{pittel1984growing}, who proved that $h_n/\log n \to \alpha \in (0,\infty)$ almost surely, but did not identify the constant $\alpha$. Random binary search trees lie within the more general {\em increasing tree model}, for which the first order behaviour of the height has been well-characterized~\cite{broutin2008height,drmota2009height}. Building on Devroye's results, Reed~\cite{reed2003height} and Drmota~\cite{drmota2003analytic} found two conceptually different proofs that the variance of $h_n$ is bounded in $n$. 

The study of random binary trees, random increasing trees, and their ilk, is intimately connected to the properties of {\em branching random walk}; results on the height of random trees are often extracted (at varying levels of difficulty) from results on the maximal displacement of a corresponding branching random walk. For example, the results of~\cite{broutin2008height,chauvin2005martingales,devroye1986note} rely on the Hammersley-Kingman-Biggins theorem~\cite{biggins1976first,hammersley1974postulates, kingman1975first}, which provides a law of large numbers for the maximum of branching random walks; and the arguments of Reed~\cite{reed2003height} proceeds by relating the height of binary search trees to the minimal position in a binary branching random walk with exponential step distribution. Further related results on minima in branching random walks can be found in \cite{addario2009minima,aidekon2013convergence}. The lecture notes~\cite{shi2015branching} provide an excellent introduction to the theory of branching random walks.

Finally, Mallows trees were introduced by S.N.~Evans, R.~Gr\"ubel, and A.~Wakolbinger~\cite{evans2012trickle}, who studied properties of the tree and generating processes. They showed, among other results, that Mallows trees are a specific case of \textit{trickle down process}; that is to say, they can be generated in a sequential manner, by adding one leaf at a time. 

\subsection{Notation}\label{sec:notation}
For functions $f:\R \to \R$, $g:\R \to \R$ or $f:\N \to \R$, $g:\N \to \R$, 
we write $f=O(g)$ to mean $f(n)=O(g(n))$ as $n \to \infty$ unless a different limit is specified. If $f=O(g)$ then we also write $g=\Omega(f)$. We also use the notation $f=o(g)$ and its synonym $g=\omega(f)$. If $f=o(g)$ then we will also write $f \ll g$ and $g \gg f$. 
We write $f\sim g$ to mean that $f(n)=(1+o(1))g(n)$ as $n \to \infty$. 

For sequences of random variables $(X_n)_{n \ge 0}$ and $(Y_n)_{n \ge 0}$, we write $X_n = O_{\mathbb{P}}(Y_n)$ if, for all $\varepsilon > 0$, there exists $K > 0$ such that 
\[
\limsup_{n \ge 0} \mathbb{P}\big(|X_n| \ge K Y_n\big) < \varepsilon\,.
\]
We write $X_n=o_{\mathbb{P}}(Y_n)$ if, for all $\varepsilon > 0$, 
\[
\lim_{n \to \infty} \mathbb{P}\big(|X_n|> \varepsilon Y_n\big) = 0\,.
\]

For random variables $X$ and $Y$, we write $X \stackrel{\mathrm{d}}{=} Y$ if $X$ and $Y$ have the same distribution. We write $X\preceq Y$, or equivalently $Y\succeq X$, if for all $x\in\mathbb{R}$, we have
\begin{align*}
    \mathbb{P}\big(X\geq x\big)\leq\mathbb{P}\big(Y\geq x\big)\,; 
\end{align*}
in this case we say $X$ is stochastically smaller than $Y$.

\subsection{Mallows trees and random binary search trees}\label{sec:MTandRBST}

Many properties of random binary search trees can be extended to Mallows trees. Perhaps the most fundamental of these are the {\em branching property}, which means that disjoint subtrees of a Mallows tree are conditionally independent given their sizes, and the {\em projective consistency}, which is the fact that subtrees of Mallows trees are again Mallows trees. The following proposition, due to Evans, Gr\"ubel and Wakolbinger,~\cite{evans2012trickle}, formalizes these properties, and additionally describes the joint distribution of the sizes of the left and right subtrees of the root in a Mallows tree.

\begin{prop}[{\cite[Section 7]{evans2012trickle}}]\label{prop:ProjCons}
    For all $q\in[0,1]$ and $n\geq1$, for any $0\leq k\leq n-1$, we have
    \begin{align*}
        \mathbb{P}\Big(\big|T_{n,q}(\overline{0})\big|=k\Big)&=\mathbb{P}\Big(\big|T_{n,q}(\overline{1})\big|=n-1-k\Big)=\left\{\begin{array}{ll}
            \frac{(1-q)q^k}{1-q^n} & \textrm{if $q\in[0,1)$} \\
            \frac{1}{n} & \textrm{if $q=1$}
        \end{array}\right.\,.
    \end{align*}
    Moreover, $T_{n,q}(\overline{0})$ and $T_{n,q}(\overline{1})$ are conditionally independent Mallows trees given their sizes. That is, for every $0\leq k\leq n-1$, for any trees $t_0$ and $t_1$ rooted at $\varnothing$, of respective sizes $k$ and $n-1-k$, we have
    \begin{align*}
        \mathbb{P}\Big(T_{n,q}(\overline{0})=\overline{0}t_0,T_{n,q}(\overline{1})=\overline{1}t_1\,\Big|\,\big|T_{n,q}(\overline{0})\big|=k\Big)&=\mathbb{P}\big(T_{k,q}=t_0\big)\mathbb{P}\big(T_{n-1-k,q}=t_1\big).
    \end{align*}
    Conversely, these properties characterize Mallows trees.
\end{prop}

From this proposition, one can see that the split between left and right subtree is not symmetric; the right subtree at any node is stochastically larger than its left subtree. This observation straightforwardly leads to the following proposition, stating that the rightmost path in $T_{n,q}$ is the stochastically longest path.

\begin{prop}\label{prop:StochBounds} 
    For all $q\in[0,1]$ and $n\geq1$, for all $v\in T_\infty$,
    \begin{align*}
        \mathbb{P}\big(v\in T_{n,q}\big)\leq\mathbb{P}\big(\overline{1}^{|v|}\in T_{n,q}\big)
    \end{align*}
\end{prop}

We prove Proposition~\ref{prop:StochBounds} in Section~\ref{sec:UpperBoundSmall}. With the result of this last proposition, we can bound the height of $T_{n,q_n}$ from above by using a union bound over all nodes at a given depth, as
\begin{align*}
    \mathbb{P}\big(h(T_{n,q_n})\geq h\big)&\leq\sum_{v\in T_\infty:|v|=h}\mathbb{P}\big(v\in T_{n,q_n}\big)\\
    &\leq2^h\mathbb{P}\big(\overline{1}^h\in T_{n,q_n}\big).
\end{align*}
When $n(1-q_n)/\log n\rightarrow0$, this bound is tight enough to prove the upper bound of Theorem~\ref{thm:Prob}. It will also be useful in proving that $\big(\big(\frac{h(T_{n,q_n})}{n(1-q_n)+c^*\log n}\big)^p\big)_{n\geq1}$ is uniformly integrable for any sequence $(q_n)_{n\geq0}$, thereby extending the convergence in probability to the $L_p$ convergence.

For $(q_n)_{n\geq0}$ such that $n(1-q_n)/\log n\rightarrow0$, we also use a comparison  argument, albeit a slightly more complicated one, to prove the lower bound. At the heart of the argument is the following computation. Let $U$ be $\textsc{Uniform}([0,1])$. Then, for any $q\in[0,1)$, $n\geq1$ and $0\leq k\leq n-1$,
\begin{align*}
    \mathbb{P}\left(\left\lfloor\frac{\log\left(1-U(1-q^n)\right)}{\log q}\right\rfloor=k\right)&=\mathbb{P}\left(k\leq\frac{\log\left(1-U(1-q^n)\right)}{\log q}<k+1\right)\\
    &=\mathbb{P}\left(\frac{1-q^k}{1-q^n}\leq U<\frac{1-q^{k+1}}{1-q^n}\right)\\
    &=\frac{(1-q)q^k}{1-q^n}
\end{align*}
This identity and Proposition~\ref{prop:ProjCons} together imply that we can generate a $\mt(n,q)$-distributed tree as follows. Let $(U_v)_{v\in T_\infty}$ be independent $\textsc{Uniform}([0,1])$ random variables indexed by the nodes of $T_\infty$. Set $S_{n,q}(\root)=n$. Then, for $v\in T_\infty$, inductively define
\begin{align*}
    S_{n,q}(v\overline{0})&=\left\{\begin{array}{ll}
        \left\lfloor\frac{\log\left(1-U_v(1-q^{S_{n,q}(v)})\right)}{\log q}\right\rfloor & \textrm{if $q\in(0,1)$} \\
        \lfloor S_{n,q}(v)U_v\rfloor & \textrm{if $q=1$} \\
        0 & \textrm{if $q=0$}
    \end{array}\right.
\end{align*}
and
\begin{align*}
    S_{n,q}(v\overline{1})=S_{n,q}(v)-1-S_{n,q}(v\overline{0})\,.
\end{align*}
Then the tree $T_{n,q}=\big\{v\in T_\infty:S_{n,q}(v)\geq1\big\}$ is $\mt(n,q)$-distributed, and $S_{n,q}$ corresponds to the size of the subtree at any given node: $S_{n,q}(u)=\big|T_{n,q}(u)\big|$.

This construction couples the trees $(T_{n,q})$ as both $q$ and $n$ vary. Using this coupling, we will be able to prove the following proposition.

\begin{prop}\label{prop:couplingMallowsRBST}
    For all $q\in[0,1)$ and $n\geq1$, for any $0\leq\ell\leq n$, with $m=\left\lfloor\frac{1-q^n}{1-q^{\ell+2}}\right\rfloor$, then
    \begin{align*}
        \mathbb{P}\big(h(T_{n,q})\leq\ell\big)\leq\mathbb{P}\big(h(T_{m,1})\leq\ell\big)\,.
    \end{align*}
\end{prop}
The proof of Proposition~\ref{prop:couplingMallowsRBST} can be found in Section~\ref{sec:lowerBoundSmall}. When $(q_n)_{n\geq0}$ is such that $n(1-q_n)/\log n\rightarrow0$, this stochastic bound combined with results of Devroye~\cite{devroye1986note} will yield the desired lower bound.

\subsection{Right depth and height of Mallows trees}\label{sec:IBMandRD}

The results of the previous section, relating Mallows trees to random binary search trees, give tight bounds on the height only when $n(1-q_n)/\log n\rightarrow0$; in this case the two tree models show strong resemblance. When $(q_n)_{n\geq0}$ does not satisfy this condition, the rightmost path of a Mallows tree is playing a more important role in its height. We now study the properties of this path and its connection to the rest of the tree.

Note that if $f:[n] \to \Z_+$ is an injective function and $f'=f|_{[n-1]}$, then $\bst{f'}$ is the subtree of $\bst{f}$ consisting of the nodes with labels $f(1),\ldots,f(n-1)$. More precisely, $\bst{f'} = \{v \in T_f: \tau\langle f\rangle(v) \ne f(n)\}$, and $\tau\langle f'\rangle$ is the restriction of $\tau\langle f\rangle$ to $\bst{f'}$. For example, in Figure~\ref{fig:BST}, with $f'=(f(1),\ldots,f(7))=(4,1,9,7,2,6,12)$, the tree $\bst{f'}$ is obtained from the depicted tree $T(f)$ by removing the node with label $8$. We next use this fact to describe an explicit construction of a nested sequence of Mallows trees, which will be useful for our analysis. 

%

Given an infinite $\{0,1\}$-valued matrix $b=(b_{i,j})_{i,j\geq1}$ with infinitely many ones in every row, define an injective function $f^{b}=(f^{b}(i),i \ge 1)$ as follows. Let $f^{b}(1)=\inf\{j\in\mathbb{N}:b_{1,j}=1\}$. Having defined $(f^{b}(k),1 \le k < i)$, let $F^{b}_{i-1}=\{f^{b}(k),1 \le k < i\}$ and set $f^{b}(i)=\inf\{j\in\mathbb{N}\setminus F^{b}_{i-1}:b_{i,j}=1\}$.
We write $T^b_n$ as shorthand for $\bst{f^b_n}$, where $f^b_n=(f^b(i),1 \le i \le n)$. An example is shown in Figure~\ref{fig:bset}. For the matrix $b$ shown in that figure, we obtain $f^b_8=(4,1,9,7,2,6,12,8)$, so the tree $T^b_8$ is precisely the binary search tree shown in Figure~\ref{fig:BST}. 

\begin{figure}[ht]
    \centering
    \begin{tikzpicture}[scale=1.2]
        \draw[very thin,gray!50!black] (-0.35,0.85)--(-0.35,-4.5);
        \draw[very thin,gray!50!black] (-0.85,0.85)--(-0.85,-4.5);
        \draw[very thin,gray!50!black] (-0.85,0.35)--(7,0.35);
        \draw[very thin,gray!50!black] (-0.85,0.85)--(7,0.85);
        \node[scale=0.9,gray!50!black] at (-0.6,0.6){$b_{i,j}$};
        \node[scale=0.9,gray!50!black] at (-0.6,0){$1$};
        \node[scale=0.9,gray!50!black] at (-0.6,-0.5){$2$};
        \node[scale=0.9,gray!50!black] at (-0.6,-1){$3$};
        \node[scale=0.9,gray!50!black] at (-0.6,-1.5){$4$};
        \node[scale=0.9,gray!50!black] at (-0.6,-2){$5$};
        \node[scale=0.9,gray!50!black] at (-0.6,-2.5){$6$};
        \node[scale=0.9,gray!50!black] at (-0.6,-3){$7$};
        \node[scale=0.9,gray!50!black] at (-0.6,-3.5){$8$};
        \node[scale=0.9,gray!50!black] at (-0.6,-4){$\vdots$};
        \node[scale=0.9,gray!50!black] at (0,0.6){$1$};
        \node[scale=0.9,gray!50!black] at (0.5,0.6){$2$};
        \node[scale=0.9,gray!50!black] at (1,0.6){$3$};
        \node[scale=0.9,gray!50!black] at (1.5,0.6){$4$};
        \node[scale=0.9,gray!50!black] at (2,0.6){$5$};
        \node[scale=0.9,gray!50!black] at (2.5,0.6){$6$};
        \node[scale=0.9,gray!50!black] at (3,0.6){$7$};
        \node[scale=0.9,gray!50!black] at (3.5,0.6){$8$};
        \node[scale=0.9,gray!50!black] at (4,0.6){$9$};
        \node[scale=0.9,gray!50!black] at (4.5,0.6){$10$};
        \node[scale=0.9,gray!50!black] at (5,0.6){$11$};
        \node[scale=0.9,gray!50!black] at (5.5,0.6){$12$};
        \node[scale=0.9,gray!50!black] at (6,0.6){$13$};
        \node[scale=0.9,gray!50!black] at (6.5,0.6){$\cdots$};
        \node at (0.0,-0.0){$0$};
        \node at (0.5,-0.0){$0$};
        \node at (1.0,-0.0){$0$};
        \node[draw] at (1.5,-0.0){$1$};
        \node at (2.0,-0.0){$1$};
        \node at (2.5,-0.0){$0$};
        \node at (3.0,-0.0){$0$};
        \node at (3.5,-0.0){$1$};
        \node at (4.0,-0.0){$1$};
        \node at (4.5,-0.0){$0$};
        \node at (5.0,-0.0){$1$};
        \node at (5.5,-0.0){$0$};
        \node at (6.0,-0.0){$0$};
        \node at (6.5,-0.0){$\cdots$};
        \node[draw] at (0.0,-0.5){$1$};
        \node at (0.5,-0.5){$0$};
        \node at (1.0,-0.5){$1$};
        \node[red!70!black] at (1.5,-0.5){$0$};
        \node at (2.0,-0.5){$0$};
        \node at (2.5,-0.5){$1$};
        \node at (3.0,-0.5){$1$};
        \node at (3.5,-0.5){$0$};
        \node at (4.0,-0.5){$0$};
        \node at (4.5,-0.5){$1$};
        \node at (5.0,-0.5){$0$};
        \node at (5.5,-0.5){$0$};
        \node at (6.0,-0.5){$1$};
        \node[red!70!black] at (0.0,-1.0){$0$};
        \node at (0.5,-1.0){$0$};
        \node at (1.0,-1.0){$0$};
        \node[red!70!black] at (1.5,-1.0){$1$};
        \node at (2.0,-1.0){$0$};
        \node at (2.5,-1.0){$0$};
        \node at (3.0,-1.0){$0$};
        \node at (3.5,-1.0){$0$};
        \node[draw] at (4.0,-1.0){$1$};
        \node at (4.5,-1.0){$1$};
        \node at (5.0,-1.0){$1$};
        \node at (5.5,-1.0){$0$};
        \node at (6.0,-1.0){$0$};
        \node at (6.5,-1.0){$\cdots$};
        \node[red!70!black] at (0.0,-1.5){$0$};
        \node at (0.5,-1.5){$0$};
        \node at (1.0,-1.5){$0$};
        \node[red!70!black] at (1.5,-1.5){$0$};
        \node at (2.0,-1.5){$0$};
        \node at (2.5,-1.5){$0$};
        \node[draw] at (3.0,-1.5){$1$};
        \node at (3.5,-1.5){$0$};
        \node[red!70!black] at (4.0,-1.5){$1$};
        \node at (4.5,-1.5){$0$};
        \node at (5.0,-1.5){$0$};
        \node at (5.5,-1.5){$1$};
        \node at (6.0,-1.5){$1$};
        \node at (6.5,-1.5){$\cdots$};
        \node[red!70!black] at (0.0,-2.0){$0$};
        \node[draw] at (0.5,-2.0){$1$};
        \node at (1.0,-2.0){$1$};
        \node[red!70!black] at (1.5,-2.0){$0$};
        \node at (2.0,-2.0){$1$};
        \node at (2.5,-2.0){$0$};
        \node[red!70!black] at (3.0,-2.0){$1$};
        \node at (3.5,-2.0){$1$};
        \node[red!70!black] at (4.0,-2.0){$0$};
        \node at (4.5,-2.0){$0$};
        \node at (5.0,-2.0){$0$};
        \node at (5.5,-2.0){$0$};
        \node at (6.0,-2.0){$1$};
        \node at (6.5,-2.0){$\cdots$};
        \node[red!70!black] at (0.0,-2.5){$1$};
        \node[red!70!black] at (0.5,-2.5){$1$};
        \node at (1.0,-2.5){$0$};
        \node[red!70!black] at (1.5,-2.5){$0$};
        \node at (2.0,-2.5){$0$};
        \node[draw] at (2.5,-2.5){$1$};
        \node[red!70!black] at (3.0,-2.5){$0$};
        \node at (3.5,-2.5){$0$};
        \node[red!70!black] at (4.0,-2.5){$1$};
        \node at (4.5,-2.5){$0$};
        \node at (5.0,-2.5){$0$};
        \node at (5.5,-2.5){$1$};
        \node at (6.0,-2.5){$0$};
        \node at (6.5,-2.5){$\cdots$};
        \node[red!70!black] at (0.0,-3.0){$0$};
        \node[red!70!black] at (0.5,-3.0){$0$};
        \node at (1.0,-3.0){$0$};
        \node[red!70!black] at (1.5,-3.0){$1$};
        \node at (2.0,-3.0){$0$};
        \node[red!70!black] at (2.5,-3.0){$1$};
        \node[red!70!black] at (3.0,-3.0){$1$};
        \node at (3.5,-3.0){$0$};
        \node[red!70!black] at (4.0,-3.0){$0$};
        \node at (4.5,-3.0){$0$};
        \node at (5.0,-3.0){$0$};
        \node[draw] at (5.5,-3.0){$1$};
        \node at (6.0,-3.0){$0$};
        \node at (6.5,-3.0){$\cdots$};
        \node[red!70!black] at (0.0,-3.5){$1$};
        \node[red!70!black] at (0.5,-3.5){$0$};
        \node at (1.0,-3.5){$0$};
        \node[red!70!black] at (1.5,-3.5){$1$};
        \node at (2.0,-3.5){$0$};
        \node[red!70!black] at (2.5,-3.5){$0$};
        \node[red!70!black] at (3.0,-3.5){$1$};
        \node[draw] at (3.5,-3.5){$1$};
        \node[red!70!black] at (4.0,-3.5){$0$};
        \node at (4.5,-3.5){$1$};
        \node at (5.0,-3.5){$0$};
        \node[red!70!black] at (5.5,-3.5){$0$};
        \node at (6.0,-3.5){$0$};
        \node at (6.5,-3.5){$\cdots$};
        \node[red!70!black] at (0.0,-4.0){$\vdots$};
        \node[red!70!black] at (0.5,-4.0){$\vdots$};
        \node at (1.0,-4.0){$\vdots$};
        \node[red!70!black] at (1.5,-4.0){$\vdots$};
        \node at (2.0,-4.0){$\vdots$};
        \node[red!70!black] at (2.5,-4.0){$\vdots$};
        \node[red!70!black] at (3.0,-4.0){$\vdots$};
        \node[red!70!black] at (3.5,-4.0){$\vdots$};
        \node[red!70!black] at (4.0,-4.0){$\vdots$};
        \node at (4.5,-4.0){$\vdots$};
        \node at (5.0,-4.0){$\vdots$};
        \node[red!70!black] at (5.5,-4.0){$\vdots$};
        \node at (6.0,-4.0){$\vdots$};
        \node at (6.5,-4.0){$\ddots$};
        \node[scale=0.7,blue!70!black] at (1.75,-0.15){$4$};
        \node[scale=0.7,blue!70!black] at (0.25,-0.65){$1$};
        \node[scale=0.7,blue!70!black] at (4.25,-1.15){$9$};
        \node[scale=0.7,blue!70!black] at (3.25,-1.65){$7$};
        \node[scale=0.7,blue!70!black] at (0.75,-2.15){$2$};
        \node[scale=0.7,blue!70!black] at (2.75,-2.65){$6$};
        \node[scale=0.7,blue!70!black] at (5.8,-3.15){$12$};
        \node[scale=0.7,blue!70!black] at (3.75,-3.65){$8$};
    \end{tikzpicture}
    \caption{An example of the top-left corner of an infinite $\{0,1\}$-valued matrix $b=(b_{i,j})_{i,j\geq1}$. The boxed $1$'s are in position $(i,f^b(i))$ and their column numbers are indicated as blue subscripts. For a given column $i$, the numbers in red correspond to positions $(i,j)$ where $j\in F^b_{i-1}$ and can be found below the boxed $1$'s; the boxed $1$ in row $i$ is always the first non-red $1$ in row $i$.}
    \label{fig:bset}
\end{figure}

The \textit{infinite $b$-model} for trees is the sequence of (labelled) trees $\big(T^b_n\big)_{n\geq0}$, which by construction is increasing, in that $T^b_n$ is a subtree of $T^b_{n+1}$ for all $n\geq0$. The corresponding sequence of labelling functions is defined by $(\tau^b_n)_{n\geq0}:=\big(\tau\langle f^b_{n}\rangle\big)_{n\geq0}$. We sometimes omit the matrix $b$ when it is clear from context. 
The utility of this construction is explained by the following proposition.
\begin{prop}\label{prop:IBM}
    Fix $q\in[0,1)$ and let $B=\big(B_{i,j}\big)_{i,j\geq1}$ have independent $\textsc{Bernoulli}(1-q)$ entries, and for $n\geq1$, let $\sigma^B_n\in\mathcal{S}_n$ be the permutation of $[n]$ defined by $\sigma^B_n(i)=\textrm{rank}\big\{f^B_n(i), F^B_n\big\}$. Then $\sigma^B_n$ is $\pi_{n,q}$-distributed for all $n\geq0$.
\end{prop}
We prove Proposition~\ref{prop:IBM} in Section~\ref{sec:IBMGeneralResults}. It follows that for $B$ as in the proposition, $T^B_n$ is $\mt(n,q)$-distributed for all $n$. The random trees in the sequence $(T^B_n)_{n \ge 0}$ may be viewed as the successive states of a transient Markov chain taking values in the set of finite subtrees of $T_{\infty}$. This chain was already defined in \cite{evans2012trickle}, where aspects of its asymptotic behaviour were studied; however, the observation that its one-dimensional marginals are all Mallows-distributed appears to be new.

The next corollary is a direct consequence of Proposition~\ref{prop:IBM} and the fact that $(T^b_n)_{n\geq0}$ is increasing for all $b$.

\begin{cor}\label{cor:TnqIncreasing}
    Let $n\geq0$ and $q\in[0,1]$. Then, for all $u\in T_\infty$
    \begin{align*}
        \mathbb{P}\big(u\in T_{n,q}\big)\leq\mathbb{P}\big(u\in T_{n+1,q}\big)\,.
    \end{align*}
\end{cor}

We write $T^B = \bigcup_{n \ge 0} T^B_n$ for the infinite tree which is the increasing limit of the sequence $(T^B_n)_{n\geq0}$, and $\tau^B$ for the corresponding labelling. It is immediate from the construction that $T^B = \bst{f^B}$.

For $B$ as in the proposition, the random function $f^B:\Z_+ \to \Z_+$ defined previously is a.s.\ a bijection, and its law is the so-called $\textsc{Mallows}(q)$ distribution on $S_{\infty}=\{\sigma:\Z_+ \to \Z_+:\sigma\mbox{ a permutation}\}$, introduced in \cite{gnedin2010q}; the fact that $f^B$ is $\textsc{Mallows}(q)$-distributed was proved in \cite{gladkich2018cycle}.

Let $M^B_0=0$ and for $n \ge 1$, let $M^B_n = \max\big(f^B(1),\ldots,f^B(n)\big)$. Then set $R^B_0=0$ and for $n \ge 1$, let $R^B_n = \#\{i \in [n]: M^B_i > M^B_{i-1}\}$ be the number of records in the sequence $\big(f^B(1),\ldots,f^B(n)\big)$. Note that $R^B_n$ is precisely the {\em right depth} of $T^B_n$, i.e., $R^B_n = \max\{d: \overline{1}^d \in T^B_n\}$. Also, for all $k \ge 1$, if $R^B_{n-1}=k-1$ and $R^B_n=k$ then $\tau^B(\overline{1}^k)=M^B_n$.

For any $k\geq0$ and any node $u\in T^B(\overline{1}^k\overline{0})$, we have $\tau^B(\overline{1}^{k-1})<\tau^B(u)<\tau^B(\overline{1}^k)$; here we write $\overline{1}^0=\root$ and for $k=0$, set $\tau^B(\overline{1}^{k-1})=0$. Since $f^B$ is a bijection, it follows that the subtree $T^B(\overline{1}^k\overline{0})$ contains exactly $\tau^B(\overline{1}^k)-\tau^B(\overline{1}^{k-1})-1$ nodes, and the labels assigned to these nodes by $\tau^B$ are precisely the elements of the set $\big\{\tau^B(\overline{1}^{k-1})+1,\ldots,\tau^B(\overline{1}^{k})-1\big\}$. Moreover, since any infinite sequence of positive integers contains infinitely many records, necessarily $T^B$ contains the infinite rightward path $P_R:=\{\overline{1}^k,k\geq0\}$, and the left subtrees $\big(T^B(\overline{1}^k\overline{0}),k \ge 0\big)$ hanging from $P_R$ have respective sizes $\big(\tau^B(\overline{1}^k)-\tau^B(\overline{1}^{k-1})-1,k \ge 1\big)$.

Much of our analysis will be based on the decompositions of $T^B$ and $T^B_n$ as
\begin{align*}
    T^B&=P_R\cup\bigcup_{k\geq0}T^B\big(\overline{1}^k\overline{0}\big)\, ,\\
    T^B_n&=\{\overline{1}^k,0 \le k \le R^B_n\}\cup\bigcup_{0 \le k \le R^B_n}T^B_n\big(\overline{1}^k\overline{0}\big)\, .
\end{align*}
From the second decomposition, it is immediate that
\begin{align}\label{eq:first_height_bd}
    h(T_{n,q})\stackrel{\mathrm{d}}{=} h(T^B_n)
     =\max_{0\leq k\leq R^B_n}\Big\{h\Big(T^B_n\big(\overline{1}^k\overline{0}\big)\Big)+k+1\Big\}
    \le 
    \max_{0\leq k\leq R^B_n}\Big\{h\Big(T^B\big(\overline{1}^k\overline{0}\big)\Big)+k+1\Big\}
\end{align}

In order to use (\ref{eq:first_height_bd}) to get useful information about the height, we need to understand the distributions of $R^B_n$ and of the subtrees $T^B_n(\overline{1}^k\overline{0})$ and $T^B(\overline{1}^k\overline{0})$. The last of these is the easiest to describe. We say a random variable $G$ is $\textsc{Geometric}(c)$-distributed if $\p(G=k)=(1-c)^kc$ for $k \in \N=\{0,1,2,...\}$.

\begin{lemma}\label{lem:LeftSubtreesDistribution} 
    Fix $q\in[0,1)$ and let $B=\big(B_{i,j}\big)_{i,j\geq1}$ have independent $\textsc{Bernoulli}(1-q)$ entries. Then the random trees $\big(T^B(\overline{1}^k\overline{0})\big)_{k\geq0}$ are independent and identically distributed with 
    \begin{align*} 
        T^B\big(\overline{1}^k\overline{0}\big)\overset{\textrm{d}}{=}\overline{1}^k\overline{0}T_{G(q),q},
    \end{align*}
    where $G(q)$ is $\textsc{Geometric}(1-q)$-distributed and is independent of the trees $\big(T_{n,q}\big)_{n\geq0}$. In other words, for all $k,n\geq0$ and any tree $t\subseteq T_\infty$ with $|t|=n$, we have
    \begin{align*}
        \mathbb{P}\Big(T^B\big(\overline{1}^k\overline{0}\big)=\overline{1}^k\overline{0}t\Big)&=q^{n}(1-q)\cdot\mathbb{P}\big(T_{n,q}=t\big).
    \end{align*}
\end{lemma}

We prove Lemma~\ref{lem:LeftSubtreesDistribution} in Section~\ref{sec:IBMGeneralResults}. Combined with (\ref{eq:first_height_bd}), this lemma yields a key distributional upper bound on $h(T_{n,q})$. We now have
\begin{align*}
    h(T_{n,q}) \stackrel{\mathrm{d}}{=}
\max_{0\leq k\leq R^B_n}\Big\{h\Big(T^B_n\big(\overline{1}^k\overline{0}\big)\Big)+k+1\Big\}
&     \le 
    \max_{0\leq k\leq R^B_n}\Big\{h\Big(T^B\big(\overline{1}^k\overline{0}\big)\Big)+k+1\Big\} 
\\
& = R^B_n+1 + 
\max_{0\leq k\leq R^B_n}\Big\{h\Big(T^B\big(\overline{1}^{R^B_n-k}\overline{0}\big)\Big)-k\Big\}\\
& \stackrel{\mathrm{d}}{=}
R^B_n+1 + 
\max_{0\leq k\leq R^B_n}\Big\{h\Big(T^B\big(\overline{1}^{k}\overline{0}\big)\Big)-k\Big\}\\
& \le R_n^B+1+\sup_{k \ge 0} \Big\{h\Big(T^B\big(\overline{1}^{k}\overline{0}\big)\Big)-k\Big\}\, .
\end{align*}
Considering that $R^B_n$ corresponds to the depth of the rightmost path in $T^B_n$, we also have $R^B_n\leq h(T^B_n)\overset{d}{=}h(T_{n,q})$. In combination with (\ref{eq:first_height_bd}), this yields that
\begin{align}\label{eq:htnq_upper}
    R^B_n\preceq h(T_{n,q})\preceq R_n^B+1+\sup_{k \ge 0} \Big\{h\Big(T^B\big(\overline{1}^{k}\overline{0}\big)\Big)-k\Big\}\,,
\end{align}
where we recall that $\preceq$ denotes stochastic inequality. Lemma~\ref{lem:LeftSubtreesDistribution} tells us the trees whose heights appear in the final supremum are independent and $T_{G(q),q}$-distributed. Combining this inequality with a union-bound and Proposition~\ref{prop:StochBounds}, we now have
\begin{align*}
    \mathbb{P}\big(h(T_{n,q})\geq h\big)&\leq\mathbb{P}\left(R_n^B+1+\sup_{k \ge 0} \Big\{h\Big(T^B\big(\overline{1}^{k}\overline{0}\big)\Big)-k\Big\}\geq h\right)\\
    &\leq\inf_{0\leq\ell\leq h}\bigg\{\mathbb{P}\Big(R^B_n+1\geq h-\ell\Big)+\sum_{k\geq0}\mathbb{P}\Big(h\big(T^B(\overline{1}^{k}\overline{0})\big)\geq\ell+k\Big)\bigg\}\\
    &=\inf_{0\leq\ell\leq h}\bigg\{\mathbb{P}\Big(R^B_n+1\geq h-\ell\Big)+\sum_{k\geq0}\mathbb{P}\Big(h\big(T_{G(q),q}\big)\geq\ell+k\Big)\bigg\}\\
    &\leq\inf_{0\leq\ell\leq h}\bigg\{\mathbb{P}\Big(R^B_n+1\geq h-\ell\Big)+\sum_{k\geq0}2^{k+\ell}\mathbb{P}\Big(R^B_{G(q)}\geq\ell+k\Big)\bigg\}
\end{align*}
where in the final line $G(q)$ should be understood to be independent of the random variables in $B$. The preceding argument gives us a way to derive upper tail bounds on $\big(h(T_{n,q})\big)$ exclusively by controlling the upper tails of the random variables $(R^B_n,n\geq1)$. The next proposition is our key tool for doing so. 

\begin{prop}\label{prop:BivariateGenFun}
    Fix $q\in[0,1)$ and let $B=\big(B_{i,j}\big)_{i,j\geq1}$ have independent $\textsc{Bernoulli}(1-q)$ entries.
Then the sequence $(R^B_n,M^B_n)_{n\geq0}$ is a Markov chain with transition probabilities given by
    \begin{align*}
        \mathbb{P}\Big(R^B_{n+1}=r+k,M^B_{n+1}=m+\ell\,\Big|\,R^B_{n}=r,M^B_{n}=m\Big)=\left\{\begin{array}{ll}
            q^{m+\ell-n-1}(1-q) & \textrm{if $k=1$ and $\ell\geq1$} \\
            1-q^{m-n} & \textrm{if $k=0$ and $\ell=0$} \\
            0 & \textrm{otherwise}
        \end{array}\right..
    \end{align*}
    Moreover, for $x,y\in\mathbb{C}$ such that $q|y|<1$, we have
    \begin{align*}
        \mathbb{E}\left[x^{R^B_n+1}y^{M^B_{n}}\right]=y^n\prod_{1\leq k\leq n}\frac{q+(1-q)x-q^k}{1-q^ky}.
    \end{align*}
\end{prop}
We prove Proposition~\ref{prop:BivariateGenFun} in Section~\ref{sec:IBMGeneralResults}. We obtain the moment generating function of $R^B_n$ from Proposition~\ref{prop:BivariateGenFun} by taking $x=e^t$ and $y=1$. 
Using the moment generating function to control the behaviour of $R^B_n-\mathbb{E}[R^B_n]$ yields Chernoff-type bounds for both the upper and lower tail. The bounds are strong enough that they allow us to prove both the upper and lower bounds of Theorem~\ref{thm:AS}, implying the bounds of Theorem~\ref{thm:Prob} when $(q_n)_{n\geq0}$ is such that $n(1-q_n)/\log n\rightarrow\infty$. For the upper bound, the key consequence of Proposition~\ref{prop:BivariateGenFun} is the following proposition, which allows us to control the right hand side of (\ref{eq:htnq_upper}). It will also be used in the analysis for other ranges of $(q_n)_{n\geq 0}$.

\begin{prop}\label{prop:boundsLeftSubtree}
    There exist universal constants $M$, $C$ and $\lambda$ such that, for all $q\in[0,1)$ and $\xi\in\mathbb{R}$, we have
    \begin{align*}
        \mathbb{P}\left(\sup_{k\geq0}\Big\{h\Big(T^B\big(\overline{1}^{k}\overline{0}\big)\Big)-k\Big\}\geq c^*\log\left(\frac{1}{1-q}\right)+M\sqrt{\log\left(\frac{1}{1-q}\right)}+\xi\right)\leq Ce^{-\lambda\xi}.
    \end{align*}
\end{prop}

As above, $c^*$ is the unique solution of $c\log\left(\frac{2e}{c}\right)=1$ with $c\geq2$. The proof of this proposition can be found in Section~\ref{sec:bdsLeftSubtrees}.

\subsection{Intermediate values}\label{sec:intValues}

The most technical part of the proof of Theorem~\ref{thm:Prob} appears when $n(1-q_n)/\log n=\Theta(1)$. In this situation, the proof uses a combination of the techniques from the two previous cases, when $n(1-q_n)/\log n\rightarrow0$ and when $n(1-q_n)/\log n\rightarrow\infty$.

In order to obtain the identities and bounds in (\ref{eq:htnq_upper}), we decomposed $T^B_n$ into its rightmost path $\big(\overline{1}^k, 0\leq k\leq R^B_n\big)$, together with the left subtrees hanging from each of its nodes. Because $T^B_n\big(\overline{1}^{R^B_n+1}\big)=\emptyset$, this decomposition can be rewritten as
\begin{align*}
    T^B_n&=\{\overline{1}^k,0 \le k<R^B_n+1\}\cup\left(\bigcup_{0 \le k<R^B_n+1}T^B_n\big(\overline{1}^k\overline{0}\big)\right)\cup T^B_n\big(\overline{1}^{R^B_n+1}\big)\,.
\end{align*}
For any $0\leq d\leq R^B_n+1$, we may similarly decompose $T^B_n$ along the initial segment $\big(\overline{1}^k, 0\leq k<d\big)$ of the rightmost path to obtain
\begin{align*}
    T^B_n&=\{\overline{1}^k,0 \le k < d\}\cup\left(\bigcup_{0 \le k < d}T^B_n\big(\overline{1}^k\overline{0}\big)\right)\cup T^B_n\big(\overline{1}^d\big)\,.
\end{align*}
From this decomposition, it is immediate that, for all $d\geq0$,
\begin{align*}
    h(T_{n,q})\stackrel{\mathrm{d}}{=} h(T^B_n)
     &\leq\max\left\{\max_{0\leq k<d}\Big\{h\Big(T^B_n\big(\overline{1}^k\overline{0}\big)\Big)+k+1\Big\},d+h\Big(T^B_n\big(\overline{1}^d\big)\Big)\right\}\,,
\end{align*}
with equality whenever $d\leq R^B_n+1$. This inequality implies that, for all $d\geq0$, we have
\begin{align}\label{eq:second_height_bd}
    h(T_{n,q})&\preceq \max\left\{\max_{0\leq k<d}\Big\{h\Big(T^B\big(\overline{1}^k\overline{0}\big)\Big)+k+1\Big\},d+h\Big(T^B_n\big(\overline{1}^d\big)\Big)\right\}\notag\\
    &\stackrel{\mathrm{d}}{=}d+\max\left\{\max_{0\leq k<d}\Big\{h\Big(T^B\big(\overline{1}^k\overline{0}\big)\Big)-k\Big\},h\Big(T^B_n\big(\overline{1}^d\big)\Big)\right\}\notag\\
    &\le d+\max\left\{\sup_{k\geq0}\Big\{h\Big(T^B\big(\overline{1}^k\overline{0}\big)\Big)-k\Big\},h\Big(T^B_n\big(\overline{1}^d\big)\Big)\right\}\,;
\end{align}
The first term in the maximum is the same supremum as appears in (\ref{eq:htnq_upper}), and Proposition~\ref{prop:boundsLeftSubtree} gives an essentially sharp upper tail bound for this term of the form $c^*\log n+O_\mathbb{P}\left(\sqrt{\log n}\right)$. Because we aim to prove that $\frac{h(T_{n,q_n})}{n(1-q_n)+c^*\log n}\rightarrow1$ in probability, a natural choice for $d$ is then $n(1-q_n)$. This choice indeed gives the desired bound for $h\big(T^B_n(\overline{1}^d)\big)$, due to the following proposition.

\begin{prop}\label{prop:convRightSubtree}
    Let $(q_n)_{n\geq0}$ be such that $n(1-q_n)/\log n=\Theta(1)$. Then,
    \begin{align*}
        \frac{h\Big(T_{n,q_n}\big(\overline{1}^{\lfloor n(1-q_n)\rfloor}\big)\Big)}{c^*\log n}\rightarrow1
    \end{align*}
    in probability.
\end{prop}
The proof of this proposition can be found in Section~\ref{sec:ThmProb}, and boils down to showing that $\big|T_{n,q_n}\big(\overline{1}^{\lfloor n(1-q_n)\rfloor}\big)\big|=o_\mathbb{P}(\log n/(1-q_n))$, so that
\begin{align*}
    (1-q_n)\cdot\big|T_{n,q_n}\big(\overline{1}^{\lfloor n(1-q_n)\rfloor}\big)\big|=o_\mathbb{P}\big(\log\big|T_{n,q_n}\big(\overline{1}^{\lfloor n(1-q_n)\rfloor}\big)\big|\big)\, .
\end{align*}
This proposition, combined with the upper bound on the supremum, proves the upper bound of Theorem~\ref{thm:Prob} when $n(1-q_n)=\Theta(1)$.

To prove the corresponding lower bound, we use the simple inequality
\begin{align*}
    h(T_{n,q})\stackrel{\mathrm{d}}{=} h(T^B_n)
     &\geq d+h\Big(T^B_n\big(\overline{1}^d\big)\Big)\,,
\end{align*}
which holds for any $0\leq d\leq R^B_n+1$. Again taking $d=\lfloor n(1-q_n)\rfloor$, which is at most $R^B_n$ with high-probability, Proposition~\ref{prop:convRightSubtree} then yields that $h(T_{n,q_n})\geq n(1-q_n)+\big(c^*+o_\mathbb{P}(1)\big)\log n$, which is the lower bound of Theorem~\ref{thm:Prob} when $n(1-q_n)=\Theta(1)$.

\subsection{Distributional limits}\label{sec:DistrLimits}

In Section~\ref{sec:IBMandRD}, we described the strong connection between $h(T_{n,q_n})$ and the right depth $R^B_n$ whenever $(q_n)_{n\geq0}$ is such that $n(1-q_n)/\log n\rightarrow\infty$; in this regime, we can transfer many results on the asymptotic behaviour of $(R^B_n)_{n\geq0}$ to the sequence $\big(h(T_{n,q_n})\big)_{n\geq0}$.

If not only $n(1-q_n)/\log n\rightarrow\infty$ but, more strongly, $nq_n\rightarrow\lambda\in[0,\infty)$, then it is straightforward to prove that $\mathbb{P}\big(h(T^B_n)=R^B_n\big)=1-o(1)$. In this case, by studying the characteristic function of $n-1-R^B_n$, it follows fairly easily that
\begin{align*}
    n-1-R^B_n\overset{d}{\longrightarrow}\textsc{Poisson}(\lambda)\,,
\end{align*}
from which Theorem~\ref{thm:Poisson} follows. The details of this argument appear in Section~\ref{sec:Poisson}.

If we assume now that $n(1-q_n)/\log n\rightarrow\infty$ and $nq_n\rightarrow\infty$, by analyzing the moment generating function of $R^B_n$ given in Proposition~\ref{prop:BivariateGenFun}, we can prove a central limit theorem for the right depth; this is stated in the following proposition.

\begin{prop}\label{prop:CLTRD}
    Let $(q_n)_{n\geq0}$ be such that $n(1-q_n)/\log n\rightarrow\infty$ and $nq_n\rightarrow\infty$. Then
    \begin{align*}
        \frac{R^B_n-n(1-q_n)-\log\big((1-q_n)^{-1}\big)}{\sqrt{n(1-q_n)q_n}}\overset{\textrm{d}}{\longrightarrow}\textsc{Normal}(0,1).
    \end{align*}
\end{prop}

The proof can be found in Section~\ref{sec:CLTRD}. If we furthermore assume that $n(1-q_n)/(\log n)^2\rightarrow\infty$, then the $\log\big((1-q_n)^{-1}\big)$ term in the numerator of this proposition can be removed; in this case, using (\ref{eq:htnq_upper}) to compare $h(T_{n,q_n})$ with $R^B_n$ gives tight enough bounds to establish the conclusion of Theorem~\ref{thm:CLT}.

For the remaining regime of $(q_n)_{n\geq0}$, i.e. when $n(1-q_n)/\log n\rightarrow\infty$ and $n(1-q_n)/(\log n)^2=O(1)$, both terms $n(1-q_n)$ and $\log\big((1-q_n)^{-1}\big)$ contribute to the asymptotic behaviour of $R^B_n$. In this regime, the proof of the central limit theorem for $h(T_{n,q_n})$ requires a similar technique to the one of the proof of Theorem~\ref{thm:Prob} in the intermediate case (when $n(1-q_n)/\log n=\Theta(1)$), but with a different choice of $d$.

Previously, we deterministically chose $d=\lfloor n(1-q_n)\rfloor$. In the current setting, we instead require $d$ to be a random variable defined as follows. For $n\geq0$, let $m=m(n)$ be the smallest integer such that $m(1-q_n)+\log m\geq n(1-q_n)$ and let $D=D(n)=R^B_m+1$. The same chain of reasoning that yielded (\ref{eq:htnq_upper}) and (\ref{eq:second_height_bd}) now gives us the bounds
\begin{align}\label{eq:third_height_bd}
    h\Big(T^B_n\big(\overline{1}^{D(n)}\big)\Big)\leq h\big(T^B_n\big)-D(n)\leq\max\left\{\sup_{k\geq0}\Big\{h\Big(T^B\big(\overline{1}^k\overline{0}\big)\Big)-k\Big\},h\Big(T^B_n\big(\overline{1}^{D(n)}\big)\Big)\right\}\,.
\end{align}
It is not hard to see that $m(n)=n-(1+o(1))\frac{\log n}{1-q_n}$ and then we use Proposition~\ref{prop:CLTRD} to prove that
\begin{align*}
    \frac{R^B_{m(n)}-n(1-q_n)}{\sqrt{n(1-q_n)q_n}}\overset{\textrm{d}}{\longrightarrow}\textsc{Normal}(0,1)\,.
\end{align*}
Moreover, Proposition~\ref{prop:boundsLeftSubtree} straightforwardly implies that the supremum on the right hand side of (\ref{eq:third_height_bd}) is at most $c^*\log n+O_\mathbb{P}\left(\sqrt{\log n}\right)$. Finally, the next proposition gives the last ingredient to conclude the proof of Theorem~\ref{thm:CLT} in the case $n(1-q_n)/\log n\rightarrow\infty$ and $n(1-q_n)/(\log n)^2=O(1)$.

\begin{prop}\label{prop:convRenSubtree}
    Let $(q_n)_{n\geq0}$ be such that $\log\big(n(1-q_n)\big)=O\left(\sqrt{\log n}\right)$ and $n(1-q_n)=\omega\left(\sqrt{\log n}\right)$. For $n\geq0$, let $m=m(n)=\min\big\{\ell\geq0:\ell(1-q_n)+\log\ell\geq n(1-q_n)\big\}$. Then, the sequence of random variables
    \begin{align*}
        \left(\frac{h\Big(T^B_n\big(\overline{1}^{R^B_m+1}\big)\Big)-c^*\log n}{\sqrt{\log n}}\right)_{n\geq 2}
    \end{align*}
    is tight.
\end{prop}

The rather technical proof of Proposition~\ref{prop:convRenSubtree}  can be found in Section~\ref{sec:CLTHeight}. Proposition~\ref{prop:convRenSubtree} implies that $h\big(T^B_n(\overline{1}^{R^B_m+1})\big)/c^*\log n\rightarrow1$ in probability, which suggests some relation to Proposition~\ref{prop:convRightSubtree}. However, we did not see a simple way to give a unified statement. The reason for the hypothesis on $(q_n)_{n\geq0}$ in Proposition~\ref{prop:convRenSubtree} is mainly due to an error term of order $O\left(\sqrt{\log n}\right)$ in several of the asymptotic estimates which arise in our analysis.

Combining Proposition~\ref{prop:CLTRD} and \ref{prop:convRightSubtree}, the bounds in (\ref{eq:htnq_upper}) and (\ref{eq:third_height_bd}), and a subsequence argument as for Theorem~\ref{thm:Prob}, we can conclude the proof of Theorem~\ref{thm:CLT}. This also concludes the sketch of the proofs of the three theorems.

\section{Connection to random binary search trees}

In this section, we prove Theorem~\ref{thm:Prob} in the case when $(q_n)_{n\geq0}$ is such that $n(1-q_n)/\log n\rightarrow0$. As explained in Section~\ref{sec:MTandRBST}, the proof will be divided into upper and lower bound. We also prove the following proposition, which gives a bound on the second order term for the convergence in probability and will be useful for Theorem~\ref{thm:CLT}.

\begin{prop}\label{prop:convProbSmall}
    Let $(q_n)_{n\geq0}$ be taking values in $[0,1)$ such that $n(1-q_n)/\log n\rightarrow0$. Then, for any sequence $(\gamma_n)_{n\geq0}$ such that $\gamma_n/\big(n(1-q_n)\vee\sqrt{\log n}\big)\rightarrow\infty$, we have
    \begin{align*}
        \lim_{n\rightarrow\infty}\mathbb{P}\Big(\Big|h\big(T_{n,q_n}\big)-c^*\log n\Big|\geq\gamma_n\Big)=0.
    \end{align*}
\end{prop}

\subsection{A useful function}\label{sec:usefulFunction}

Before proving Proposition~\ref{prop:convProbSmall}, we introduce and study a very useful function, which will play a role in the analysis for all the regimes of $(q_n)_{n\geq0}$. For $n\geq0$, $q\in[0,1)$ and $\alpha\geq1$, let
\begin{align*}
    \mu_\alpha(n,q)&:=\sum_{1<k\leq n}\left(\frac{1-q}{1-q^k}\right)^\alpha\,.
\end{align*}
The following fact gives a first indication of the importance of $\mu_\alpha$.

\begin{fact}\label{fact:ExpRnmu}
    For all $n\geq0$ and $q\in[0,1)$, with $B=(B_{i,j})_{i,j\geq1}$ having independent $\textsc{Bernoulli}(1-q)$ entries, we have
    \begin{align*}
        \mathbb{E}\left[R^B_n\right]=\mu_1(n,q)
    \end{align*}
    and
    \begin{align*}
	\mathrm{Var}\left[R^B_n\right]&=\mu_1(n,q)-\mu_2(n,q)\,.
    \end{align*}
\end{fact}

More generally, there is a formula relating the expectation of $(R^B_n)^\alpha$ to the functions $\mu_i$ for $i\leq\alpha$, as given in the next proposition.

\begin{prop}\label{prop:MomentsOfRB}
    Let $n\geq0$, $q\in[0,1)$, and $B=(B_{i,j})_{i,j\geq1}$ have independent $\textsc{Bernoulli}(1-q)$ entries. Write $S_\beta=\{(s_1,\ldots,s_\beta):s_1+2s_2+\cdots+\beta s_\beta=\beta\}$ and for $s=(s_1,\ldots,s_\beta)\in S_\beta$, let $|s|=s_1+\cdots+s_\beta$. Then, for all integer $\alpha\geq1$, we have
    \begin{align*}
        \mathbb{E}\left[\big(R^B_n\big)^\alpha\right]=\sum_{1\leq\beta\leq\alpha\wedge(n-1)}\sum_{s\in S_\beta}(-1)^{\beta+|s|}\beta!\genfrac{\{}{\}}{0pt}{0}{\alpha}{\beta}\prod_{1\leq i\leq\beta}\frac{\mu_i(n,q)^{s_i}}{i^{s_i}s_i!}\,,
    \end{align*}
    where $\big(\genfrac{\{}{\}}{0pt}{1}{\alpha}{\beta}\big)_{\alpha,\beta\geq1}$ are Stirling numbers of the second kind.
\end{prop}

The proofs of Fact~\ref{fact:ExpRnmu} and Proposition~\ref{prop:MomentsOfRB} can be found in Appendix~\ref{app:MomentsRD}. Because $\mu_1(n,q)$ relates to the expected value of $R^B_n$, it can also be used to give concentration bounds on the right depth, as shown below.

\begin{lemma}\label{lem:ConcBoundRn}
    Fix $n\geq0$ and $q\in[0,1)$. Then, for any $c>1$, we have
    \begin{align*}
        \mathbb{P}\Big(R^B_n>c\mu_1(n,q)\Big)\leq\exp\bigg(\left[c\log\left(\frac{e}{c}\right)-1\right]\mu_1(n,q)\bigg)
    \end{align*}
    and
    \begin{align*}
        \mathbb{P}\Big(R^B_n<c^{-1}\mu_1(n,q)\Big)\leq\exp\bigg(\Big[c^{-1}\log\big(ce\big)-1\Big]\mu_1(n,q)\bigg)
    \end{align*}
\end{lemma}

\begin{proof}
    The proof simply follows from Chernoff's bounds. For all $t>0$, by Proposition~\ref{prop:BivariateGenFun}, we have
    \begin{align*}
        \mathbb{P}\Big(R^B_n>c\mu_1(n,q)\Big)&\leq\frac{\mathbb{E}\left[e^{tR^B_n}\right]}{e^{tc\mu_1(n,q)}}\\
        &=e^{-tc\mu_1(n,q)}\prod_{1<k\leq n}\left(1+\big(e^t-1\big)\frac{1-q}{1-q^k}\right).
    \end{align*}
    By the convexity of the exponential function, we have $1+x\leq e^x$, hence
    \begin{align*}
        \prod_{1<k\leq n}\left(1+\big(e^t-1\big)\frac{1-q}{1-q^k}\right)\leq\prod_{1<k\leq n}e^{\big(e^t-1\big)\frac{1-q}{1-q^k}}=\exp\Big((e^t-1)\mu_1(n,q)\Big).
    \end{align*}
    Using this bound in the previous equation, we obtain
    \begin{align*}
        \mathbb{P}\Big(R^B_n>c\mu_1(n,q)\Big)&\leq\exp\Big(-tc\mu_1(n,q)+\big(e^t-1\big)\mu_1(n,q)\Big)
    \end{align*}
    The optimal value for $t$ corresponds to $e^t=c$ and yields the first bound of the lemma. Similarly, for any $t>0$, we have
    \begin{align*}
        \mathbb{P}\Big(R^B_n<c^{-1}\mu_1(n,q)\Big)&\leq e^{tc^{-1}\mu_1(n,q)}\prod_{1<k\leq n}\left(1+\big(e^{-t}-1\big)\frac{1-q}{1-q^k}\right)\\
        &\leq\exp\Big(tc^{-1}\mu_1(n,q)+\big(e^{-t}-1\big)\mu_1(n,q)\Big),
    \end{align*}
    and the optimal bound is given by $e^{-t}=c^{-1}$, yielding the second bound.
\end{proof}

In order to use these results, we need to better understand the behaviour of $\mu_\alpha$. The following two propositions describe the asymptotic growth of $\mu_1$, and of $\mu_\alpha$ for $\alpha>1$, respectively.

\begin{prop}\label{prop:ConvOfMu}
    Consider any sequence $(q_n)_{n\geq0}$ taking values in $[0,1)$. Then we have
    \begin{align*}
        \mu_1(n,q_n)&=n(1-q_n)+\log\left(n\wedge\frac{1}{1-q_n}\right)+O\left(\sqrt{\log\left(n\wedge\frac{1}{1-q_n}\right)}\right)\,.
    \end{align*}
\end{prop}

\begin{prop}\label{prop:ConvOfMuAlpha}
    Let $\alpha>1$. Consider any sequence $(q_n)_{n\geq0}$ taking values in $[0,1)$. Then we have
    \begin{align*}
        \mu_\alpha(n,q_n)=n(1-q_n)^\alpha+\zeta(\alpha)-1+O\left(\left((1-q_n)\vee\frac{1}{n}\right)^\frac{\alpha-1}{\alpha+1}\right)=n(1-q_n)^\alpha+O(1)\,,
    \end{align*}
    where $\zeta(\alpha)$ is the Riemann zeta function: $\zeta(\alpha)=\sum_{k\geq1}\frac{1}{k^\alpha}$.
\end{prop}

The second equality in Proposition~\ref{prop:ConvOfMuAlpha} directly follows from the first one as $(1-q_n)\vee\frac{1}{n}\leq 1$. The proofs of these two propositions can be found in Appendix~\ref{app:asymptoticMu}. Since the details of the proofs are tedious, we provide a brief sketch of the idea. The proof technique heavily relies on the fact that $\frac{1-q}{1-q^k}\sim\frac{1}{k}$ whenever $k(1-q)=o(1)$, and that $\frac{1-q}{1-q^k}\sim(1-q)$ whenever $k(1-q)=\omega(1)$. Pretending for the moment that the first asymptotic holds whenever $k(1-q)\leq1$ and that the second one holds whenever $k(1-q)>1$, and assuming that $1\ll\frac{1}{1-q}\ll n$, it follows that
\begin{align*}
    \mu_\alpha(n,q)&=\sum_{1<k\leq n}\left(\frac{1-q}{1-q^k}\right)^\alpha\\
    &\simeq\sum_{1<k\leq\frac{1}{1-q}}\frac{1}{k^\alpha}+\sum_{\frac{1}{1-q}<k\leq n}(1-q)^\alpha\\
    &\simeq n(1-q)^\alpha-1+\sum_{1\leq k\leq\frac{1}{1-q}}\frac{1}{k^\alpha}\,;
\end{align*}
the second term asymptotically simplifies to $\log\left(\frac{1}{1-q}\right)$ when $\alpha=1$ and to $\zeta(\alpha)$ when $\alpha>1$.

The fact that when $\alpha>1$, the second order term of the development of $\mu_\alpha$ is $\zeta(\alpha)-1$, is interesting to us, and we can imagine it has already appeared in the literature; however, we were unable to find a reference.

\subsection{Upper tail bound}\label{sec:UpperBoundSmall}

To prove the upper bound of Proposition~\ref{prop:convProbSmall}, we start by proving Proposition~\ref{prop:StochBounds}, as this stochastic bound is at the heart of the proof.

\begin{proof}[Proof of Proposition~\ref{prop:StochBounds}]
    We must show that, for all $v\in T_\infty$, $n\geq0$ and $q\in[0,1]$, we have
    \begin{align*}
        \mathbb{P}\big(v\in T_{n,q}\big)\leq\mathbb{P}\big(\overline{1}^{|v|}\in T_{n,q}\big)\,.
    \end{align*}
    First of all, if $q=0$, then $T_{n,q}=\{\overline{1}^k, 0\leq k\leq n\}$, and the inequality is clearly true. Secondly, if $q=1$, then $T_{n,q}$ is a random binary search tree, and by symmetry we have
    \begin{align*}
        \mathbb{P}(v\in T_{n,q})&=\mathbb{P}(u\in T_{n,q})
    \end{align*}
    for all $u$, $v$ such that $|u|=|v|$, which also proves the inequality.
    
    Fix now $q\in(0,1)$. We prove the result by induction on $n$. For $n\leq1$ the assertion holds because either $T_{n,q}$ is empty or $T_{n,q}=\{\root\}$. Consider now some node $v\in T_\infty$. Write $T^L_{n,q}$ and $T^R_{n,q}$ respectively for the left and right subtrees of $T_{n,q}$ re-rooted at $\root$; in other words, $T_{n,q}(\overline{0})=\overline{0}T^L_{n,q}$ and $T_{n,q}(\overline{1})=\overline{1}T^R_{n,q}$.
    
    Assume first that $v=\overline{1}v'$. In this case, $v\in T_{n,q}$ if and only if $v'\in T^R_{n,q}$. Moreover, by Proposition~\ref{prop:ProjCons}, we know that, conditioned on its size, $T^R_{n,q}$ is a Mallows tree; since $\big|T^R_{n,q}\big|<n$, by induction we thus have
    \begin{align*}
        \mathbb{P}\big(v\in T_{n,q}\big)=\mathbb{P}\big(v'\in T^R_{n,q}\big)\leq\mathbb{P}\big(\overline{1}^{|v'|}\in T^R_{n,q}\big)=\mathbb{P}\big(\overline{1}^{|v|}\in T_{n,q}\big)\,,
    \end{align*}
    which proves the statement in this case.
    
    Assume next that $v=\overline{0}v'$. In this case, $v\in T_{n,q}$ if and only if $v'\in T^L_{n,q}$. Moreover, by Proposition~\ref{prop:ProjCons}, for $0\leq k\leq n-1$, we have
    \begin{align*}
        \mathbb{P}\Big(\big|T^L_{n,q}\big|\leq k\Big)&=\frac{1-q^{k+1}}{1-q^n}\geq q^{n-k-1}\frac{1-q^{k+1}}{1-q^n}=\frac{q^{n-k-1}-q^n}{1-q^n}=\mathbb{P}\Big(\big|T^L_{n,q}\big|\geq n-k-1\Big).
    \end{align*}
    Since $\mathbb{P}\big(|T^L_{n,q}|\geq n-k-1\big)=\mathbb{P}\big(|T^R_{n,q}|\leq k\big)$, it follows that $\mathbb{P}\big(|T^L_{n,q}|\leq k\big)\geq\mathbb{P}\big(|T^R_{n,q}|\leq k\big)$, which means that $\big|T^L_{n,q}\big|\preceq\big|T^R_{n,q}\big|$. By Corollary~\ref{cor:TnqIncreasing}, it follows that
    \begin{align*}
        \mathbb{P}\big(v\in T_{n,q}\big)=\mathbb{P}\big(v'\in T^L_{n,q}\big)&\leq\mathbb{P}\big(v'\in T^R_{n,q}\big)\leq\mathbb{P}\big(\overline{1}^v\in T_{n,q}\big)\,,
    \end{align*}
    proving the assertion in this case.
    
    Finally, if $v=\root$, then $v=\overline{1}^{|v|}$ and $\mathbb{P}\big(v\in T_{n,q}\big)=\mathbb{P}\big(\overline{1}^{|v|}\in T_{n,q}\big)$. This completes the proof.
\end{proof}

Before proving Proposition~\ref{prop:convProbSmall}, we state and prove the following proposition, which allows us to extend the results from convergence in probability to convergence in $L_p$ in Theorem~\ref{thm:Prob}.

\begin{prop}\label{prop:UI}
    For any sequence $(q_n)_{n\geq0}$ and any $p>0$, the family of random variables 
    \begin{align*} 
        \left(\left(\frac{h(T_{n,q_n})}{n(1-q_n)+c^*\log n}\right)^p\right)_{n\geq1}
    \end{align*}
    is uniformly integrable.
\end{prop}

\begin{proof}
    First of all, by Proposition~\ref{prop:ConvOfMu} we have $n(1-q_n)+c^*\log n=\Theta\big(\mu_1(n,q_n)\big)$; hence, for simplification and without loss of generality, we now prove that $\big(h(T_{n,q_n})/\mu_1(n,q_n)\big)^p$ is uniformly integrable.
    
    Fix $a\in\mathbb{R}$. Combine Proposition~\ref{prop:StochBounds} with a union bound to obtain
    \begin{align}
        \mathbb{P}\left(\left(\frac{h(T_{n,q_n})}{\mu(n,q_n)}\right)^p\geq a\right)&=\mathbb{P}\Big(h(T_{n,q_n})\geq a^\frac{1}{p}\mu(n,q_n)\Big)\notag\\
        &=\mathbb{P}\Big(\exists v\in T_{n,q_n}\textrm{ such that }|v|=\big\lceil a^\frac{1}{p}\mu_1(n,q_n)\big\rceil\Big)\notag\\
        &\leq2^{a^\frac{1}{p}\mu(n,q_n)+1}\mathbb{P}\Big(R^B_n\geq a^\frac{1}{p}\mu_1(n,q_n)\Big)\label{eq:boundUI}\,.
    \end{align}
    Using Chernoff's bound together with the moment generating function of $R^B_n$ from Proposition~\ref{prop:BivariateGenFun}, we have
    \begin{align*}
        \mathbb{P}\Big(R^B_n\geq a^\frac{1}{p}\mu_1(n,q_n)\Big)&\leq e^{-ta^\frac{1}{p}\mu_1(n,q_n)}\mathbb{E}\left[e^{tR^B_n}\right]\leq\exp\Big(-ta^\frac{1}{p}\mu_1(n,q_n)+(e^t-1)\mu_1(n,q_n)\Big)\,.
    \end{align*}
    Taking $t=\log2+1$ and using this bound in (\ref{eq:boundUI}), it follows that
    \begin{align*}
        \mathbb{P}\left(\left(\frac{h(T_{n,q_n})}{\mu(n,q_n)}\right)^p\geq a\right)&\leq2\exp\Big(-\big[a^\frac{1}{p}+1-2e\big]\mu_1(n,q_n)\Big)\,.
    \end{align*}
    Thus, for $a>(2e-1)^p$, we have
    \begin{align*}
        \lim_{n\rightarrow\infty}\mathbb{P}\left(\left(\frac{h(T_{n,q_n})}{\mu(n,q_n)}\right)^p\geq a\right)=0\,,
    \end{align*}
    which proves the claimed uniform integrability.
\end{proof}

This result implies that the convergence in $L_p$ in Theorem~\ref{thm:Prob} follows from the convergence in probability. Therefor, from now on we can only focus our attention on proving the latter type of convergence.

Using a similar argument as for the previous proof, we conclude this section by proving the upper tail bound of Proposition~\ref{prop:convProbSmall}.

\begin{proof}[Proof of Proposition~\ref{prop:convProbSmall} (Upper Tail)]
    Fix $n\geq0$ and $h=h_n>\mu_1(n,q_n)$. By combining Proposition~\ref{prop:StochBounds} with a union bound, we have
    \begin{align*}
        \mathbb{P}\Big(h\big(T_{n,q_n}\big)\geq h\Big)&=\mathbb{P}\Big(\exists v\in T_\infty:|v|=h, v\in T_{n,q_n}\Big)\\
        &\leq2^h\mathbb{P}\Big(\overline{1}^h\in T_{n,q_n}\Big)\\
        &=2^h\mathbb{P}\Big(R^B_n\geq h\Big)\,.
    \end{align*}
    Because $h>\mu_1(n,q_n)$, it follows from Lemma~\ref{lem:ConcBoundRn} that
    \begin{align*}
        \mathbb{P}\Big(R^B_n\geq h\Big)&\leq\exp\left(\left[\frac{h}{\mu_1(n,q_n)}\log\left(\frac{e\mu_1(n,q_n)}{h}\right)-1\right]\mu_1(n,q_n)\right)\,,
    \end{align*}
    hence
    \begin{align}\label{eq:unifIntHeight}
        \mathbb{P}\Big(h\big(T_{n,q_n}\big)\geq h\Big)&\leq\exp\left(h\log\left(\frac{2e\mu_1(n,q_n)}{h}\right)-\mu_1(n,q_n)\right)\,.
    \end{align}

    Fix a sequence $(\gamma_n)_{n\geq0}$ such that $\gamma_n=\omega\big(n(1-q_n)\vee\sqrt{\log n}\big)$ and with $\gamma_n=o(\log n)$; this second hypothesis on $(\gamma_n)_{n\geq0}$ suffices to prove Proposition~\ref{prop:convProbSmall} as the probabilities we aim to bound are decreasing functions of $\gamma_n$. By Proposition~\ref{prop:ConvOfMu}, we have $\mu_1(n,q_n)=n(1-q_n)+\log n+O\left(\sqrt{\log n}\right)=\log n+o(\gamma_n)$. Taking $h_n=\lfloor c^*\log n+\gamma_n\rfloor$ and since $c^*>1$, we thus have that $h_n>\mu_1(n,q_n)$ for $n$ large enough, so (\ref{eq:unifIntHeight}) gives
    \begin{align}
        \mathbb{P}\Big(h\big(T_{n,q_n}\big)\geq h_n\Big)&\leq\exp\left(h_n\log\left(\frac{2e\mu_1(n,q_n)}{h_n}\right)-\mu_1(n,q_n)\right)\,.\label{eq:boundsHeightSmall}
    \end{align}
    To bound the right hand side of (\ref{eq:boundsHeightSmall}), we note that $c^*\mu_1(n,q_n)=c^*\log n+o(\gamma_n)=h_n-(1+o(1))\gamma_n$, so
    \begin{align*}
        \log\left(\frac{2e\mu_1(n,q_n)}{h_n}\right)&=\log\left(\frac{2e}{c^*}\cdot\frac{c^*\mu_1(n,q_n)}{h_n}\right)\\
        &=\log\left(\frac{2e}{c^*}\right)+\log\left(1-(1+o(1))\frac{\gamma_n}{h_n}\right)\\
        &=\frac{1}{c^*}-(1+o(1))\frac{\gamma_n}{h_n}\,,
    \end{align*}
    the last identity holding since $c^*\log\left(\frac{2e}{c^*}\right)=1$ and since $\gamma_n=o(h_n)$. Together with (\ref{eq:boundsHeightSmall}), this yields
    \begin{align*}
        \mathbb{P}\Big(h\big(T_{n,q_n}\big)\geq h_n\Big)&\leq\exp\left(\frac{h_n}{c^*}-(1+o(1))\gamma_n-\mu_1(n,q_n)\right)=\exp\left(\left(\frac{1}{c^*}-1+o(1)\right)\gamma_n\right)\,.
    \end{align*}
    Because $h_n\leq c^*\log n+\gamma_n$ and $\gamma_n\rightarrow\infty$ this concludes the proof of the upper bound.
\end{proof}

\subsection{Lower tail bound}\label{sec:lowerBoundSmall}

We now prove the lower tail bound of Proposition~\ref{prop:convProbSmall}. In order to do so, we use the coupling explained in Section~\ref{sec:MTandRBST}: $S_{n,q}(\root)=n$ and given $S_{n,q}(v)$, we have
\begin{align*}
    S_{n,q}(v\overline{0})&=\left\{\begin{array}{ll}
        \left\lfloor\frac{\log\left(1-U_v(1-q^{S_{n,q}(v)})\right)}{\log q}\right\rfloor & \textrm{if $q\in(0,1)$} \\
        \lfloor S_{n,q}(v)U_v\rfloor & \textrm{if $q=1$} \\
        0 & \textrm{if $q=0$}
    \end{array}\right.
\end{align*}
and
\begin{align*}
    S_{n,q}(v\overline{1})=S_{n,q}(v)-1-S_{n,q}(v\overline{0})\,.
\end{align*}
As we saw in Section~\ref{sec:MTandRBST}, it follows from Proposition~\ref{prop:ProjCons} that $\{v\in T_\infty:S_{n,q}(v)\geq1\}$ is $\mt(n,q)$-distributed.

For any node $v\in T_\infty$, write $X_{(v,v\overline{0})}=U_v$ and $X_{(v,v\overline{1})}=1-U_v$. and let $P_v=\prod_{e\prec v}X_e$, where $e\prec v$ denotes the set of edges on the path from $\root$ to $v$. The following proposition gives useful bounds for $S_{n,q}$ using $P$.

\begin{lemma}\label{lem:LowBoundS}
    Let $n\geq1$ and $q\in(0,1)$. Then, a.s. for all $v\in T_\infty$, we have
    \begin{align*}
        \frac{\log\big(1-P_v(1-q^n)\big)}{\log q}-|v|\leq S_{n,q}(v)\leq n-\frac{\log\big(q^n+P_v(1-q^n)\big)}{\log q}\,.
    \end{align*}
    Moreover, these inequalities naturally extend to $q=1$ as follows: almost surely
    \begin{align*}
        nP_v-|v|\leq S_{n,1}(v)\leq nP_v\,.
    \end{align*}
\end{lemma}

\begin{proof}
    We only consider the case $q\in(0,1)$ as in the case $q=1$ the assertion was already proven in \cite{devroye1986note}, and can be obtained from our formula by taking the limit $q\rightarrow1$.
    
    We prove this lemma by induction on the depth of $v$. For $v=\root$, $P_v=1$ and the inequalities hold. Assume now it holds for some $v\in T_\infty$. We start by proving that the lower bound holds for the left child of $v$. By definition, we know that
    \begin{align*}
        S_{n,q}(v\overline{0})&=\left\lfloor\frac{\log\left(1-U_v\left(1-q^{S_{n,q}(v)}\right)\right)}{\log q}\right\rfloor\geq\frac{\log\left(1-U_v\left(1-q^{S_{n,q}(v)}\right)\right)}{\log q}-1\,,
    \end{align*}
    and the right hand side is increasing in $S_{n,q}(v)$. By the induction hypothesis, we know that $S_{n,q}(v)\geq\frac{\log\big(1-P_v(1-q^n)\big)}{\log q}-|v|$. Replacing $S_{n,q}(v)$ by its lower bound into the previous formula gives us
    \begin{align*}
        S_{n,q}(v\overline{0})&\geq\frac{\log\left(1-U_v\left(1-q^{\frac{\log\left(1-P_v(1-q^n)\right)}{\log q}-|v|}\right)\right)}{\log q}-1\\
        &=\frac{\log\left(1-U_v+U_vq^{-|v|}\left(1-P_v(1-q^n)\right)\right)}{\log q}-1.
    \end{align*}
    Since $\log q<0$ and $(1-U_v)\leq q^{-|v|}(1-U_v)$, it follows that
    \begin{align*}
        S_{n,q}(v\overline{0})&\geq\frac{\log\left(q^{-|v|}(1-U_v)+U_vq^{-|v|}\left(1-P_v(1-q^n)\right)\right)}{\log q}-1\\
        &=\frac{\log\left(1-U_vP_v(1-q^n)\right)}{\log q}-|v|-1,
    \end{align*}
    which is the desired lower bound.
    
    We now prove that the upper bound holds for the right child of $v$. Using the definition again, we have
    \begin{align*}
        S_{n,q}(v\overline{1})&=S_{n,q}(v)-1-S_{n,q}(v\overline{0})\\
        &=S_{n,q}(v)-1-\left\lfloor\frac{\log\left(1-U_v(1-q^{S_{n,q}(v)})\right)}{\log q}\right\rfloor\\
        &=\left\lceil S_{n,q}(v)-\frac{\log\left(1-U_v(1-q^{S_{n,q}(v)})\right)}{\log q}\right\rceil-1\\
        &\overset{\textrm{a.s.}}{=}\left\lfloor S_{n,q}(v)-\frac{\log\left(1-U_v(1-q^{S_{n,q}(v)})\right)}{\log q}\right\rfloor\,,
    \end{align*}
    the last equality following from the fact that $\frac{\log\left(1-U_v(1-q^{S_{n,q}(v)})\right)}{\log q}$ is a.s. not an integer. From this equality, we obtain
    \begin{align*}
        S_{n,q}(v\overline{1})&\leq S_{n,q}(v)-\frac{\log\left(1-U_v\left(1-q^{S_{n,q}(v)}\right)\right)}{\log q}\,,
    \end{align*}
    and the right hand side is increasing in $S_{n,q}(v)$, which can be checked by direct computation. Applying the induction hypothesis, we obtain
    \begin{align*}
        S_{n,q}(v\overline{1})&\leq n-\frac{\log\big(q^n+P_v(1-q^n)\big)}{\log q}-\frac{\log\left(1-U_v\left(1-q^{n-\frac{\log\big(q^n+P_v(1-q^n)\big)}{\log q}}\right)\right)}{\log q}\\
        &=n-\frac{\log\big(q^n+(1-U_v)P_v(1-q^n)\big)}{\log q},
    \end{align*}
    which is the desired upper bound.
    
    For the two remaining bounds, note that, for any integer $S\geq1$ and $U\in(0,1)$, we have
    \begin{align}
        &\left(S-\frac{\log\left(1-U\left(1-q^{S}\right)\right)}{\log q}\right)-\left(\frac{\log\left(1-(1-U)\left(1-q^{S}\right)\right)}{\log q}\right)\notag\\
        &\hspace{1cm}=\frac{S\log q-\log\Big(q^{S}+U(1-U)\big(1-q^{S}\big)^2\Big)}{\log q}\notag\\
        &\hspace{1cm}=\frac{1}{\log\frac{1}{q}}\log\left(1+\frac{U(1-U)\big(1-q^{S}\big)^2}{q^{S}}\right)\geq0.\label{eq:Sconnections}
    \end{align}
    From this inequality, it follows that
    \begin{align*}
        S_{n,q}(v\overline{1})&=\left\lfloor S_{n,q}(v)-\frac{\log\left(1-U_v\left(1-q^{S_{n,q}(v)}\right)\right)}{\log q}\right\rfloor\\
        &\geq\left\lfloor\frac{\log\left(1-(1-U_v)\left(1-q^{S_{n,q}(v)}\right)\right)}{\log q}\right\rfloor,
    \end{align*}
    and then the same technique as the one used to bound $S_{n,q}(v\overline{0})$ from below gives us that
    \begin{align*}
        S_{n,q}(v\overline{1})&\geq\frac{\log\left(1-(1-U_v)P_v(1-q^n)\right)}{\log q}-|v|-1,
    \end{align*}
    which yields the desired lower bound. Similarly, (\ref{eq:Sconnections}) gives that
    \begin{align*}
        S_{n,q}(v\overline{0})&=\left\lfloor\frac{\log\left(1-U_v\left(1-q^{S_{n,q}(v)}\right)\right)}{\log q}\right\rfloor\\
        &\leq\left\lfloor S_{n,q}(v)-\frac{\log\left(1-(1-U_v)\left(1-q^{S_{n,q}(v)}\right)\right)}{\log q}\right\rfloor,
    \end{align*}
    from which we deduce that
    \begin{align*}
        S_{n,q}(v\overline{0})&\leq n-\frac{\log\big(q^n+U_vP_v(1-q^n)\big)}{\log q}
    \end{align*}
    by the same technique used to bound $S_{n,q}(v\overline{1})$ from above. This last inequality concludes the induction and the proof of the lemma.
\end{proof}

With this results, we can now compare $(S_{n,q})$ for different values of $n$ and $q$ using $P$. This leads to Proposition~\ref{prop:couplingMallowsRBST}.

\begin{proof}[Proof of Proposition~\ref{prop:couplingMallowsRBST}]
    We must prove that
    \begin{align*}
        \mathbb{P}\big(h(T_{n,q})\leq\ell\big)\leq\mathbb{P}\big(h(T_{m,1})\leq\ell\big)
    \end{align*}
    where $m=\left\lfloor\frac{1-q^n}{1-q^{\ell+2}}\right\rfloor$. First of all, if $q=0$, then $h(T_{n,q})=n$ and $m=\left\lfloor\frac{1-q^n}{1-q^{\ell+2}}\right\rfloor=1$, so the inequality holds. Assume now that $q\in(0,1)$. Because $\{v\in T_\infty:S_{n,q}(v)\geq1\}$ is $\mt(n,q)$-distributed, we have
    \begin{align*}
        \mathbb{P}\Big(h\big(T_{n,q}\big)\leq\ell\Big)&=\mathbb{P}\Big(\forall v\in T_\infty\textrm{ with }|v|=\ell+1:S_{n,q}(v)<1\Big)\,.
    \end{align*}
    Using the lower bound in Lemma~\ref{lem:LowBoundS}, this implies that
    \begin{align*}
        \mathbb{P}\Big(h\big(T_{n,q}\big)\leq\ell\Big)&\leq\mathbb{P}\left(\forall v\in T_\infty\textrm{ with }|v|=\ell+1:\frac{\log\big(1-P_v(1-q^n)\big)}{\log q}-|v|<1\right)\\
        &=\mathbb{P}\left(\forall v\in T_\infty\textrm{ with }|v|=\ell+1:P_v<\frac{1-q^{\ell+2}}{1-q^n}\right)\,.
    \end{align*}
    
    Consider now some $m\geq1$ and use the same technique but with the upper bound of Lemma~\ref{lem:LowBoundS}:
    \begin{align*}
        \mathbb{P}\Big(h\big(T_{m,1}\big)\leq\ell\Big)&=\mathbb{P}\Big(\forall v\in T_\infty\textrm{ with }|v|=\ell+1:S_{m,1}(v)<1\Big)\\
        &\geq\mathbb{P}\left(\forall v\in T_\infty\textrm{ with }|v|=\ell+1:P_v<\frac{1}{m}\right)\,.
    \end{align*}
    The proposition now follows by defining $m=\left\lfloor\frac{1-q^n}{1-q^{\ell+2}}\right\rfloor$, since $\frac{1}{m}\geq\frac{1-q^{\ell+2}}{1-q^n}$.
\end{proof}

We now have all the results necessary to prove the lower bound of Proposition~\ref{prop:convProbSmall}.

\begin{proof}[Proof of Proposition~\ref{prop:convProbSmall} (Lower Bound)]
    Consider $(q_n)_{n\geq0}$ and $(\gamma_n)_{n\geq0}$ as in the proposition and let $m_n=\left\lfloor\frac{1-(q_n)^n}{1-(q_n)^{\ell_n+2}}\right\rfloor$ where $\ell_n=\lfloor c^ *\log n-\gamma_n\rfloor$. We again assume without loss of generality that $\gamma_n=o(\log n)$. 
    
    Let us first study the asymptotic behaviour of $\log m_n$. Note that $\left\lfloor\frac{1-q^n}{1-q^{\ell+2}}\right\rfloor$ is monotone in $q$ provided $\ell+2\leq n$. Because $n(1-q_n)<\log n$ for $n$ large enough, it follows that
    \begin{align*}
        \frac{1-\left(1-\frac{\log n}{n}\right)^n}{1-\left(1-\frac{\log n}{n}\right)^{\ell_n+2}}-1\leq m_n\leq\lim_{q\rightarrow1}\frac{1-q^n}{1-q^{\ell+2}}=\frac{n}{\ell_n+2}\,.
    \end{align*}
    Moreover, $\left(1-\frac{\log n}{n}\right)^n=o(1)$ and $\left(1-\frac{\log n}{n}\right)^{\ell_n+2}=e^{-(1+o(1))\frac{c^*(\log n)^2}{n}}=1-(1+o(1))\frac{c^*(\log n)^2}{n}$, which implies that
    \begin{align*}
        \log\left(\frac{1-\left(1-\frac{\log n}{n}\right)^n}{1-\left(1-\frac{\log n}{n}\right)^{\ell_n+2}}-1\right)=\log\left((1+o(1))\frac{n}{c^*(\log n)^2}-1\right)=\log n+O(\log\log n)\,.
    \end{align*}
    It follows that
    \begin{align*}
        \log n+O(\log\log n)\leq\log m_n\leq\log n-\log(\ell_n+2)=\log n+O(\log\log n)\,,
    \end{align*}
    hence $\log m_n=\log n+O(\log\log n)$.
    
    Let us now study $(T_{m_n,1})_{n\geq0}$. Recall that when $q=1$, a $\mt(n,q)$-distributed tree has the distribution of a random binary search tree. We use the results of Reed~\cite{reed2003height} and Drmota~\cite{drmota2003analytic}, who prove that
    \begin{align*}
        \mathbb{E}\left[h\big(T_{m_n,1}\big)\right]&=c^*\log m_n+O\left(\log\log m_n\right)
    \end{align*}
    and
    \begin{align*}
        \mathrm{Var}\left[h\big(T_{m_n,1}\big)\right]&=O(1).
    \end{align*}
    Next,
    \begin{align*}
        \mathbb{P}\Big(h\big(T_{m_n,1}\big)\leq\ell_n\Big)&=\mathbb{P}\Big(h\big(T_{m_n,1}\big)-\mathbb{E}\left[h\big(T_{m_n,1}\big)\right]\leq\ell_n-\mathbb{E}\left[h\big(T_{m_n,1}\big)\right]\Big)\,,
    \end{align*}
    and
    \begin{align*}
        \ell_n-\mathbb{E}\left[h\big(T_{m_n,1}\big)\right]=c^*\log n-\gamma_n-c^*\log m_n+O(\log\log m_n)=(-1+o(1))\gamma_n\,,
    \end{align*}
    the last equality holding since $\gamma_n=\omega\left(\sqrt{\log n}\right)=\omega(\log\log n)$. Applying Chebyshev's inequality, we obtain
    \begin{align*}
        \mathbb{P}\Big(h\big(T_{m_n,1}\big)\leq\ell_n\Big)&\leq\frac{\mathrm{Var}\left[h\big(T_{m_n,1}\big)\right]}{\left(\ell_n-\mathbb{E}\left[h\big(T_{m_n,1}\big)\right]\right)^2}=O\left(\frac{1}{(\gamma_n)^2}\right)=o(1)\,.
    \end{align*}
    
    The lower bound of Proposition~\ref{prop:convProbSmall} follows by applying Proposition~\ref{prop:couplingMallowsRBST} to obtain
    \begin{align*}
        \mathbb{P}\Big(h\big(T_{n,q_n}\big)\leq c^*\log n-\gamma_n\Big)=\mathbb{P}\Big(h\big(T_{n,q_n}\big)\leq\ell_n\Big)\leq\mathbb{P}\Big(h\big(T_{m_n,1}\big)\leq\ell_n\Big)=o(1)\,.
    \end{align*}
\end{proof}

\section{Right depth of Mallows trees}

In this section, we study the right depth $R^B_n$ of a Mallows tree and use its properties to prove Theorem~\ref{thm:AS}.

\subsection{General results}\label{sec:IBMGeneralResults}

We start by proving the general results regarding the infinite $b$-model that was defined in Section~\ref{sec:IBMandRD}.

\begin{proof}[Proof of Proposition~\ref{prop:IBM}]
    Fix $n\geq0$ and $q\in[0,1)$, and let $B=(B_{i,j})$ have independent $\textsc{Bernoulli}(1-q)$ entries. We want to prove that $\sigma^B_n$ is $\pi_{n,q}$-distributed. For any $\sigma\in\mathcal{S}_n$ define
    \begin{align*}
        \theta(\sigma):=\big\{f:[n]\rightarrow\Z^+:\forall i\in[n],\textrm{rank}\big(f(i),\{f(1),\ldots,f(n)\}\big)=\sigma(i)\big\}
    \end{align*}
    for the set of functions $f$ with the same ordering as $\sigma$. By definition, we have
    \begin{align}
        \mathbb{P}\big(\sigma^B_n=\sigma\big)&=\mathbb{P}\big(f^B_n\in\theta(\sigma)\big)\notag\\
        &=\sum_{f\in\theta(\sigma)}\mathbb{P}\Big(\forall i\in[n],f^B(i)=f(i)\Big)\notag\\
        &=\sum_{f\in\theta(\sigma)}\prod_{1\leq i\leq n}\mathbb{P}\Big(f^B(i)=f(i)\,\Big|\,\forall j\in[i-1],f^B(j)=f(j)\Big)\,.\label{eq:probSigmaB}
    \end{align}
    
    Recall that $F^B_{i-1}=\{f^B(1),\ldots,f^B(i-1)\}$ and note that $\big|[f(i)-1]\setminus F^B_{i-1}\big|=f(i)-1-\big|\big\{j\in[i-1]:f^B(j)<f(i)\big\}\big|$. By using the definition $f^B(i)=\inf\big\{j\in\Z^+\setminus F^B_{i-1}:B_{i,j}=1\big\}$, it follows that
    \begin{align*}
        &\mathbb{P}\Big(f^B(i)=f(i)\,\Big|\,\forall j\in[i-1],f^B(j)=f(j)\Big)\\
        &\hspace{1cm}=\mathbb{P}\Big(\forall \ell\in[f(i)-1]\setminus F^B_{i-1},B_{i,\ell}=0\textrm{ and }B_{i,f(i)}=1\,\Big|\,\forall j\in[i-1],f^B(j)=f(j)\Big)\\
        &\hspace{1cm}=q^{f(i)-1-|\{j\in[i-1]:f(j)<f(i)\}|}(1-q)\,.
    \end{align*}
    Plugging this result back into (\ref{eq:probSigmaB}), we obtain
    \begin{align*}
        \mathbb{P}\big(\sigma^B_n=\sigma\big)&=\sum_{f\in\theta(\sigma)}\prod_{1\leq i\leq n}\Big(q^{f(i)-1-|\{j\in[i-1]:f(j)<f(i)\}|}(1-q)\Big)\\
        &=\sum_{f\in\theta(\sigma)}q^{\sum_{1\leq i\leq n}\big(f(i)-1-|\{j\in[i-1]:f(j)<f(i)\}|\big)}(1-q)^n\,.
    \end{align*}
    The sum in the power can be divided into two parts:
    \begin{align*}
        &\sum_{1\leq i\leq n}\big(f(i)-1-\big|\{j\in[i-1]:f(j)<f(i)\}\big|\big)\\
        &\hspace{1cm}=\sum_{1\leq i\leq n}\big(f(i)-1\big)-\sum_{1\leq i\leq n}\big|\{j\in[i-1]:f(j)<f(i)\}\big|\,;
    \end{align*}
    and because $f$ has the same ordering as $\sigma$, the second sum can be rewritten as
    \begin{align*}
        \sum_{1\leq i\leq n}\big|\{j\in[i-1]:f(j)<f(i)\}\big|&=\Big|\Big\{(i,j)\in[n]^2:j<i\textrm{ and }f(j)<f(i)\Big\}\Big|=\binom{n}{2}-\textrm{Inv}(\sigma)\,.
    \end{align*}
    This proves that (\ref{eq:probSigmaB}) can be rewritten as
    \begin{align*}
        \mathbb{P}\big(\sigma^B_n=\sigma\big)&=\sum_{f\in\theta(\sigma)}q^{\textrm{Inv}(\sigma)-\binom{n}{2}+\sum_{1\leq i\leq n}f(i)-1}(1-q)^n\\
        &=q^{\textrm{Inv}(\sigma)}\left[\sum_{f\in\theta(\sigma)}q^{-\frac{n(n+1)}{2}+\sum_{1\leq i\leq n}f(i)}(1-q)^n\right]\,.
    \end{align*}
    To conclude the proof, note that the term in brackets is independent of $\sigma$ just by reordering the $f(i)$ in the sum. This means that $\mathbb{P}(\sigma^B_n=\sigma)$ is proportional to $q^{\textrm{Inv}(\sigma)}$, and so $\sigma^B_n$ is $\pi_{n,q}$-distributed, as required.
\end{proof}

With Proposition~\ref{prop:IBM} proven, we can now use this equivalent model to study Mallows permutations and trees. Moreover, Corollary~\ref{cor:TnqIncreasing} directly follows from this result and the increasing property of the infinite $b$-model explained in Section~\ref{sec:IBMandRD}.

The infinite $b$-model is also used to prove Lemma~\ref{lem:LeftSubtreesDistribution} and Proposition~\ref{prop:BivariateGenFun}.

\begin{proof}[Proof of Lemma~\ref{lem:LeftSubtreesDistribution}]
    Fix $q\in[0,1)$ and $B$ with independent $\textsc{Bernoulli}(1-q)$ entries. We study the distribution of $T^B\big(\overline{1}^k\overline{0}\big)$ for all $k\geq0$. First of all, we explained in Section~\ref{sec:IBMandRD} that, for all $k\geq0$
    \begin{align*}
        \Big|T^B\big(\overline{1}^k\overline{0}\big)\Big|&=\tau^B\big(\overline{1}^k\big)-\tau^B\big(\overline{1}^{k-1}\big)-1\,.
    \end{align*}
    Moreover, the random variables $\big(\tau^B\big(\overline{1}^k\big)-\tau^B\big(\overline{1}^{k-1}\big)-1\big)_{k\geq0}$ are independent and $\textsc{Geometric}(1-q)$-distributed. Finally, by Proposition~\ref{prop:ProjCons}, conditionally on having a given size $n$, $T^B(\overline{1}^k\overline{0})$ is $\mt(n,q)$-distributed, from which the result follows.
\end{proof}

\begin{proof}[Proof of Proposition~\ref{prop:BivariateGenFun}]
    Fix $q\in[0,1)$ and let $B=\big(B_{i,j}\big)_{i,j\geq1}$ have independent $\textsc{Bernoulli}(1-q)$ entries. Let us first study the transition probabilities of the random process $(R^B_n,M^B_n)_{n\geq0}$. Define $\mathcal{L}_n=\sigma\big(B_{i,j},1\leq i\leq n,j\geq1\big)$ for the $\sigma$-algebra generated by the first $n$ rows of the matrix $B$ and note that $R^B_n$ and $M^B_n$ are $\mathcal{L}_n$-measurable. Since $(R^B_n)_{n\geq0}$ corresponds to the number of records of the sequence $(M^B_n)_{n\geq0}$, we have
    \begin{align}
        &\mathbb{P}\Big(R^B_{n+1}=r,M^B_{n+1}=m\,\Big|\,R^B_n=r,M^B_n=m,\mathcal{L}_n\Big)\notag\\
        &\hspace{1cm}=\mathbb{P}\Big(f^B(n+1)<M^B_n\,\Big|\,R^B_n=r,M^B_n=m,\mathcal{L}_n\Big)\notag\\
        &\hspace{1cm}=\mathbb{P}\Big(\exists j\in[M^B_n-1]\setminus F^B_n,B_{n+1,j}=1\,\Big|\,R^B_n=r,M^B_n=m,\mathcal{L}_n\Big)\notag\\
        &\hspace{1cm}=1-\mathbb{P}\Big(\forall j\in[M^B_n-1]\setminus F^B_n,B_{n+1,j}=0\,\Big|\,R^B_n=r,M^B_n=m,\mathcal{L}_n\Big)\,.\label{eq:recursiveRM}
    \end{align}
    Moreover, since $R^B_n$, $M^B_n$ and $F_n^B$ are $\mathcal{L}_n$-measurable, and $\big(B_{n+1,j}\big)_{j\geq1}$ is independent of $\mathcal{L}_n$, we have
    \begin{align}
        &\mathbb{P}\Big(\forall j\in[M^B_n-1]\setminus F^B_n,B_{n+1,j}=0\,\Big|\,R^B_n=r,M^B_n=m,\mathcal{L}_n\Big)\notag\\
        &\hspace{0.5cm}=\mathbb{E}\Big[q^{\left|[M^B_n-1]\setminus F^B_n\right|}\,\Big|\,R^B_n=r,M^B_n=m,\mathcal{L}_n\Big]\notag\\
        &\hspace{0.5cm}=q^{m-n}\,;\label{eq:probNoRecord}
    \end{align}
    the final equality holding since
    \begin{align*}
        \big|[M^B_n-1]\setminus F^B_n\big|&=\big|[M^B_n-1]\setminus\{f^B(1),...,f^B(n)\}\big|\\
        &=\big|[M^B_n]\setminus\{f^B(1),...,f^B(n)\}\big|\\
        &=M^B_n-n\,.
    \end{align*}
    Combining (\ref{eq:recursiveRM}) and (\ref{eq:probNoRecord}) shows that the desired transition probability has the claimed value when $k=0$ and $\ell=0$. Similarly, for $\ell\geq1$, we have
    \begin{align*}
        &\mathbb{P}\Big(R^B_{n+1}=r+1,M^B_{n+1}=m+\ell\,\Big|\,R^B_n=r,M^B_n=m,\mathcal{L}_n\Big)\\
        &\hspace{0.5cm}=\mathbb{P}\Big(f^B(n+1)=m+\ell\,\Big|\,R^B_n=r,M^B_n=m,\mathcal{L}_n\Big)\\
        &\hspace{0.5cm}=\mathbb{P}\Big(\forall j\in[M^B_n+\ell-1]\setminus F^B_n,B_{n+1,j}=0\textrm{ and }B_{n+1,M^B_n+\ell}=1\,\Big|\,R^B_n=r,M^B_n=m,\mathcal{L}_n\Big)\\
        &\hspace{0.5cm}=q^{m+\ell-n-1}(1-q)\,.
    \end{align*}
    This establishes the Markov property and proves that the transition probability is as claimed.
    
    We now move on to the derivation of the bivariate characteristic function, and we do so by induction. The equation does hold when $n=0$, as can be straightforwardly checked, but for the sake of the proof, it is more natural to start the induction at $n=1$. In this case, we have
    \begin{align*}
        \mathbb{E}\left[x^{R^B_1+1}y^{M^B_1}\right]&=\sum_{j\geq1}\mathbb{E}\left[x^{R^B_1+1}y^{M^B_1}\,\Big|\, f^B(1)=j\right]\mathbb{P}\left(f^B(1)=j\right)\\
        &=\sum_{j\geq1}x^1y^jq^{j-1}(1-q)\\
        &=\frac{xy(1-q)}{1-qy}\\
        &=y^1\frac{q+(1-q)x-q^1}{1-q^1y},
    \end{align*}
    which is the desired formula. Assume now that the formula holds for some $n\geq1$. First, by conditioning on $(R^B_n,M^B_n)$, we have
    \begin{align}
        \mathbb{E}\left[x^{R^B_{n+1}+1}y^{M^B_{n+1}}\right]&=\sum_{r,m}\mathbb{E}\left[x^{R^B_{n+1}+1}y^{M^B_{n+1}}\,\Big|\,R^B_n=r,M^B_n=m\right]\mathbb{P}\left(R^B_n=r,M^B_n=m\right)\,.\label{eq:lem312}
    \end{align}
    For the first term inside the sum, we have
    \begin{align*}
        &\mathbb{E}\left[x^{R^B_{n+1}+1}y^{M^B_{n+1}}\,\Big|\,R^B_n=r,M^B_n=m\right]\\
        &\hspace{1cm}=\sum_{k\in\{0,1\}}\sum_{\ell\geq0}x^{r+k+1}y^{m+\ell}\mathbb{P}\left(R^B_{n+1}=r+k,M^B_{n+1}=m+\ell\,\Big|\,R^B_n=r,M^B_n=m\right)\,.
    \end{align*}
    We now split the sum according to whether $\ell=0$ or $\ell\geq1$. Note that if $\ell=0$, then $k=0$, and if $\ell\geq1$, then $k=1$. Using the previously derived transition probabilities, we obtain
    \begin{align*}
        \mathbb{E}\left[x^{R^B_{n+1}+1}y^{M^B_{n+1}}\,\Big|\,R^B_n=r,M^B_n=m\right]&=x^{r+1}y^m\left(1-q^{m-n}\right)+\sum_{\ell\geq1}x^{r+2}y^{m+\ell}q^{m+\ell-n-1}(1-q)\\
        &=x^{r+1}y^m\left(1-q^{m-n}\right)+\frac{x^{r+2}y^{m+1}q^{m-n}(1-q)}{1-qy}\,.
    \end{align*}
    Plugging this back into (\ref{eq:lem312}), we have
    \begin{align*}
        \mathbb{E}\left[x^{R^B_{n+1}+1}y^{M^B_{n+1}}\right]&=\sum_{r,m}\left[x^{r+1}y^m\left(1-q^{m-n}\right)+\frac{x^{r+2}y^{m+1}q^{m-n}(1-q)}{1-qy}\right]\mathbb{P}\left(R^B_n=r,M^B_n=m\right)\\
        &=\mathbb{E}\left[x^{R^B_n+1}y^{M^B_n}\right]-\frac{1}{q^n}\mathbb{E}\left[x^{R^B_n+1}(qy)^{M^B_n}\right]+\frac{xy(1-q)}{q^n(1-qy)}\mathbb{E}\left[x^{R^B_n+1}(qy)^{M^B_n}\right]\\
        &=\mathbb{E}\left[x^{R^B_n+1}y^{M^B_n}\right]+\frac{1}{q^n}\left[\frac{xy(1-q)}{1-qy}-1\right]\mathbb{E}\left[x^{R^B_n+1}(qy)^{M^B_n}\right]\,.
    \end{align*}
    Using the induction hypothesis, this yields the equation
    \begin{align*}
        &\mathbb{E}\left[x^{R^B_{n+1}+1}y^{M^B_{n+1}}\right]\\
        &\hspace{1cm}=\left[y^n\prod_{1\leq k\leq n}\frac{q+(1-q)x-q^k}{1-q^ky}\right]+\frac{1}{q^n}\left[\frac{xy(1-q)}{1-qy}-1\right]\left[(qy)^n\prod_{1\leq k\leq n}\frac{q+(1-q)x-q^k}{1-q^k(qy)}\right]\\
        &\hspace{1cm}=y^n\prod_{1\leq k\leq n}\frac{q+(1-q)x-q^k}{1-q^ky}+\frac{(xy(1-q)-1+qy)y^n}{1-qy}\prod_{1\leq k\leq n}\frac{q+(1-q)x-q^k}{1-q^{k+1}y}\,.
    \end{align*}
    Finally, factorizing by $y^n\prod_{1\leq k\leq n}\frac{q+(1-q)x-q^k}{1-q^ky}$, it follows that
    \begin{align*}
        \mathbb{E}\left[x^{R^B_{n+1}+1}y^{M^B_{n+1}}\right]&=\left[y^n\prod_{1\leq k\leq n}\frac{q+(1-q)x-q^k}{1-q^ky}\right]\left[1+\frac{xy(1-q)-1+qy}{1-q^{n+1}y}\right]\\
        &=\left[y^n\prod_{1\leq k\leq n}\frac{q+(1-q)x-q^k}{1-q^ky}\right]\left[\frac{y\left(q+(1-q)x-q^{n+1}\right)}{1-q^{n+1}y}\right]\\
        &=y^{n+1}\prod_{1\leq k\leq n+1}\frac{q+(1-q)x-q^k}{1-q^ky},
    \end{align*}
    establishing the second assertion of the proposition.
\end{proof}

With these two results, we now have strong properties regarding Mallows trees and we can prove bounds on their height in the case when $n(1-q_n)/\log n\rightarrow\infty$.

\subsection{Bounds on the left subtrees}\label{sec:bdsLeftSubtrees}

Before proving Theorem~\ref{thm:AS}, we prove Proposition~\ref{prop:boundsLeftSubtree}, whose bounds will be useful in in the proof of all three theorems of this paper.

\begin{proof}[Proof of Proposition~\ref{prop:boundsLeftSubtree}]
    Before starting the actual proof, let us briefly explain why
    \begin{align}
        M:=c^*\times\sup_{n\geq2,q\in(0,1)}\left\{\frac{\mu_1(n,q)-n(1-q)-\log\left(n\wedge\frac{1}{1-q}\right)}{\sqrt{\log\left(n\wedge\frac{1}{1-q}\right)}}\right\}\,.\label{eq:boundsForM}
    \end{align}
    is finite. Assuming it is not, we can find a sequence $(n_k,q_k)_{k\geq0}$ such that the term inside the supremum of (\ref{eq:boundsForM}) goes to infinity. By extracting subsequences, we can assume that $(q_k)_{k\geq0}$ converges in $[0,1]$ and that $(n_k)_{k\geq0}$ is constant equal to some $n\geq2$ or diverges to $\infty$.
    
    If $n_k=n$, the numerator is bounded and the only way for this term to diverge is to have $\log\big(n\wedge\frac{1}{1-q_k}\big)\rightarrow0$, hence $q_k\rightarrow0$. In this case, we have
    \begin{align*}
        \mu_1(n_k,q_k)-n_k(1-q_k)-\log\left(n_k\wedge\frac{1}{1-q_k}\right)&=\sum_{1\leq i\leq n}\frac{1-q_k}{1-q_k^i}-n(1-q_k)-\log\left(\frac{1}{1-q_k}\right)\\
        &=O(q_k)\,,
    \end{align*}
    and since $\log\big(n_k\wedge\frac{1}{1-q_k}\big)\sim q_k$, the term inside the supremum of (\ref{eq:boundsForM}) actually converges to $0$ and not infinity. On the other hand, if $n_k\rightarrow\infty$, we can apply Proposition~\ref{prop:ConvOfMu} to bound the numerator in (\ref{eq:boundsForM}) and see that the supremum is again finite.
    
    We will prove the bound stated in the proposition for this value of $M$. By the definition of $M$, for all $n\geq2$ and $q\in(0,1)$, we have
    \begin{align*}
        \mu_1(n,q)-n(1-q)-\log\left(n\wedge\frac{1}{1-q}\right)\leq\frac{M}{c^*}\sqrt{\log\left(n\wedge\frac{1}{1-q}\right)}\,;
    \end{align*}
    moreover, this inequality remains true when $n=1$ or $q=0$ as both sides equal $0$.
    
    Now, without loss of generality, fix $\xi>0$, write
    \begin{align*}
        \mathbb{P}_{\textrm{LS}}:=\mathbb{P}\left(\sup_{k\geq0}\Big\{h\Big(T^B\big(\overline{1}^{k}\overline{0}\big)-k\Big)\Big\}\geq c^*\log\left(\frac{1}{1-q}\right)+M\sqrt{\log\left(\frac{1}{1-q}\right)}+\xi\right)\,,
    \end{align*}
    and let $\ell=\left\lceil c^*\log\left(\frac{1}{1-q}\right)+M\sqrt{\log\left(\frac{1}{1-q}\right)}+\xi\right\rceil$. By taking a union bound over $k$, we obtain
    \begin{align}
        \mathbb{P}_{\textrm{LS}}&\leq\sum_{k\geq0}\mathbb{P}\left(h\Big(T^B\big(\overline{1}^{k}\overline{0}\big)\Big)\geq\ell+k\right)\,.\label{eq:boundLeft1}
    \end{align}
    
    By Lemma~\ref{lem:LeftSubtreesDistribution}, $T^B(\overline{1}^k\overline{0})\overset{d}{=}T_{G(q),q}$, so
    \begin{align*}
        \mathbb{P}\left(h\Big(T^B\big(\overline{1}^{k}\overline{0}\big)\Big)\geq\ell+k\right)&=\mathbb{P}\Big(h\big(T_{G(q),q}\big)\geq\ell+k\Big)\,.
    \end{align*}
    Consider now some $\lambda>0$ and divide the probability according to whether $G(q)\geq\frac{\lambda(k+\xi)+2-q}{1-q}$ or $G(q)<\frac{\lambda(k+\xi)+2-q}{1-q}$, to obtain
    \begin{align}
        \mathbb{P}\left(h\Big(T^B\big(\overline{1}^{k}\overline{0}\big)\Big)\geq\ell+k\right)&=\mathbb{P}\left(h\big(T_{G(q),q}\big)\geq\ell+k,G(q)\geq\frac{\lambda(k+\xi)+2-q}{1-q}\right)\notag\\
        &\hspace{0.5cm}+\mathbb{P}\left(h\big(T_{G(q),q}\big)\geq\ell+k,G(q)<\frac{\lambda(k+\xi)+2-q}{1-q}\right)\,.\label{eq:bdsLeftSubtreeGq}
    \end{align}
    For the first term, drop the first event and use that $G(q)$ is geometric to obtain
    \begin{align*}
        \mathbb{P}\Big(h\big(T_{G(q),q}\big)\geq\ell+k,G(q)\geq\frac{\lambda(k+\xi)+2-q}{1-q}\Big)&\leq q^{\frac{\lambda(k+\xi)+2-q}{1-q}}\leq e^{-\lambda(k+\xi)-1}\,,
    \end{align*}
    where the last inequality follows from the fact that $0\leq q\leq e^{q-1}$. For the second term, use the increasing property of $(T_{n,q})_{n\geq0}$ from Corollary~\ref{cor:TnqIncreasing} to bound $G(q)$ by its maximal value, and then drop the second event, to obtain
    \begin{align*}
        \mathbb{P}\left(h\big(T_{G(q),q}\big)\geq\ell+k,G(q)<\frac{\lambda(k+\xi)+2-q}{1-q}\right)\leq\mathbb{P}\Big(h\big(T_{m,q}\big)\geq\ell+k\Big)\,,
    \end{align*}
    where $m=\left\lfloor\frac{\lambda(k+\xi)+2-q}{1-q}\right\rfloor$. Putting this back into (\ref{eq:bdsLeftSubtreeGq}) gives us
    \begin{align}
         \mathbb{P}\left(h\Big(T^B\big(\overline{1}^{k}\overline{0}\big)\Big)\geq\ell+k\right)&\leq e^{-\lambda(k+\xi)-1}+\mathbb{P}\Big(h\big(T_{m,q}\big)\geq\ell+k\Big)\,.\label{eq:boundLeft2}
    \end{align}
    
    In order to bound $\mathbb{P}_{\textrm{LS}}$, we now need to bound $\mathbb{P}\big(h\big(T_{m,q}\big)\geq\ell+k\big)$. Taking a union bound over all nodes at depth $\ell+k$ and then applying Proposition~\ref{prop:StochBounds}, we obtain
    \begin{align*}
        \mathbb{P}\Big(h\big(T_{m,q}\big)\geq\ell+k\Big)&\leq2^{\ell+k}\mathbb{P}\Big(R^B_m\geq\ell+k\Big)\,.
    \end{align*}
    Using Chernoff's bound with the moment generating function of $R^B_m$ from Proposition~\ref{prop:BivariateGenFun}, it follows that, for any $t>0$
    \begin{align*}
        \mathbb{P}\Big(R^B_m\geq\ell+k\Big)&\leq e^{-t(\ell+k)}\mathbb{E}\left[e^{tR^B_m}\right]\leq\exp\Big(-t(\ell+k)+(e^t-1)\mu_1(m,q)\Big)\,.
    \end{align*}
    Putting this back into the previous inequality and taking $t=\log c^*$ gives us
    \begin{align*}
        \mathbb{P}\Big(h\big(T_{m,q}\big)\geq\ell+k\Big)&\leq\exp\left(\log\left(\frac{2}{c^*}\right)(\ell+k)+(c^*-1)\mu_1(m,q)\right)\,.
    \end{align*}
    Since $m\geq\frac{1}{1-q}$ and by using the definition of $M$, we know that
    \begin{align*}
        \mu_1(m,q)\leq m(1-q)+\log\left(\frac{1}{1-q}\right)+\frac{M}{c^*}\sqrt{\log\left(\frac{1}{1-q}\right)}\,.
    \end{align*}
    From this inequality, we obtain
    \begin{align*}
        &\mathbb{P}\Big(h\big(T_{m,q}\big)\geq\ell+k\Big)\\
        &\hspace{1cm}\leq\exp\left(\log\left(\frac{2}{c^*}\right)(\ell+k)+(c^*-1)\left(m(1-q)+\log\left(\frac{1}{1-q}\right)+\frac{M}{c^*}\sqrt{\log\left(\frac{1}{1-q}\right)}\right)\right)\,.
    \end{align*}
  From their respective definitions, we know that $\ell\geq c^*\log\left(\frac{1}{1-q}\right)+M\sqrt{\log\left(\frac{1}{1-q}\right)}+\xi$ and that $m\leq\frac{\lambda(k+\xi)+2-q}{1-q}$. Using these bounds together with the fact that $\log\left(\frac{2}{c^*}\right)=\frac{1}{c^*}-1<0$ in the preceding inequality, we obtain
    \begin{align*}
         \mathbb{P}\Big(h\big(T_{m,q}\big)\geq\ell+k\Big)&\leq\exp\left(-\frac{c^*-1}{c^*}\left(c^*\log\left(\frac{1}{1-q}\right)+M\sqrt{\log\left(\frac{1}{1-q}\right)}+\xi+k\right)\right.\\
         &\hspace{0.5cm}\left.+(c^*-1)\left(\lambda(k+\xi)+2-q+\log\left(\frac{1}{1-q}\right)+\frac{M}{c^*}\sqrt{\log\left(\frac{1}{1-q}\right)}\right)\right)\,\\
         &\leq\exp\left(-\frac{(c^*-1)(1-c^*\lambda)}{c^*}(\xi+k)+2(c^*-1)\right)\,.
    \end{align*}
    
    Using this bound in (\ref{eq:boundLeft2}), and then plugging the result into (\ref{eq:boundLeft1}), we obtain
    \begin{align*}
        \mathbb{P}_{\textrm{LS}}&\leq\sum_{k\geq0}\left[\exp\Big(-\lambda(k+\xi)-1\Big)+\exp\left(-\frac{(c^*-1)(1-c^*\lambda)}{c^*}(\xi+k)+2(c^*-1)\right)\right]\,.
    \end{align*}
    Choosing $\lambda=\frac{c^*-1}{(c^*)^2}$ so that $\lambda=\frac{(c^*-1)(1-c^*\lambda)}{c^*}$, this bound becomes
    \begin{align*}
        \mathbb{P}_{\textrm{LS}}&\leq e^{-\lambda\xi}\sum_{k\geq0}\Big[\exp\big(-\lambda k-1\big)+\exp\big(-\lambda k+2(c^*-1)\big)\Big]=\left(\frac{e^{-1}+e^{2(c^*-1)}}{1-e^{-\lambda}}\right)e^{-\lambda\xi}\,,
    \end{align*}
    proving the proposition.
\end{proof}

\subsection{Almost sure convergence}

To conclude this section, we prove Theorem~\ref{thm:AS}, from which we deduce Theorem~\ref{thm:Prob} in the case when $n(1-q_n)/\log n\rightarrow\infty$ by combining it with Proposition~\ref{prop:UI}.

\begin{proof}[Proof of Theorem~\ref{thm:AS}]
    Let $(q_n)_{n\geq0}$ be a sequence such that $n(1-q_n)/\log n\rightarrow\infty$ and fix $\varepsilon>0$ and $\lambda>0$. We prove that
    \begin{align*}
        \mathbb{P}\left(\left|\frac{h(T_{n,q_n})}{n(1-q_n)}-1\right|>\varepsilon\right)=O\left(\frac{1}{n^\lambda}\right)
    \end{align*}
    by bounding the lower tail and the upper tail separately.
    
    We start with the lower tail as the technique for the upper tail bound is similar but more involved. Using $R^B_n$ as a stochastic lower bound for $h(T_{n,q_n})$ as in (\ref{eq:htnq_upper}), we have
    \begin{align*}
        \mathbb{P}\Big(h\big(T_{n,q_n}\big)<(1-\varepsilon)n(1-q_n)\Big)&\leq\mathbb{P}\Big(R^B_n<(1-\varepsilon)n(1-q_n)\Big)\,.
    \end{align*}
    Since $n(1-q_n)=\omega(\log n)$, by applying Proposition~\ref{prop:ConvOfMuAlpha} we know that $\mu_1(n,q_n)\sim n(1-q_n)$. It follows that, for $n$ large enough, we have $(1-\varepsilon)n(1-q_n)/\mu_1(n,q_n)<1$ and we can apply the second bound of Lemma~\ref{lem:ConcBoundRn} to the previous inequality to obtain
    \begin{align*}
        \mathbb{P}\Big(h\big(T_{n,q_n}\big)<(1-\varepsilon)n(1-q_n)\Big)&\leq\exp\left(\left[\frac{(1-\varepsilon)n(1-q_n)}{\mu_1(n,q_n)}\log\left(\frac{e\mu_1(n,q_n)}{(1-\varepsilon)n(1-q_n)}\right)-1\right]\mu_1(n,q_n)\right)\\
        &=\exp\left(\left[(1-\varepsilon)\log\left(\frac{e}{1-\varepsilon}\right)-1+o(1)\right]n(1-q_n)\right)\,.
    \end{align*}
    By convexity, $(1-\varepsilon)\log\left(\frac{e}{1-\varepsilon}\right)<1$. Since $n(1-q_n)=\omega(\log n)$, this proves that
    \begin{align*}
        \mathbb{P}\Big(h\big(T_{n,q_n}\big)<(1-\varepsilon)n(1-q_n)\Big)&=O\left(\frac{1}{n^\gamma}\right)\,,
    \end{align*}
    which is the desired lower bound for $h(T_{n,q_n})$.
    
    Let us now prove that the upper bound also holds. Using the second stochastic inequality given in (\ref{eq:htnq_upper}), we have
    \begin{align*}
        \mathbb{P}\Big(h\big(T_{n,q_n}\big)>(1+\varepsilon)n(1-q_n)\Big)&\leq\mathbb{P}\left(R_n^B+1+\sup_{k \ge 0} \Big\{h\Big(T^B\big(\overline{1}^{k}\overline{0}\big)\Big)-k\Big\}>(1+\varepsilon)n(1-q_n)\right)\,.
    \end{align*}
    Using that $X+Y>x+y$ implies that $X>x$ or $Y>y$ to bound the probability on the right, we obtain
    \begin{align*}
        \mathbb{P}\Big(h\big(T_{n,q_n}\big)>(1+\varepsilon)n(1-q_n)\Big)&\leq\mathbb{P}\Big(R_n^B>(1+\varepsilon/2)n(1-q_n)\Big)\\
        &\hspace{0.5cm}+\mathbb{P}\left(1+\sup_{k \ge 0} \Big\{h\Big(T^B\big(\overline{1}^{k}\overline{0}\big)\Big)-k\Big\}>\varepsilon n(1-q_n)/2\right)\,.
    \end{align*}
    For the first term, using the first bound of Lemma~\ref{lem:ConcBoundRn} and the same arguments as for the lower bound, it follows that
    \begin{align*}
        \mathbb{P}\Big(R_n^B>(1+\varepsilon/2)n(1-q_n)\Big)=O\left(\frac{1}{n^\gamma}\right)\,.
    \end{align*}
    For the second term, apply Proposition~\ref{prop:boundsLeftSubtree} with \begin{align*}
        \xi=\xi_n=\varepsilon n(1-q_n)/2-1-c^*\log\left(\frac{1}{1-q_n}\right)-M\sqrt{\log\left(\frac{1}{1-q_n}\right)}
    \end{align*}
    to obtain
    \begin{align*}
        \mathbb{P}\left(1+\sup_{k \ge 0} \Big\{h\Big(T^B\big(\overline{1}^{k}\overline{0}\big)\Big)-k\Big\}>\varepsilon n(1-q_n)/2\right)\leq Ce^{-\lambda\xi_n}\,.
    \end{align*}
    Since $n(1-q_n)/\log n\rightarrow\infty$, it follows that $\log\frac{1}{1-q_n}=O(\log n)$ and then $\xi_n\sim\varepsilon n(1-q_n)/2$ which proves that
    \begin{align*}
        \mathbb{P}\left(1+\sup_{k \ge 0} \Big\{h\Big(T^B\big(\overline{1}^{k}\overline{0}\big)\Big)-k\Big\}>\varepsilon n(1-q_n)/2\right)=O\left(\frac{1}{n^\gamma}\right)\,.
    \end{align*}
    This yields the desired upper bound for $h(T_{n,q_n})$ and concludes the proof of Theorem~\ref{thm:AS}.
\end{proof}

\section{Intermediate values and threshold process}\label{sec:IntValuesThreshold}

In this section, we will prove Proposition~\ref{prop:convRightSubtree}, which will allow us to conclude the proof of Theorem~\ref{thm:Prob} in the last case, i.e. when $n(1-q_n)/\log n=\Theta(1)$. In order to do so, we prove the following proposition, which in fact handles a somewhat wider range of asymptotic behaviour for the sequence $(q_n)_{n\geq0}$. The bounds in Proposition~\ref{prop:BoundsOnN}, below, are actually tight enough that they will also be used in Section~\ref{sec:CLTRD} to prove Proposition~\ref{prop:convRenSubtree}, which is a key input of the central limit theorem for the height of Mallows trees.

\begin{prop}\label{prop:BoundsOnN}
    Let $(q_n)_{n\geq0}$ be such that $\log\big(n(1-q_n)\big)=O\left(\sqrt{\log n}\right)$ and $n(1-q_n)=\omega\left(\sqrt{\log n}\right)$. For $n\geq0$, let $m=m(n)=\min\big\{\ell\geq0:\ell(1-q_n)+\log\ell\geq n(1-q_n)\big\}$. Then, for all $(\beta_n)_{n\geq0}$ such that $\beta_n=\omega\left(\sqrt{\log n}\right)$, we have
    \begin{align*}
        \lim_{n\rightarrow\infty}\mathbb{P}\bigg(e^{-\beta_n}\leq\Big|T^B_n\big(\overline{1}^{R^B_m+1}\big)\Big|(1-q_n)\leq\beta_n\bigg)=1\,.
    \end{align*}
\end{prop}

From this proposition, with $m=m(n)$ as previously defined, we will show that
\begin{align*}
    \Big|T^B_n\big(\overline{1}^{R^B_m+1}\big)\Big|(1-q_n)=o\bigg(\log\Big(\Big|T^B_n\big(\overline{1}^{R^B_m+1}\big)\Big|\Big)\bigg)
\end{align*}
which allows us to apply Proposition~\ref{prop:convProbSmall} to prove that
\begin{align*}
    h\Big(T^B_n\big(\overline{1}^{R^B_m+1}\big)\Big)=\big(c^*+o_\mathbb{P}(1)\big)\log n\,.
\end{align*}
It is not hard to show that $R^B_m=\big(1+o_\mathbb{P}(1)\big)n(1-q_n)$. Once we establish this, we will be able to prove Proposition~\ref{prop:convRightSubtree} by comparing the sizes of $T^B_n\big(\overline{1}^{R^B_m+1}\big)$ and $T^B_n\big(\overline{1}^{\lfloor n(1-q_n)\rfloor}\big)$.

In order to study the size of $T^B_n\big(\overline{1}^{R^B_m+1}\big)$, recall the definition of $(T^B_n)_{n\geq0}$ from Section~\ref{sec:IBMandRD} and note that, for all $d\geq0$
\begin{align*}
    \Big|T^B_n\big(\overline{1}^{d+1}\big)\Big|=\Big|\Big\{i\in[n]:f^B(i)>\tau^B\big(\overline{1}^d\big)\Big\}\Big|\,,
\end{align*}
where $\tau^B$ is the labelling function of the tree $T^B$. Moreover, since $\tau^B\big(\overline{1}^{R^B_m}\big)=M^B_m$ where $M^B_m=\max\big(f^B(i),i\in[m]\big)$, we have
\begin{align*}
    \Big|T^B_n\big(\overline{1}^{R^B_m+1}\big)\Big|=\Big|\Big\{i\in[n]:f^B(i)>M^B_m\Big\}\Big|\,.
\end{align*}

In order to study the size of the random set on the right hand side of this equation, we define the \textit{threshold process} $N^B$ as follows. For all $n\in\N$ and $s\in\N$, let
\begin{align}
    N^B(n,s)&:=\Big|\Big\{i\in[n]:f^B(i)>s\Big\}\Big|\,.\label{eq:DeftNB}
\end{align}
The two preceding displays show that
\begin{align*}
    \Big|T^B_n\big(\overline{1}^{R^B_m+1}\big)\Big|=N^B\big(n,M^B_m\big)\,,
\end{align*}
which will be our key tool for bounding $\big|T^B_n\big(\overline{1}^{R^B_m+1}\big)\big|$. The analysis of this identity is made easier by the following lemma, which partially decouples $M^B$ and $N^B$.

\begin{lemma}\label{lem:N'andN}
    For all $n\geq0$, $q\in[0,1)$ and $0\leq m\leq n$, we have
    \begin{align*}
        \Big|T^B_n\big(\overline{1}^{R^B_m+1}\big)\Big|\overset{d}{=}N^{B^*}\big(n-m,M^B_m-m\big)\,,
    \end{align*}
    where $B=(B_{i,j})$ and $B^*=(B^*_{i,j})$ are independent matrices with independent $\textsc{Bernoulli}(1-q)$ entries.
\end{lemma}

Using this distributional identity, the proof of Proposition~\ref{prop:BoundsOnN} will be divided into the following three steps:
\begin{description}
    \item[Step 1] We study the threshold process $N^B(n,s)$ for all values of $n$ and $s$ and prove bounds for its upper and lower tail probabilities (see Proposition~\ref{prop:UBN} and \ref{prop:LBN}).
    \item[Step 2] We prove that $M^B_{m(n)}=n+O_\mathbb{P}\big(\sqrt{\log n}/(1-q_n)\big)$ using Proposition~\ref{prop:BivariateGenFun}.
    \item[Step 3] We combine these two results to prove Proposition~\ref{prop:BoundsOnN}.
\end{description}

The rest of the section is organized as follows. The results of step~1, 2 and 3 are respectively stated and proven in Section~\ref{sec:ThresholdProcess}, \ref{sec:BoundsM} and \ref{sec:BoundsThresholdSubtree}. Finally, in Section~\ref{sec:ThmProb}, we prove Proposition~\ref{prop:convRightSubtree} and deduce Theorem~\ref{thm:Prob} in the case when $n(1-q_n)/\log n=\Theta(1)$.

\subsection{Threshold process}\label{sec:ThresholdProcess}

This section is focused on the behaviour of $N^B(n,s)$ as defined in (\ref{eq:DeftNB}). From the definition, we see that $N^B$ is increasing in $n$ and decreasing in $s$. Since $f^B$ is bijective, it is also straightforward to verify that $(n-s)_+\leq N(n,s)\leq n$ where $x_+=\max(x,0)$. The next proposition gives further properties related to the distribution of $N$.

\begin{prop}\label{prop:RecursiveN}
    Let $q\in[0,1)$ and $B=(B_{i,j})_{i,j\geq0}$ have independent $\textsc{Bernoulli}(1-q)$ entries. For $k\geq0$, write $\mathcal{L}_k(B)=\sigma(B_{i,j},1\leq i\leq k,j\geq1)$ for the $\sigma$-algebra generated by the first $k$ rows of $B$. Then, for any integers $n,s\geq1$, for all $\ell\geq0$, we have
    \begin{align*}
        \mathbb{P}\Big(N^B(n,s)=\ell\,\Big|\,f^B(1)\leq s,\mathcal{L}_k(B)\Big)&\overset{d}{=}\mathbb{P}\Big(N^B(n-1,s-1)=\ell\,\Big|\,\mathcal{L}_{k-1}(B)\Big)
    \end{align*}
    and
    \begin{align*}
        \mathbb{P}\Big(N^B(n,s)=\ell\,\Big|\,f^B(1)>s,\mathcal{L}_k(B)\Big)&\overset{d}{=}\mathbb{P}\Big(N^B(n-1,s)=\ell-1\,\Big|\,\mathcal{L}_{k-1}(B)\Big)\,.
    \end{align*}
\end{prop}

\begin{proof}
    Write $B^{(i,j)}$ for the minor of $B$ obtained by deleting the $i$-th row and the $j$-th column, and note that $B^{(i,j)}\overset{d}{=}B$. Moreover, given that $f^B(1)=r$, the rest of the values $f^B(2),f^B(3),\ldots$ becomes independent of the first row and the $r$-th column of $B$. Hence, for $1\leq r\leq s$, we have
    \begin{align*}
        \mathbb{P}\Big(N^B(n,s)=\ell\,\Big|\,f^B(1)=r,\mathcal{L}_k(B)\Big)&=\mathbb{P}\Big(N^{B^{(1,r)}}(n-1,s-1)=\ell\,\Big|\,\mathcal{L}_k(B)\Big)\\
        &\overset{d}{=}\mathbb{P}\Big(N^B(n-1,s-1)=\ell\,\Big|\,\mathcal{L}_{k-1}(B)\Big)\,,
    \end{align*}
    where the second equality holds since $\mathcal{L}_k(B)$ is generated by $\mathcal{L}_{k-1}(B^{(1,r)})$ and $\sigma\big(\{B_{1,j},j\geq1\}\cup\{B_{i,r},i\geq1\}\big)$, and $B^{(1,r)}$ is independent of the second of these $\sigma$-algebras. Similarly, for $r>s$, we have
    \begin{align*}
        \mathbb{P}\Big(N^B(n,s)=\ell\,\Big|\,f^B(1)=r,\mathcal{L}_k(B)\Big)&=\mathbb{P}\Big(N^{B^{(1,r)}}(n-1,s)=\ell-1\,\Big|\,\mathcal{L}_k(B)\Big)\\
        &\overset{d}{=}\mathbb{P}\Big(N^B(n-1,s)=\ell-1\,\Big|\,\mathcal{L}_{k-1}(B)\Big)\,.
    \end{align*}
    This proves the two desired equalities.
\end{proof}

Applying this proposition, we can now prove Lemma~\ref{lem:N'andN}.

\begin{proof}[Proof of Lemma~\ref{lem:N'andN}]
    By definition, we know that
    \begin{align*}
        \Big|T^B_n\big(\overline{1}^{R^B_m+1}\big)\Big|=N^B\big(n,M^B_m\big)\,.
    \end{align*}
    Let $s\geq1$ be an integer. Conditioning on the value of $M^B_m$ and applying Proposition~\ref{prop:RecursiveN} with $k=1$, we obtain
    \begin{align*}
        \mathbb{P}\Big(N^B(n,M^B_m)=\ell\,\Big|\,M^B_m=s\Big)&=\mathbb{P}\Big(N^B(n,s)=\ell\,\Big|\,M^B_m=s,f^B(1)\leq s\Big)\\
        &=\mathbb{E}\left[\mathbb{P}\Big(N^B(n,s)=\ell\,\Big|\,M^B_m=s,f^B(1)\leq s,\mathcal{L}_1\Big)\right]\\
        &=\mathbb{P}\Big(N^B(n-1,s-1)=\ell\,\Big|\,M^B_{m-1}=s-1\Big)\,.
    \end{align*}
    Applying this identity $m-1$ times, we obtain that
    \begin{align*}
        \mathbb{P}\Big(N^B(n,M^B_m)=\ell\,\Big|\,M^B_m=s\Big)&=\mathbb{P}\Big(N^B(n-m+1,s-m+1)=\ell\,\Big|\,M^B_1=s-m+1\Big)\,.
    \end{align*}
    For the last step, since $\{f^B(1)\leq s-m+1,M^B_1=s-m+1\}=\{f^B(1)=s-m+1\}$, it follows that
    \begin{align*}
        &\mathbb{P}\Big(N^B(n,M^B_m)=\ell\,\Big|\,M^B_m=s\Big)\\
        &\hspace{1cm}=\mathbb{P}\Big(N^B(n-m+1,s-m+1)=\ell\,\Big|\, M^B_1=s-m+1,f^B(1)\leq s-m+1\Big)\\
        &\hspace{1cm}=\mathbb{P}\Big(N^B(n-m+1,s-m+1)=\ell\,\Big|\,f^B(1)=s-m+1\Big)\\
        &\hspace{1cm}=\mathbb{P}\Big(N^B(n-m,s-m)=\ell\Big)\,.
    \end{align*}
    Thus, the proof of the lemma is immediate by observing that
    \begin{align*}
        \mathbb{P}\Big(N^{B^*}(n-m,M^B_m-m)=\ell\,\Big|\,M^B_m=s\Big)&=\mathbb{P}\Big(N^{B^*}(n-m,s-m)=\ell\Big)\,.\qedhere
    \end{align*}
\end{proof}

In order to bound the size of $T^B_n\big(\overline{1}^{R^B_m+1}\big)$ using the threshold process, we now state and prove an exact formula for the probability mass function of $N^B(n,s)$. For the remainder of the section, we write $N(n,s)=N^B(n,s)$.

\begin{prop}\label{prop:GenFunN}
    Let $n,s\geq0$. Then, for all $\ell\in\N$, we have
    \begin{align*}
        \mathbb{P}\Big(N(n,s)=\ell\Big)&=q^{(s-n)\ell+\frac{\ell(\ell+1)}{2}}\left(\prod_{s+1-n+\ell\leq i\leq s}(1-q^i)\right)\sum_{A\subseteq[n]:|A|=\ell}q^{\sum_{a\in A}(a-1)}\,.
    \end{align*}
\end{prop}

\begin{proof}
    First, note that the right hand side of the equality is $0$ if $\ell\leq n-s-1$ or $\ell\geq n+1$ since either the product $\prod_{s+1-n+\ell\leq i\leq s}(1-q^i)$ equals $0$ or the sum $\sum_{A\subseteq[n]:|A|=\ell}$ is empty. For such $\ell$, $\mathbb{P}\big(N(n,s)=\ell\big)=0$ as well, so the claimed equality holds when $\ell\leq n-s-1$ or $\ell>n$. We now prove that the equality holds for $(n-s)_+\leq\ell\leq n$ by induction on $n+s$.
    
    First, if $n+s=0$, then $N(n,s)=0$ and the right hand side is equal to $1$ if and only if $\ell=0$, which proves the formula.
    
    Fix some $n,s\geq0$ and assume the formula holds for any $n'$, $s'$ such that $n'+s'<n+s$. Let $\ell$ be such that $(n-s)_+\leq\ell\leq n$. By considering the possible values for $f^B(1)$ and using the two formulas in Proposition~\ref{prop:RecursiveN}, we obtain
    \begin{align*}
        \mathbb{P}\Big(N(n,s)=\ell\Big)&=\mathbb{P}\Big(N(n,s)=\ell\,\Big|\,f^B(1)\leq s\Big)\mathbb{P}\Big(f^B(1)\leq s\Big)\\
        &\hspace{0.5cm}+\mathbb{P}\Big(N(n,s)=\ell\,\Big|\,f^B(1)>s\Big)\mathbb{P}\Big(f^B(1)>s\Big)\\
        &=\mathbb{P}\Big(N(n-1,s-1)=\ell\Big)\mathbb{P}\Big(f^B(1)\leq s\Big)\\
        &\hspace{0.5cm}+\mathbb{P}\Big(N(n-1,s)=\ell-1\Big)\mathbb{P}\Big(f^B(1)>s\Big)\,.
    \end{align*}
    Using the definition of $f^B$, we know that
    \begin{align*}
        \mathbb{P}\Big(f^B(1)\leq s\Big)&=\sum_{1\leq j\leq s}\mathbb{P}\Big(f^B(1)=j\Big)=1-q^s
    \end{align*}
    and
    \begin{align*}
        \mathbb{P}\Big(f^B(1)>s\Big)&=q^s\,.
    \end{align*}
    This gives us that
    \begin{align}
        \mathbb{P}\Big(N(n,s)=\ell\Big)&=(1-q^s)\mathbb{P}\Big(N(n-1,s-1)=\ell\Big)+q^s\mathbb{P}\Big(N(n-1,s)=\ell-1\Big)\,.\label{eq:recursiveRelationN}
    \end{align}
    
    Using the induction hypothesis, we know that
    \begin{align*}
        \mathbb{P}\Big(N(n-1,s-1)=\ell\Big)&=q^{(s-n)\ell+\frac{\ell(\ell+1)}{2}}\left(\prod_{s+1-n+\ell\leq i\leq s-1}(1-q^i)\right)\sum_{A\subseteq[n-1]:|A|=\ell}q^{\sum_{a\in A}(a-1)}\,.
    \end{align*}
    Hence, multiplying by $(1-q^s)$ on both sides and putting it into the product, we obtain
    \begin{align*}
        (1-q^s)\mathbb{P}\Big(N(n-1,s-1)=\ell\Big)&=q^{(s-n)\ell+\frac{\ell(\ell+1)}{2}}\left(\prod_{s+1-n+\ell\leq i\leq s}(1-q^i)\right)\sum_{A\subseteq[n-1]:|A|=\ell}q^{\sum_{a\in A}(a-1)}
    \end{align*}
    Similarly, we have
    \begin{align*}
        &q^s\mathbb{P}\Big(N(n-1,s)=\ell-1\Big)\\
        &\hspace{1cm}=q^s\cdot q^{(s-n+1)(\ell-1)+\frac{\ell(\ell-1)}{2}}\left(\prod_{s+1-n+\ell\leq i\leq s}(1-q^i)\right)\sum_{A\subseteq[n-1]:|A|=\ell-1}q^{\sum_{a\in A}(a-1)}\\
        &\hspace{1cm}=q^{(s-n)\ell+\frac{\ell(\ell+1)}{2}}\left(\prod_{s+1-n+\ell\leq i\leq s}(1-q^i)\right)\sum_{A\subseteq[n-1]:|A|=\ell-1}q^{n-1+\sum_{a\in A}(a-1)}\,.
    \end{align*}
    
    Putting the previous formulas into (\ref{eq:recursiveRelationN}), we obtain
    \begin{align*}
        \mathbb{P}\Big(N(n,s)=\ell\Big)&=q^{(s-n)\ell+\frac{\ell(\ell+1)}{2}}\left(\prod_{s+1-n+\ell\leq i\leq s}(1-q^i)\right)\sum_{A\subseteq[n-1]:|A|=\ell}q^{\sum_{a\in A}(a-1)}\\
        &\hspace{0.5cm}+q^{(s-n)\ell+\frac{\ell(\ell+1)}{2}}\left(\prod_{s+1-n+\ell\leq i\leq s}(1-q^i)\right)\sum_{A\subseteq[n-1]:|A|=\ell-1}q^{n-1+\sum_{a\in A}(a-1)}\,.
    \end{align*}
    In order to conclude, note that
    \begin{align*}
        \sum_{A\subseteq[n-1]:|A|=\ell-1}q^{n-1+\sum_{a\in A}(a-1)}&=\sum_{A\subseteq[n-1]:|A|=\ell-1}q^{\sum_{a\in A\cup\{n\}}(a-1)}\,,
    \end{align*}
    which implies that
    \begin{align*}
        \sum_{A\subseteq[n-1]:|A|=\ell}q^{\sum_{a\in A}(a-1)}+\sum_{A\subseteq[n-1]:|A|=\ell-1}q^{n-1+\sum_{a\in A}(a-1)}&=\sum_{A\subseteq[n]:|A|=\ell}q^{\sum_{a\in A}(a-1)}\,,
    \end{align*}
    and this proves the desired formula for $\mathbb{P}\big(N(n,s)=\ell\big)$. The induction and the proposition follow.
\end{proof}

We conclude this section on the threshold process with upper and lower tail bounds for $N$. Both of these bounds use the following inequalities:
\begin{align}
    \sum_{A\subseteq[n]:|A|=\ell}q^{\sum_{a\in A}(a-1)}\leq\frac{1}{\ell!}\sum_{a_1,\ldots,a_\ell\geq0}q^{a_1+\cdots+a_\ell}\leq\frac{1}{\ell!}\left(\frac{1}{1-q}\right)^\ell\,.\label{eq:ineqSumA}
\end{align}

\begin{prop}\label{prop:UBN}
    Let $n,s\geq0$. For any integer $\xi$ such that $q^{\xi+s-n}\leq\xi(1-q)$, we have
    \begin{align*}
        \mathbb{P}\Big(N(n,s)\geq\xi\Big)&\leq\frac{n}{\xi!}q^\frac{\xi^2}{2}\left(\frac{q^{s-n}}{1-q}\right)^\xi\,.
    \end{align*}
\end{prop}

\begin{proof}
    This proof will be very straightforward using (\ref{eq:ineqSumA}). Fix some $\xi$ respecting the given conditions. If $\xi>n$, the inequality holds as the left hand side equals $0$. Assume now that $\xi\leq n$. Using Proposition~\ref{prop:GenFunN}, we have
    \begin{align*}
        \mathbb{P}\Big(N(n,s)\geq\xi\Big)&=\sum_{\xi\leq\ell\leq n}q^{(s-n)\ell+\frac{\ell(\ell+1)}{2}}\left(\prod_{s+1-n+\ell\leq i\leq s}(1-q^i)\right)\sum_{A\subseteq[n]:|A|=\ell}q^{\sum_{a\in A}(a-1)}\,.
    \end{align*}
    Applying (\ref{eq:ineqSumA}) along with the fact that $(1-q^i)\leq1$ when $s+1-n+\ell\leq i\leq s$, it follows that
    \begin{align}
        \mathbb{P}\Big(N(n,s)\geq\xi\Big)&\leq\sum_{\xi\leq\ell\leq n}q^{(s-n)\ell+\frac{\ell(\ell+1)}{2}}\frac{1}{\ell!}\left(\frac{1}{1-q}\right)^\ell=\sum_{\xi\leq\ell\leq n}\frac{q^\frac{\ell(\ell+1)}{2}}{\ell!}\left(\frac{q^{s-n}}{1-q}\right)^\ell\,.\label{eq:SumUTBN}
    \end{align}
    
    To conclude the proof, we now show that the summands on the right are decreasing in $\ell$. To see this, note that
    \begin{align*}
        \frac{q^\frac{\ell(\ell+1)}{2}}{\ell!}\left(\frac{q^{s-n}}{1-q}\right)^\ell=\frac{q^{\ell+s-n}}{\ell(1-q)}\times\frac{q^\frac{\ell(\ell-1)}{2}}{(\ell-1)!}\left(\frac{q^{s-n}}{1-q}\right)^{\ell-1}\leq\frac{q^\frac{\ell(\ell-1)}{2}}{(\ell-1)!}\left(\frac{q^{s-n}}{1-q}\right)^{\ell-1}\,,
    \end{align*}
    the last ineuality holding since the function $\ell\mapsto\frac{q^\ell}{\ell}$ is decreasing and since we assumed that $q^{\xi+s-n}\leq\xi(1-q)$. Bounding all summands on the right hand side of (\ref{eq:SumUTBN}) by the $\ell=\xi$ term, we obtain
    \begin{align*}
        \mathbb{P}\Big(N(n,s)\geq\xi\Big)&\leq(n-\xi+1)\frac{q^\frac{\xi(\xi+1)}{2}}{\xi!}\left(\frac{q^{s-n}}{1-q}\right)^\xi\leq\frac{n}{\xi!}q^\frac{\xi^2}{2}\left(\frac{q^{s-n}}{1-q}\right)^\xi\,,
    \end{align*}
    where the second inequality uses that $n-\xi+1\leq n$ and $q^\frac{\xi}{2}\leq1$.
\end{proof}

\begin{prop}\label{prop:LBN}
    Let $n,s\geq0$. For any integer $\xi$ such that $q^{\xi+s-n}\geq\xi(1-q)^2$, we have
    \begin{align*}
        \mathbb{P}\Big(N(n,s)\leq\xi\Big)\leq\frac{2(1-q)^{n-s}}{\big((\xi-1)_+\big)!\big((s-n)_+\big)!}q^{\frac{\xi^2}{2}}\left(\frac{q^{s-n}}{(1-q)^2}\right)^\xi\,.
    \end{align*}
\end{prop}

\begin{proof}
    The proof will be very similar to the previous one as we will first bound the probability with a sum over $\ell$ and then consider the largest term. Fix some $\xi$ respecting the given condition. If $\xi<(n-s)_+$, then the inequality holds as the left hand side equals $0$. Assume now that $\xi\geq(n-s)_+$. First, by applying Proposition~\ref{prop:GenFunN} along with (\ref{eq:ineqSumA}), we obtain
    \begin{align}
        \mathbb{P}\Big(N(n,s)\leq\xi\Big)&=\sum_{(n-s)_+\leq\ell\leq\xi}q^{(s-n)\ell+\frac{\ell(\ell+1)}{2}}\left(\prod_{s+1-n+\ell\leq i\leq s}(1-q^i)\right)\sum_{A\subseteq[n]:|A|=\ell}q^{\sum_{a\in A}(a-1)}\notag\\
        &\leq\sum_{(n-s)_+\leq\ell\leq\xi}q^{(s-n)\ell+\frac{\ell(\ell+1)}{2}}\frac{1}{\ell!}\left(\frac{1}{1-q}\right)^\ell\left(\prod_{s+1-n+\ell\leq i\leq s}(1-q^i)\right)\,.\label{eq:SumLTBN}
    \end{align}
    Write $\Tilde{s}=\min\left(s,\left\lfloor\frac{1}{1-q}\right\rfloor\right)$ and use that $1-q^i\leq1$ and that $1-q^i\leq i(1-q)$, to obtain
    \begin{align*}
        \prod_{s+1-n+\ell\leq i\leq s}(1-q^i)&\leq\prod_{s+1-n+\ell\leq i\leq\Tilde{s}}i(1-q)=\frac{\Tilde{s}!(1-q)^{\Tilde{s}-s+n-\ell}}{(s-n+\ell)!}\leq\frac{1}{\big((s-n)_+\big)!}\left(\frac{1}{1-q}\right)^{s-n+\ell}\,,
    \end{align*}
    where the last inequality follows from the bounds $(s-n+\ell)!\geq\big((s-n)_+\big)!$ and
    \begin{align*}
        \Tilde{s}!(1-q)^{\Tilde{s}}=\prod_{1\leq k\leq \Tilde{s}}\Big[k(1-q)\Big]\leq1\,.
    \end{align*}
    Put this bound back into (\ref{eq:SumLTBN}) to obtain
    \begin{align}
        \mathbb{P}\Big(N(n,s)\leq\xi\Big)&\leq\frac{(1-q)^{n-s}}{\big((s-n)_+\big)!}\sum_{(n-s)_+\leq\ell\leq\xi}\frac{1}{\ell!}q^{(s-n)\ell+\frac{\ell(\ell+1)}{2}}\left(\frac{1}{1-q}\right)^{2\ell}\,.\label{eq:SumLTBN2}
    \end{align}
    
    Looking for the largest term in the sum again, we note that
    \begin{align*}
        \frac{1}{\ell!}q^{(s-n)\ell+\frac{\ell(\ell+1)}{2}}\left(\frac{1}{1-q}\right)^{2\ell}=\frac{q^{s-n+\ell}}{\ell(1-q)^2}\times\frac{1}{(\ell-1)!}q^{(s-n)(\ell-1)+\frac{\ell(\ell-1)}{2}}\left(\frac{1}{1-q}\right)^{2(\ell-1)}\,,
    \end{align*}
    and $q^{s-n+\ell}\geq\ell(1-q)^2$ for all $\ell\leq\xi$ by using the assumption on $\xi$. This implies that we can bound all terms in the sum in (\ref{eq:SumLTBN2}) from above by the $\ell=\xi$ term, and obtain
    \begin{align*}
        \mathbb{P}\Big(N(n,s)\leq\xi\Big)&\leq\Big[\xi+1-(n-s)_+\Big]\frac{(1-q)^{n-s}}{\big((s-n)_+\big)!}\frac{1}{\xi!}q^{(s-n)\xi+\frac{\xi(\xi+1)}{2}}\left(\frac{1}{1-q}\right)^{2\xi}\,.
    \end{align*}
    The desired bound follows by using that $\xi+1-(n-s)_+\leq2\xi$ and $q^\frac{\xi}{2}\leq1$.
\end{proof}

\subsection{Bounds on $M^B_m$}\label{sec:BoundsM}

For the reminder of Section~\ref{sec:IntValuesThreshold}, we define
\begin{align*}
    m=m(n)=\min\big\{\ell\geq0:\ell(1-q_n)+\log\ell\geq n(1-q_n)\big\}\,.
\end{align*}
Under the assumption that $n(1-q_n)=\omega\left(\sqrt{\log n}\right)$, this definition implies that
\begin{align}\label{eq:muOfm}
    \mu_1\big(m(n),q_n\big)\sim n(1-q_n)\,,
\end{align}
by using the asymptotic estimate for $\mu_1\big(m(n),q_n\big)$ from Proposition~\ref{prop:ConvOfMu}.

The goal of this section is to prove the following proposition.

\begin{prop}\label{prop:RangeOfM}
    Let $(q_n)_{n\geq0}$ be such that $\log\big(n(1-q_n)\big)=O\left(\sqrt{\log n}\right)$ and $n(1-q_n)=\omega\left(\sqrt{\log n}\right)$. For $n\geq0$, let $m=m(n)=\min\big\{\ell\geq0:\ell(1-q_n)+\log\ell\geq n(1-q_n)\big\}$. Then, for all $(\alpha_n)_{n\geq0}$ such that $\alpha_n=\omega\left(\sqrt{\log n}\right)$, we have
    \begin{align*}
        \lim_{n\rightarrow\infty}\mathbb{P}\left(\Big|M^B_{m(n)}-n\Big|>\frac{\alpha_n}{1-q_n}\right)=0\,.
    \end{align*}
\end{prop}

\begin{proof}
    For the rest of the proof, we omit $B$ and $n$ from the notation. We also write $p=p_n=1-q_n$. By assumption, $\log(np)=O\left(\sqrt{\log n}\right)$, so $p\rightarrow0$ and also $np^2\rightarrow0$, we will use these facts in the proof.
    
    We prove the proposition by establishing upper and lower tail bounds for $M_m$ separately. For both bounds, we will use the following inequality obtained from the moment generating function of $M_m$ given in Proposition~\ref{prop:BivariateGenFun}: for all $t\in\mathbb{R}$ such that $qe^t<1$, we have 
    \begin{align}
        \mathbb{E}\left[e^{tM_m}\right]&=\prod_{1\leq k\leq m}\left(1+(e^t-1)\frac{1}{1-q^ke^t}\right)\leq\exp\left((e^t-1)\sum_{1\leq k\leq m}\frac{1}{1-q^ke^t}\right)\,.\label{eq:IneqMGFM}
    \end{align}
    
    For the upper tail bound and for all $t>0$ such that $qe^t<1$, use Markov's inequality and (\ref{eq:IneqMGFM}) to obtain
    \begin{align*}
        \mathbb{P}\left(M_m>n+\frac{\alpha}{p}\right)&\leq\exp\left(-t\left[n+\frac{\alpha}{p}\right]+(e^t-1)\sum_{1\leq k\leq m}\frac{1}{1-q^ke^t}\right)\,.
    \end{align*}
    Let $t=-\frac{1}{2}\log q$, so that $e^t=1/\sqrt{q}$. Using that $q^ke^t=q^{k-\frac{1}{2}}\leq q^{k-1}$, it follows that
    \begin{align*}
        \sum_{1\leq k\leq m}\frac{1}{1-q^ke^t}\leq\frac{1}{1-\sqrt{q}}+\sum_{2\leq k\leq m}\frac{1}{1-q^{k-1}}=\frac{1}{1-\sqrt{q}}+\frac{1}{p}+\frac{1}{p}\mu_1(m-1,q)\,.
    \end{align*}
    From (\ref{eq:muOfm}), we know that $\mu_1(m-1,q)=\mu_1(m,q)-\frac{1-q}{1-q^m}=np+O\left(\sqrt{\log n}\right)$, since $\frac{1-q}{1-q^m}\leq1$; this implies that
    \begin{align*}
        \sum_{1\leq k\leq m}\frac{1}{1-q^ke^t}\leq n+O\left(\frac{\sqrt{\log n}}{p}\right)\,.
    \end{align*}
    Moreover, since $p=1-q\rightarrow0$, we have $t=\frac{p}{2}+O(p^2)$ and $e^t-1=\frac{p}{2}+O(p^2)$. Combining this with the previous bound on the sum, we obtain
    \begin{align*}
        \mathbb{P}\left(M_m>n+\frac{\alpha}{p}\right)&\leq\exp\left(\left(\frac{p}{2}+O(p^2)\right)\left[-n-\frac{\alpha}{p}+n+O\left(\frac{\sqrt{\log n}}{p}\right)\right]\right)\\
        &=\exp\left(-\frac{\alpha}{2}+o(\alpha)\right)\,;
    \end{align*}
    the last equality follows from the fact that $np^2\rightarrow0$ and that $\alpha=\omega\left(\sqrt{\log n}\right)$. This proves the desired upper tail bound for $M_m$.
    
    For the lower tail bound, using Markov's inequality and (\ref{eq:IneqMGFM}) again, for all $t>0$ we have
    \begin{align*}
        \mathbb{P}\left(M_m<n-\frac{\alpha}{p}\right)&\leq\exp\left(t\left[n-\frac{\alpha}{p}\right]+(e^{-t}-1)\sum_{1\leq k\leq m}\frac{1}{1-q^ke^{-t}}\right)\,.
    \end{align*}
    Let $t=-\log q$, so $e^{-t}=q$. Arguing similarly to the proof of the upper tail, we see that $t=p+O(p^2)$ and $e^{-t}-1=-p$. Moreover, by (\ref{eq:muOfm}), it follows that
    \begin{align*}
        \sum_{1\leq k\leq m}\frac{1}{1-q^ke^{-t}}=\sum_{2\leq k\leq m+1}\frac{1}{1-q^k}=\frac{1}{p}\mu_1(m+1,q)=n+O\left(\frac{\sqrt{\log n}}{p}\right)\,.
    \end{align*}
    This gives us that
    \begin{align*}
        \mathbb{P}\left(M_m<n-\frac{\alpha}{p}\right)&\leq\exp\left(\left(p+O(p^2)\right)\left[n-\frac{\alpha}{p}-n+O\left(\frac{\sqrt{\log n}}{p}\right)\right]\right)=\exp\Big(-\alpha+o(\alpha)\Big)\,,
    \end{align*}
    which concludes the proof of the lower tail bound and the proposition.
\end{proof}

\subsection{Bounds on $\big|T^B_n\big(\overline{1}^{R^B_m+1}\big)\big|$}\label{sec:BoundsThresholdSubtree}

With the results from the two previous sections, we now have all the tools required to prove Proposition~\ref{prop:BoundsOnN}.

\begin{proof}[Proof of Proposition~\ref{prop:BoundsOnN}]
    By using the distributional identity from Lemma~\ref{lem:N'andN}, it suffices to prove that
    \begin{align*}
        \lim_{n\rightarrow\infty}\mathbb{P}\left(\frac{e^{-\beta_n}}{1-q_n}\leq N^{B^*}\Big(n-m(n), M^B_{m(n)}-m(n)\Big)\leq\frac{\beta_n}{1-q_n}\right)=1\,.
    \end{align*}
    We write $p_n=1-q_n$ and note that $p_n\rightarrow0$ as $n\rightarrow\infty$. We now prove that the corresponding upper and lower tail probabilities converge to $0$ as $n\rightarrow\infty$:
    \begin{align*}
        \textrm{UB}:=\mathbb{P}\left(N^{B^*}\big(n-m(n), M^B_{m(n)}-m(n)\big)>\frac{\beta_n}{p_n}\right)\longrightarrow0
    \end{align*}
    and
    \begin{align*}
        \textrm{LB}:=\mathbb{P}\left(N^{B^*}\big(n-m(n), M^B_{m(n)}-m(n)\big)<\frac{e^{-\beta_n}}{p_n}\right)\longrightarrow0\,.
    \end{align*}
    For the rest of the proof $\alpha=(\alpha_n)_{n\geq0}$ refers to a sequence such that $\alpha_n=\omega\left(\sqrt{\log n}\right)$ and $\alpha_n=o(\beta_n)$. In other words, we have $\sqrt{\log n}\ll\alpha_n\ll\beta_n$. From now on, we omit $n$ and the superscript $B$ and $B^*$ from the notations, since the random variables $N=\big(N(n,s)\big)_{n,s\geq0}=\big(N^{B^*}(n,s)\big)_{n,s\geq0}$ are independent of the random variables $M=(M_m)_{m\geq0}=(M^B_m)_{m\geq0}$.
    
    For the upper tail, divide the probability according to the values of $M_m$ as follows
    \begin{align*}
        \textrm{UB}&=\mathbb{P}\left(N\big(n-m, M_m-m\big)>\frac{\beta}{p},M_m> n-\frac{\alpha}{p}\right)\\
        &\hspace{0.5cm}+\mathbb{P}\left(N\big(n-m, M_m-m\big)>\frac{\beta}{p},M_m\leq n-\frac{\alpha}{p}\right)\,.
    \end{align*}
    Applying Proposition~\ref{prop:RangeOfM}, we know that the second term converges to $0$. Recall now that $N(n,s)$ is decreasing in $s$. By independence of $N$ and $M$,taking $s=\left\lfloor n-\frac{\alpha}{p}\right\rfloor=n-\left\lceil\frac{\alpha}{p}\right\rceil$, it follows that
    \begin{align}
        \mathbb{P}\left(N\big(n-m, M_m-m\big)>\frac{\beta}{p},M_m\geq n-\frac{\alpha}{p}\right)&\leq\mathbb{P}\left(N\big(n-m, s-m\big)>\frac{\beta}{p}\right)\,.\label{eq:UTBN}
    \end{align}
    Write $\xi=\left\lfloor\frac{\beta}{p}\right\rfloor$. We now need to verify that $q^{\xi+(s-m)-(n-m)}\leq\xi(1-q)$, so that Proposition~\ref{prop:UBN} applies. For this, since $\xi\sim\frac{\beta}{p}$ and $(s-m)-(n-m)\sim-\frac{\alpha}{p}$, and using that $\alpha=o(\beta)$ and that $q^\frac{1}{p}=(1-p)^\frac{1}{p}\rightarrow e^{-1}$, we have
    \begin{align*}
        \frac{q^{\xi+(s-m)-(n-m)}}{\xi p}=\big(1+o(1)\big)\frac{e^{-(\beta-\alpha)(1+o(1))}}{\beta}=o(1)\,,
    \end{align*}
    which proves that this ratio is less than $1$ for $n$ large enough. This means that we can indeed apply Proposition~\ref{prop:UBN} and obtain
    \begin{align*}
        \mathbb{P}\left(N\big(n-m, s-m\big)>\frac{\beta}{p}\right)&\leq\frac{n-m}{\xi!}q^\frac{\xi^2}{2}\left(\frac{q^{s-n}}{p}\right)^\xi\\
        &=\exp\left(\log(n-m)-\log\xi!+\xi\left[\frac{\xi}{2}+(s-n)\right]\log q-\xi\log p\right)\,.
    \end{align*}
    By the definition of $\xi$ and $s$, we know that
    \begin{align*}
        -\log\xi!-\xi\log p=-\xi\log\xi+O(\log\xi)-\xi\log p\sim-\xi\log\beta
    \end{align*}
    and
    \begin{align*}
        \xi\left[\frac{\xi}{2}+(s-n)\right]\log q\sim\frac{\xi^2}{2}\log q\sim-\frac{\xi^2p}{2}\sim-\frac{\xi\beta}{2}\,.
    \end{align*}
    Since $\log(n-m)\leq\log n=o(\xi\beta)$ and $\xi\log\beta=o(\xi\beta)$, this implies that
    \begin{align*}
        \mathbb{P}\left(N\big(n-m, s-m\big)>\frac{\beta}{p}\right)&\leq\exp\left(-\big(1+o(1)\big)\frac{\xi\beta}{2}\right)=o(1)\,.
    \end{align*}
    Plugging this result back into (\ref{eq:UTBN}) proves the upper tail bound of the proposition.
    
    For the lower tail bound, we similarly divide the probability to obtain
    \begin{align}
        \textrm{LB}&=\mathbb{P}\left(N\big(n-m,M_m-m\big)<\frac{e^{-\beta}}{p},M_n\leq n+\frac{\alpha}{p}\right)\notag\\
        &\hspace{0.5cm}+\mathbb{P}\left(N\big(n-m, M_m-m\big)<\frac{e^{-\beta}}{p},M_n>n+\frac{\alpha}{p}\right)\notag\\
        &\leq\mathbb{P}\left(N\big(n-m,s-m\big)<\frac{e^{-\beta}}{p}\right)+o(1)\label{eq:LTBN}\,,
    \end{align}
    where $s=\left\lfloor n+\frac{\alpha}{p}\right\rfloor=n+\left\lfloor\frac{\alpha}{p}\right\rfloor$. Write $\xi=\left\lfloor\frac{e^{-\beta}}{p}\right\rfloor$. To verify that the requirement of Proposition~\ref{prop:LBN} that $q^{\xi+(s-m)-(n-m)}\geq\xi(1-q)^2$ is satisfied, note that
    \begin{align*}
        \frac{q^{\xi+(s-m)-(n-m)}}{\xi p^2}=\big(1+o(1)\big)\frac{e^{(1+o(1))(\beta-\alpha)}}{p}\rightarrow\infty\,,
    \end{align*}
    so is larger than $1$ for $n$ large enough. For such $n$, applying Proposition~\ref{prop:LBN}, we obtain
    \begin{align*}
        \mathbb{P}\left(N\big(n-m,s-m\big)<\frac{e^{-\beta}}{p}\right)&\leq\frac{2p^{n-s}}{\big((\xi-1)_+\big)!\big((s-n)_+\big)!}q^\frac{\xi^2}{2}\left(\frac{q^{s-n}}{p^2}\right)^\xi\\
        &\leq2\frac{p^{n-s-2\xi}}{(s-n)!}\\
        &=2\exp\Big((n-s-2\xi)\log p-\log(s-n)!\Big)\,.
    \end{align*}
    Using the definition of $\xi$ and $s$, we know that
    \begin{align*}
        (n-s)\log p-\log(s-n)!=(n-s)\log\big(p(s-n)\big)+O(\log(s-n))\sim-\frac{\alpha\log\alpha}{p}
    \end{align*}
    and
    \begin{align*}
        2\xi\log p\sim-\frac{2e^{-\beta}\log n}{p}=o\left(\frac{1}{p}\right)\,.
    \end{align*}
    This implies that
    \begin{align*}
        \mathbb{P}\left(N\big(n-m,s-m\big)<\frac{e^{-\beta}}{p}\right)&\leq2\exp\left(-\big(1+o(1)\big)\frac{\alpha\log\alpha}{p}\right)=o(1)\,,
    \end{align*}
    which proves the desired lower tail bound by plugging this result back into (\ref{eq:LTBN}). This concludes the proof of the proposition.
\end{proof}

\subsection{Convergence in probability}\label{sec:ThmProb}

We conclude this section with the proof of Proposition~\ref{prop:convRightSubtree} and then Theorem~\ref{thm:Prob}. We start with two straightforward lemmas.

\begin{lemma}\label{lem:convRBm}
    Let $(q_n)_{n\geq0}$ be such that $\log\big(n(1-q_n)\big)=O\left(\sqrt{\log n}\right)$ and $n(1-q_n)=\omega\left(\sqrt{\log n}\right)$. For $n\geq0$, let $m=m(n)=\min\big\{\ell\geq0:\ell(1-q_n)+\log\ell\geq n(1-q_n)\big\}$. Then we have
    \begin{align*}
        \frac{R^B_{m(n)}}{n(1-q_n)}\longrightarrow1
    \end{align*}
    in probability as $n\rightarrow\infty$.
\end{lemma}

\begin{proof}
    By applying Lemma~\ref{lem:ConcBoundRn}, and since $c\log\left(\frac{e}{c}\right)<1$ for all $c\neq1$, we have
    \begin{align*}
        \frac{R^B_{m(n)}}{\mu_1\big(m(n),q_n\big)}\overset{\mathbb{P}}{\longrightarrow}1\,.
    \end{align*}
    Moreover, since $n(1-q_n)=\omega\left(\sqrt{\log n}\right)$, by (\ref{eq:muOfm}), we have that $\mu_1\big(m(n),q_n\big)\sim n(1-q_n)$, which proves the desired result.
\end{proof}

\begin{lemma}\label{lem:SizeAtDandD'}
    Let $n\geq0$ and $q\in(0,1)$. Then, for all integers $d\geq0$ and $\ell\geq0$, we have
    \begin{align*}
        \Big|T_{n,q}\big(\overline{1}^{d+\ell}\big)\Big|\succeq\frac{\log\left(1-P_\ell\left(1-q^{\left|T_{n,q}(\overline{1}^d)\right|}\right)\right)}{\log q}-d\,,
    \end{align*}
    where $P_\ell$ is distributed as a product of $\ell$ independent $\textsc{Uniform}([0,1])$, and is independent of $\big|T_{n,q}(\overline{1}^d)\big|$.
\end{lemma}

\begin{proof}
    This is simply a restatement of the lower bound from Lemma~\ref{lem:LowBoundS}, when applied to the tree $T_{n,q}(\overline{1}^d)$. By Proposition~\ref{prop:ProjCons}, $T_{n,q}(\overline{1}^d)$ is  Mallows tree once conditioned on its size, so this application of Lemma~\ref{lem:LowBoundS} is indeed valid.
\end{proof}

The proof of Proposition~\ref{prop:convRightSubtree} now follows from combining Lemma~\ref{lem:convRBm} and \ref{lem:SizeAtDandD'} with Proposition~\ref{prop:BoundsOnN}.

\begin{proof}[Proof of Proposition~\ref{prop:convRightSubtree}]
    Since $T^B_n$ has the same distribution as $T_{n,q}$, we can prove the proposition by showing that
    \begin{align*}
        h\Big(T^B_n\big(\overline{1}^{\lfloor n(1-q_n)\rfloor)}\big)\Big)=\big(c^*+o_\mathbb{P}(1)\big)\log n\,.
    \end{align*}
    In order to prove this asymptotic result, we will show that $N=\big|T^B_n\big(\overline{1}^{\lfloor n(1-q_n)\rfloor)}\big)\big|$ satisfies that $N(1-q_n)=o_\mathbb{P}(\log N)$ and that $\log N=\big(1+o_\mathbb{P}(1)\big)\log n$. This will allow us to apply Proposition~\ref{prop:convProbSmall} to conclude the proof. For the remainder of the proof, we drop the subscript $n$ on $q_n$ and write $m=m(n)=\min\big\{\ell\geq0:\ell(1-q)+\log\ell\geq n(1-q)\big\}$.
    
    We start with the upper bound on the size of the tree. In the case where $R^B_m>\lfloor n(1-q)\rfloor$, we have that
    \begin{align*}
        T^B_n\big(\overline{1}^{\lfloor n(1-q)\rfloor}\big)&=T^B_n\big(\overline{1}^{R^B_m}\big)\cup\big\{\overline{1}^k:\lfloor n(1-q)\rfloor\leq k<R^B_m\big\}\cup\bigcup_{\lfloor n(1-q)\rfloor\leq k<R^B_m}T^B_n\big(\overline{1}^k\overline{0}\big)\\
        &\subseteq T^B_n\big(\overline{1}^{R^B_m}\big)\cup\big\{\overline{1}^k:\lfloor n(1-q)\rfloor\leq k<R^B_m\big\}\cup\bigcup_{\lfloor n(1-q)\rfloor\leq k<R^B_m}T^B\big(\overline{1}^k\overline{0}\big)\,.
    \end{align*}
    Moreover, this inclusion remains true when $R^B_m\leq\lfloor n(1-q)\rfloor$, since
    \begin{align*}
        T^B_n\big(\overline{1}^{\lfloor n(1-q)\rfloor}\big)\subseteq T^B_n\big(\overline{1}^{R^B_m}\big)\,.
    \end{align*}
    Overall, we obtain that
    \begin{align*}
        \Big|T^B_n\big(\overline{1}^{\lfloor n(1-q)\rfloor}\big)\Big|\leq\Big|T^B_n\big(\overline{1}^{R^B_m}\big)\Big|+\big(R^B_m-\lfloor n(1-q)\rfloor\big)_++\sum_{\lfloor n(1-q)\rfloor\leq k<R^B_m}\Big|T^B\big(\overline{1}^k\overline{0}\big)\Big|\,.
    \end{align*}
    Lemma~\ref{lem:convRBm} tells us that $R^B_m-\lfloor n(1-q)\rfloor=o_\mathbb{P}(\log n)$. Moreover, by Lemma~\ref{lem:LeftSubtreesDistribution}, we know that the entries of the sequence $\big(\big|T^B\big(\overline{1}^k\overline{0}\big)\big|\big)_{k\geq0}$ are independent $\textsc{Geometric}(1-q)$ random variables, which gives us that
    \begin{align*}
        \sum_{\lfloor n(1-q)\rfloor\leq k<R^B_m}\Big|T^B\big(\overline{1}^k\overline{0}\big)\big|=o_\mathbb{P}\left(\frac{\log n}{1-q}\right)\,.
    \end{align*}
    Finally, by Proposition~\ref{prop:BoundsOnN}, we know that $\big|T^B_n\big(\overline{1}^{R^B_m}\big)\big|=o_\mathbb{P}\left(\frac{\log n}{1-q}\right)$. Combining all those results, we obtain that
    \begin{align}
        \Big|T^B_n\big(\overline{1}^{\lfloor n(1-q)\rfloor}\big)\Big|=o_\mathbb{P}\left(\frac{\log n}{1-q}\right)\,.\label{eq:UBTB1D}
    \end{align}
    
    We focus now on bounding the size of the tree from below. Write $E^+=\big\{R^B_m>\lfloor n(1-q)\rfloor\big\}$ and $E^-=\big\{R^B_m\leq\lfloor n(1-q)\rfloor\big\}$. Since on $E^+$, we have
    \begin{align*}
        T^B_n\big(\overline{1}^{R^B_m}\big)\subseteq T^B_n\big(\overline{1}^{\lfloor n(1-q)\rfloor}\big)\,,
    \end{align*}
    it follows that
    \begin{align*}
        \Big|T^B_n\big(\overline{1}^{\lfloor n(1-q)\rfloor}\big)\Big|\mathbbm{1}_{E^+}\geq\Big|T^B_n\big(\overline{1}^{R^B_m}\big)\Big|\mathbbm{1}_{E^+}\,.
    \end{align*}
    Moreover, by Proposition~\ref{prop:BoundsOnN}, we know that
    \begin{align*}
        \Big|T^B_n\big(\overline{1}^{R^B_m}\big)\Big|\geq\big(1+o_\mathbb{P}(1)\big)\frac{e^{-(\log n)^\frac{3}{4}}}{1-q}\,,
    \end{align*}
    which implies that
    \begin{align*}
        \Big|T^B_n\big(\overline{1}^{\lfloor n(1-q)\rfloor}\big)\Big|\mathbbm{1}_{E^+}\geq\big(1+o_\mathbb{P}(1)\big)\frac{e^{-(\log n)^\frac{3}{4}}}{1-q}\mathbbm{1}_{E^+}\,.
    \end{align*}
    Using the fact that $-\log(1-q)\sim\log n$, we obtain
    \begin{align}
        \mathbbm{1}_{E^+}\log\Big|T^B_n\big(\overline{1}^{\lfloor n(1-q)\rfloor}\big)\Big|\geq\mathbbm{1}_{E^+}\big(1+o_\mathbb{P}(1)\big)\log n\,.\label{eq:LBTB1DHalf}
    \end{align}
    
    For the second part of the lower bound, letting $D=\lfloor n(1-q)\rfloor-R^B_m$ and applying Lemma~\ref{lem:SizeAtDandD'}, we have that
    \begin{align*}
        \Big|T^B_n\big(\overline{1}^{\lfloor n(1-q)\rfloor}\big)\Big|\mathbbm{1}_{E^-}\succeq\left[\frac{\log\left(1-P_D\left(1-q^{\left|T^B_n(\overline{1}^{R^B_m})\right|}\right)\right)}{\log q}-D\right]\mathbbm{1}_{E^-}\,.
    \end{align*}
    Using the lower bound of Proposition~\ref{prop:BoundsOnN}, we obtain
    \begin{align*}
        \frac{\log\left(1-P_D\left(1-q^{\left|T^B_n(\overline{1}^{R^B_m})\right|}\right)\right)}{\log q}\mathbbm{1}_{E^-}&\geq\frac{\log\left(1-P_D\left(1-q^{(1+o_\mathbb{P}(1))\frac{e^{-(\log n)^\frac{3}{4}}}{1-q}}\right)\right)}{\log q}\mathbbm{1}_{E^-}\\
        &=\big(1+o_\mathbb{P}(1)\big)\frac{P_De^{-(\log n)^\frac{3}{4}}}{1-q}\mathbbm{1}_{E^-}\,.
    \end{align*}
    Recalling that $P_D$ is a product of $D$ independent uniforms, it is immediate that $\log P_D=\Theta_\mathbb{P}(D)$. Moreover, by Lemma~\ref{lem:convRBm}, we know that $D=\lfloor n(1-q)\rfloor-R^B_m=o_\mathbb{P}(\log n)$. Combining these results, we obtain that
    \begin{align*}
        \left(\frac{P_De^{-(\log n)^\frac{3}{4}}}{1-q}-D\right)\mathbbm{1}_{E^-}=\left(\frac{1}{1-q}e^{o_\mathbb{P}(\log n)}+o_\mathbb{P}(\log n)\right)\mathbbm{1}_{E^-}=n^{1-o_\mathbb{P}(1)}\mathbbm{1}_{E^-}\,,
    \end{align*}
    hence
    \begin{align*}
        \mathbbm{1}_{E^-}\log\Big|T^B_n\big(\overline{1}^{\lfloor n(1-q)\rfloor}\big)\Big|\geq\mathbbm{1}_{E^-}\big(1+o_\mathbb{P}(1)\big)\log n\,.
    \end{align*}
    Combined with (\ref{eq:LBTB1DHalf}), this implies that
    \begin{align}
        \log\Big|T^B_n\big(\overline{1}^{\lfloor n(1-q)\rfloor}\big)\Big|\geq\big(1+o_\mathbb{P}(1)\big)\log n\,.\label{eq:LBTB1D}
    \end{align}
    
    On the other hand, taking the logarithm in (\ref{eq:UBTB1D}), we have that
    \begin{align*}
        \log\Big|T^B_n\big(\overline{1}^{\lfloor n(1-q)\rfloor}\big)\Big|\leq\big(1+o_\mathbb{P}(1)\big)\log n\,,
    \end{align*}
    and combining this upper bound with (\ref{eq:LBTB1D}), it follows that
    \begin{align*}
        \log\Big|T^B_n\big(\overline{1}^{\lfloor n(1-q)\rfloor}\big)\Big|=\big(1+o_\mathbb{P}(1)\big)\log n\,.
    \end{align*}
    Plugging this back into (\ref{eq:UBTB1D}) shows that
    \begin{align*}
         \Big|T^B_n\big(\overline{1}^{\lfloor n(1-q)\rfloor}\big)\Big|(1-q)=o_\mathbb{P}\left(\log\Big|T^B_n\big(\overline{1}^{\lfloor n(1-q)\rfloor}\big)\Big|\right)\,.
    \end{align*}
    This bound implies that we can apply Proposition~\ref{prop:convProbSmall} to this subtree and obtain that
    \begin{align*}
        h\Big(T^B_n\big(\overline{1}^{\lfloor n(1-q)\rfloor}\big)\Big)=\big(c^*+o_\mathbb{P}(1)\big)\log\Big|T^B_n\big(\overline{1}^{\lfloor n(1-q)\rfloor}\big)\Big|=\big(c^*+o_\mathbb{P}(1)\big)\log n\,.
    \end{align*}
    Since $T^B_n$ and $T_{n,q}$ are identically distributed, this concludes the proof of the proposition.
\end{proof}

We conclude this section with the proof of Theorem~\ref{thm:Prob}.

\begin{proof}[Proof of Theorem~\ref{thm:Prob}]
    Let us first assume that $(q_n)_{n\geq0}$ is such that $n(1-q_n)=\Theta(\log n)$, and prove that
    \begin{align*}
        h\big(T_{n,q_n}\big)=n(1-q_n)+c^*\log n+o_\mathbb{P}(\log n)\,.
    \end{align*}
    
    Using Proposition~\ref{prop:convRightSubtree} and the fact that
    \begin{align*}
        h\big(T^B_n\big)\geq\lfloor n(1-q_n)\rfloor+h\Big(T^B_n\big(\overline{1}^{\lfloor n(1-q_n)\rfloor}\big)\Big)\,,
    \end{align*}
    it follows that
    \begin{align*}
        h\big(T_{n,q_n}\big)\overset{d}{=}h\big(T^B_n\big)\geq n(1-q_n)+c^*\log n+o_\mathbb{P}(\log n)\,.
    \end{align*}
    
    For the upper bound, recall the stochastic inequality from (\ref{eq:second_height_bd}) and apply it with $d=\lfloor n(1-q_n)\rfloor$:
    \begin{align*}
        h(T_{n,q})&\preceq\lfloor n(1-q_n)\rfloor+\max\left\{\sup_{k\geq0}\Big\{h\Big(T^B\big(\overline{1}^k\overline{0}\big)\Big)-k\Big\},h\Big(T^B_n\big(\overline{1}^{\lfloor n(1-q_n)\rfloor}\big)\Big)\right\}\,.
    \end{align*}
    Now, Proposition~\ref{prop:convRightSubtree} tells us that
    \begin{align*}
        h\Big(T^B_n\big(\overline{1}^{\lfloor n(1-q_n)\rfloor}\big)\Big)=\big(c^*+o_\mathbb{P}(1)\big)\log n\,,
    \end{align*}
    and from Proposition~\ref{prop:boundsLeftSubtree}, we know that
    \begin{align*}
        \sup_{k\geq0}\Big\{h\Big(T^B\big(\overline{1}^k\overline{0}\big)\Big)-k\Big\}\leq c^*\log\left(\frac{1}{1-q_n}\right)+O_\mathbb{P}\left(\sqrt{\log\left(\frac{1}{1-q_n}\right)}\right)\,.
    \end{align*}
    Since $\log\frac{1}{1-q_n}\sim\log n$, it follows that
    \begin{align*}
        \sup_{k\geq0}\Big\{h\Big(T^B\big(\overline{1}^k\overline{0}\big)\Big)-k\Big\}\leq\big(c^*+o_\mathbb{P}(1)\big)\log n\,,
    \end{align*}
    and this proves that
    \begin{align*}
        h\big(T_{n,q_n}\big)\leq n(1-q_n)+c^*\log n+o_\mathbb{P}(\log n)\,.
    \end{align*}
    This concludes the proof of the fact
    \begin{align*}
        h\big(T_{n,q_n}\big)=n(1-q_n)+c^*\log n+o_\mathbb{P}(\log n)
    \end{align*}
    whenever $n(1-q_n)=\Theta(\log n)$.
    
    Now, in order to prove Theorem~\ref{thm:Prob}, let $(q_n)_{n\geq0}$ be any sequence taking values in $[0,1]$. By Proposition~\ref{prop:UI}, it suffices to prove that
    \begin{align*}
        \left(\frac{h\big(T_{n,q_n}\big)}{n(1-q_n)+c^*\log n}\right)_{n\geq1}
    \end{align*}
    converges in probability to $1$. By considering subsequences if necessary, we can assume that $(q_n)_{n\geq0}$ falls into one of the following regimes:
    \begin{itemize}
        \item $n(1-q_n)=\Theta(\log n)$.
        \item $n(1-q_n)=\omega(\log n)$.
        \item $n(1-q_n)=o(\log n)$ and $q_n\neq1$ for all $n\geq0$.
        \item $q_n=1$ for all $n\geq0$.
    \end{itemize}
    The case $n(1-q_n)=\Theta(\log n)$ was handled in the first part of the proof. The case $n(1-q_n)=\omega(\log n)$ follows from Theorem~\ref{thm:AS}. When $n(1-q_n)=o(\log n)$, the results follows from Proposition~\ref{prop:convProbSmall}, with the bounding seuence $(\gamma_n)_{n\geq0}$ in that proposition chosen so that $\sqrt{\log n}\vee n(1-q_n)\ll\gamma_n\ll\log n$. Finally, the case when $q_n=1$ for all $n$ is simply that of binary search trees, in which case the result was proved by Devroye~\cite{devroye1986note}. This concludes the proof of Theorem~\ref{thm:Prob}.
\end{proof}

\section{Distributional limits}\label{sec:CLT}

In this last section, we prove Theorem~\ref{thm:CLT} and Theorem~\ref{thm:Poisson}. The bulk of this section is devoted to proving the central limit theorem for the right depth $R^B_n$; this was stated as Proposition~\ref{prop:CLTRD} above. We then prove the central limit theorem for the height by combining this proposition with Proposition~\ref{prop:convRenSubtree} and \ref{prop:boundsLeftSubtree}. We conclude this section with the proof of Theorem~\ref{thm:Poisson}.

\subsection{Central limit theorem for the right depth}\label{sec:CLTRD}

Before proving the central limit theorem for $h(T_{n,q_n})$, we prove Proposition~\ref{prop:CLTRD}, which corresponds to a central limit theorem for the right depth $R^B_n$.

\begin{proof}[Proof of Proposition~\ref{prop:CLTRD}]
    Let $(q_n)_{n\geq0}$ be such that $n(1-q_n)=\omega(\log n)$ and $nq_n=\omega(1)$. Note that this implies that $n(1-q_n)q_n=\omega(1)$. Consider the characteristic function $g_n$ defined as follows:
    \begin{align*}
        g_n(t)=\mathbb{E}\left[\exp\left(it\cdot\frac{R^B_n-n(1-q_n)-\log\big((1-q_n)^{-1}\big)}{\sqrt{n(1-q_n)q_n}}\right)\right]\,.
    \end{align*}
    We will prove that $g_n(t)\rightarrow e^{-\frac{t^2}{2}}$ pointwise as $n\rightarrow\infty$ by dividing into two cases: $nq_n^3=\omega(1)$ and $nq_n^3=O(1)$ (even though the two proofs for the two cases are very similar, we did not see a way to combine them). The proposition then follows by the continuity theorem for characteristic functions. For the remainder of the proof, we drop the subscript $n$ on $q_n$ and write $p=p_n=1-q_n$.
    
    Assume first that $np=\omega(\log n)$ and $nq^3=\omega(1)$, and fix $t\in\mathbb{R}$. Using the characteristic function for $R^B_n$ from Proposition~\ref{prop:BivariateGenFun}, we have that
    \begin{align}
        g_n(t)&=\exp\left(-it\frac{np+\log(p^{-1})}{\sqrt{npq}}\right)\mathbb{E}\left[\exp\left(\frac{it}{\sqrt{npq}}R^B_n\right)\right]\notag\\
        &=\exp\left(-it\frac{np+\log(p^{-1})}{\sqrt{npq}}\right)\prod_{1<k\leq n}\left(1+\left(e^{\frac{it}{\sqrt{npq}}}-1\right)\frac{p}{1-q^k}\right)\notag\\
        &=\exp\Bigg(-it\frac{np+\log(p^{-1})}{\sqrt{npq}}+\sum_{1<k\leq n}\log\left[1+\left(e^{\frac{it}{\sqrt{npq}}}-1\right)\frac{p}{1-q^k}\right]\Bigg)\,.\label{eq:prop231}
    \end{align}
    We know that $npq=\omega(1)$, which implies that
    \begin{align*}
        e^{\frac{it}{\sqrt{npq}}}=1+\frac{it}{\sqrt{npq}}-\frac{t^2}{2npq}+O\left(\frac{1}{(npq)^\frac{3}{2}}\right)\,.
    \end{align*}
    Since $npq=\omega(1)$ and $\frac{p}{1-q^k}=\frac{1-q}{1-q^k}\leq1$ for all $1<k\leq n$, we have that
    \begin{align*}
        \log\left[1+\left(e^{\frac{it}{\sqrt{npq}}}-1\right)\frac{p}{1-q^k}\right]&=\log\left[1+\left(\frac{it}{\sqrt{npq}}-\frac{t^2}{2npq}+O\left(\frac{1}{(npq)^\frac{3}{2}}\right)\right)\frac{p}{1-q^k}\right]\\
        &=\left(\frac{it}{\sqrt{npq}}-\frac{t^2}{2npq}+O\left(\frac{1}{(npq)^\frac{3}{2}}\right)\right)\frac{p}{1-q^k}\\
        &\hspace{0.5cm}-\frac{1}{2}\left[\left(\frac{it}{\sqrt{npq}}-\frac{t^2}{2npq}+O\left(\frac{1}{(npq)^\frac{3}{2}}\right)\right)\frac{p}{1-q^k}\right]^2\\
        &\hspace{0.5cm}+O\left(\left[\left(\frac{it}{\sqrt{npq}}-\frac{t^2}{2npq}+O\left(\frac{1}{(npq)^\frac{3}{2}}\right)\right)\frac{p}{1-q^k}\right]^3\right)\,;
    \end{align*}
    The second equality following from the standard expansion for the logarithm: $\log(1+z)=z-\frac{1}{2}z^2+O(z^3)$ when $|z|\rightarrow0$. Now simplify the last equation by absorbing lower order terms inside the $O\Big(\frac{1}{(npq)^\frac{3}{2}}\Big)$ to obtain
    \begin{align*}
        \log\left[1+\left(e^{\frac{it}{\sqrt{npq}}}-1\right)\frac{p}{1-q^k}\right]&=\left(\frac{it}{\sqrt{npq}}-\frac{t^2}{2npq}+O\left(\frac{1}{(npq)^\frac{3}{2}}\right)\right)\frac{p}{1-q^k}\notag\\
        &\hspace{0.5cm}+\left(\frac{t^2}{2npq}+O\left(\frac{1}{(npq)^\frac{3}{2}}\right)\right)\left[\frac{p}{1-q^k}\right]^2\notag\\
        &\hspace{0.5cm}+O\left(\frac{1}{(npq)^\frac{3}{2}}\right)\left[\frac{p}{1-q^k}\right]^3\notag\,,
    \end{align*}
    where the implied constants in the big-$O$ terms are uniform in $k$. Take the sum of the previous terms over all $1<k\leq n$, to obtain
    \begin{align}
        &\sum_{1<k\leq n}\log\left[1+\left(e^{\frac{it}{\sqrt{npq}}}-1\right)\frac{p}{1-q^k}\right]\label{eq:prop232}\\
        &\hspace{1cm}=\sum_{1<k\leq n}\left(\frac{it}{\sqrt{npq}}-\frac{t^2}{2npq}+O\left(\frac{1}{(npq)^\frac{3}{2}}\right)\right)\frac{p}{1-q^k}&\left(=:\textsc{CLT}^{(1)}\right)&\notag\\
        &\hspace{1.5cm}+\sum_{1<k\leq n}\left(\frac{t^2}{2npq}+O\left(\frac{1}{(npq)^\frac{3}{2}}\right)\right)\left[\frac{p}{1-q^k}\right]^2&\left(=:\textsc{CLT}^{(2)}\right)&\notag\\
        &\hspace{1.5cm}+\sum_{1<k\leq n}O\left(\frac{1}{(npq)^\frac{3}{2}}\right)\left[\frac{p}{1-q^k}\right]^3\,.&\left(=:\textsc{CLT}^{(3)}\right)&\notag
    \end{align}
    
    For the first term, this gives us
    \begin{align*}
        \textsc{CLT}^{(1)}&=\left(\frac{it}{\sqrt{npq}}-\frac{t^2}{2npq}+O\left(\frac{1}{(npq)^\frac{3}{2}}\right)\right)\sum_{1<k\leq n}\frac{p}{1-q^k}\\
        &=\left(\frac{it}{\sqrt{npq}}-\frac{t^2}{2npq}+O\left(\frac{1}{(npq)^\frac{3}{2}}\right)\right)\mu_1(n,q)\,.
    \end{align*}
    Applying Proposition~\ref{prop:ConvOfMu} and using that $p^{-1}\leq n$, which by assumption holds for $n$ large enough, we know that $\mu_1(n,q)=np+\log(p^{-1})+O\big(\sqrt{|\log p|}\big)$. This gives us that
    \begin{align*}
        \textsc{CLT}^{(1)}&=\left(\frac{it}{\sqrt{npq}}-\frac{t^2}{2npq}+O\left(\frac{1}{(npq)^\frac{3}{2}}\right)\right)\Big(np+\log(p^{-1})+O\big(\sqrt{|\log p|}\big)\Big)\\
        &=it\frac{np}{\sqrt{npq}}-\frac{t^2}{2q}+it\frac{\log(p^{-1})}{\sqrt{npq}}+O\left(\frac{np}{(npq)^\frac{3}{2}}\right)+O\left(\frac{\log p}{npq}\right)+O\left(\sqrt{\frac{|\log p|}{npq}}\right)\,;
    \end{align*}
    Three of the nine terms which are obtained by formally expanding the product have been absorbed in the big-$O$ terms above as they are of smaller order. Now recall that $np=\omega(\log n)$ and $nq^3=\omega(1)$, which implies that $np=o\big((npq)^\frac{3}{2}\big)$ and that $\frac{\log p}{npq}=o(1)$; if follows that $\sqrt{\frac{|\log p|}{npq}}=o(1)$ as well. Combining all these bounds, we obtain
    \begin{align}
        \textsc{CLT}^{(1)}&=it\frac{np}{\sqrt{npq}}-\frac{t^2}{2q}+it\frac{\log(p^{-1})}{\sqrt{npq}}+o(1)\label{eq:prop23A}\,.
    \end{align}
    For the second term, since Proposition~\ref{prop:ConvOfMuAlpha} tells us that $\mu_2(n,q)=np^2+O(1)$, we have
    \begin{align*}
        \textsc{CLT}^{(2)}&=\left(\frac{t^2}{2npq}+O\left(\frac{1}{(npq)^\frac{3}{2}}\right)\right)\mu_2(n,q)\\
        &=\left(\frac{t^2}{2npq}+O\left(\frac{1}{(npq)^\frac{3}{2}}\right)\right)\Big(np^2+O(1)\Big)\\
        &=\frac{pt^2}{2q}+O\left(\frac{np^2}{(npq)^\frac{3}{2}}\right)+O\left(\frac{1}{npq}\right)\,.
    \end{align*}
    We know that $p\leq1$, and as we previously explained, $np=o\big((npq)^\frac{3}{2}\big)$. This implies that $np^2=o\big((npq)^\frac{3}{2}\big)$. Combining all these bounds with the fact that $npq=\omega(1)$, we obtain that
    \begin{align}
        \textsc{CLT}^{(2)}&=\frac{pt^2}{2q}+o(1)\label{eq:prop23B}\,.
    \end{align}
    Finally, using that $\mu_3(n,q)=np^3+O(1)$ from Proposition~\ref{prop:ConvOfMuAlpha}, and similar arguments to the ones used for the previous two terms, we obtain 
    \begin{align}
        \textsc{CLT}^{(3)}&=O\left(\frac{1}{(npq)^\frac{3}{2}}\right)\mu_3(n,p)=O\left(\frac{np^3}{(npq)^\frac{3}{2}}\right)+O\left(\frac{1}{(npq)^\frac{3}{2}}\right)=o(1)\label{eq:prop23C}\,.
    \end{align}
    
    Putting the bounds for $\textsc{CLT}^{(1)}$, $\textsc{CLT}^{(2)}$ and $\textsc{CLT}^{(3)}$ obtained in (\ref{eq:prop23A}), (\ref{eq:prop23B}) and (\ref{eq:prop23C}) back into (\ref{eq:prop232}), we obtain
    \begin{align*}
        \sum_{1<k\leq n}\log\left[1+\left(e^{\frac{it}{\sqrt{npq}}}-1\right)\frac{p}{1-q^k}\right]&=it\frac{np}{\sqrt{npq}}-\frac{t^2}{2q}+it\frac{\log(p^{-1})}{\sqrt{npq}}+\frac{pt^2}{2q}+o(1)\\
        &=it\frac{np}{\sqrt{npq}}-\frac{t^2}{2}+it\frac{\log\big(p^{-1}\big)}{\sqrt{npq}}+o(1)\,.
    \end{align*}
    Once plugged back into (\ref{eq:prop231}), this gives us
    \begin{align*}
        g_n(t)&=\exp\left(-it\frac{np+\log(p^{-1})}{\sqrt{npq}}+it\frac{np}{\sqrt{npq}}-\frac{t^2}{2}+it\frac{\log(p^{-1})}{\sqrt{npq}}+o(1)\right)\\
        &=\exp\left(-\frac{t^2}{2}+o(1)\right)\,,
    \end{align*}
    which is the desired result.
    
    Assume now that $nq^3=O(1)$ and $nq=\omega(1)$. Note that this implies that $nq^4=o(1)$ and that $npq=\big(1+o(1)\big)nq\rightarrow\infty$. From the second identity, it follows that $\log(p^{-1})=o\big(\sqrt{npq}\big)$, which implies that
    \begin{align*}
        g_n(t)=\mathbb{E}\left[\exp\left(it\frac{R^B_n-np}{\sqrt{npq}}\right)\right]+o(1)\,.
    \end{align*}
    Using the characteristic function for $R^B_n$ again, we have that
    \begin{align}
        g_n(t)&=\exp\left(-it\frac{np}{\sqrt{npq}}\right)\prod_{1<k\leq n}\left(1+\left(e^{\frac{it}{\sqrt{npq}}}-1\right)\frac{p}{1-q^k}\right)+o(1)\,.\label{eq:gntSecondCase}
    \end{align}
    Note that, for any $1<k\leq n$, we have
    \begin{align*}
        \left(e^{\frac{it}{\sqrt{npq}}}-1\right)\frac{p}{1-q^k}=o(1)\,.
    \end{align*}
    Moreover, for $k\geq4$, since $q\rightarrow0$ and $p\leq1$, we have
    \begin{align*}
        p\leq\frac{p}{1-q^k}\leq\frac{p}{1-q^4}=p+O(q^4)\,,
    \end{align*}
    which implies that $\frac{p}{1-q^k}=p+O(q^4)$ where the implied bound in the big-$O$ term is independent of $k$. Using the two previous results into (\ref{eq:gntSecondCase}) and dividing the product according to whether $k<4$ or $k\geq4$, we obtain
    \begin{align*}
        g_n(t)&=\big(1+o(1)\big)\exp\left(-it\frac{np}{\sqrt{npq}}\right)\prod_{4\leq k\leq n}\bigg(1+\left(e^{\frac{it}{\sqrt{npq}}}-1\right)\big(p+O(q^4)\big)\bigg)+o(1)\\
        &=\big(1+o(1)\big)\exp\left(-it\frac{np}{\sqrt{npq}}\right)\bigg(1+\left(e^{\frac{it}{\sqrt{npq}}}-1\right)\big(p+O(q^4)\big)\bigg)^{n-3}+o(1)\,.
    \end{align*}
    Now, since $npq\rightarrow\infty$, we have
    \begin{align*}
        \left(e^{\frac{it}{\sqrt{npq}}}-1\right)\big(p+O(q^4)\big)&=\left(\frac{it}{\sqrt{npq}}-\big(1+o(1)\big)\frac{t^2}{2npq}\right)\big(p+O(q^4)\big)\\
        &=\frac{itp}{\sqrt{npq}}-\big(1+o(1)\big)\frac{t^2p}{2npq}+O\left(\frac{q^4}{\sqrt{npq}}\right)\,,
    \end{align*}
    from which it follows that
    \begin{align*}
        &\bigg(1+\left(e^{\frac{it}{\sqrt{npq}}}-1\right)\big(p+O(q^4)\big)\bigg)^{n-3}\\
        &\hspace{1cm}=\exp\left((n-3)\log\left(1+\frac{itp}{\sqrt{npq}}-\big(1+o(1)\big)\frac{t^2p}{2npq}+O\left(\frac{q^4}{\sqrt{npq}}\right)\right)\right)\\
        &\hspace{1cm}=\exp\left((n-3)\left[\frac{itp}{\sqrt{npq}}-\big(1+o(1)\big)\frac{t^2p}{2npq}+\big(1+o(1)\big)\frac{1}{2}\frac{t^2p^2}{npq}+O\left(\frac{q^4}{\sqrt{npq}}\right)\right]\right)\\
        &\hspace{1cm}=\exp\left(it\frac{np}{\sqrt{npq}}-\frac{t^2}{2}+o(1)\right)\,;
    \end{align*}
    the last equality holds since $\sqrt{npq}\rightarrow\infty$ and $nq^4\rightarrow0$. This proves that
    \begin{align*}
        g_n(t)=\big(1+o(1)\big)\exp\left(-\frac{t^2}{2}+o(1)\right)+o(1)=e^{-\frac{t^2}{2}}+o(1)\,,
    \end{align*}
    which concludes the proof of the proposition.
\end{proof}

\subsection{Central limit theorem for the height of Mallows trees}\label{sec:CLTHeight}

Before proving Theorem~\ref{thm:CLT}, we use results from previous sections to prove Proposition~\ref{prop:convRenSubtree}.

\begin{proof}[Proof of Proposition~\ref{prop:convRenSubtree}]
    Let $(q_n)_{n\geq0}$ and $m=m(n)$ be defined as in the statement of the proposition. We want to prove that the sequence of random variables
    \begin{align*}
        \left(\frac{h\Big(T^B_n\big(\overline{1}^{R^B_m+1}\big)\Big)-c^*\log n}{\sqrt{\log n}}\right)_{n\geq 2}
    \end{align*}
    is tight. In order to do so, we will prove that, for any sequence $(\gamma_n)_{n\geq0}$ such that $\gamma_n=\omega\left(\sqrt{\log n}\right)$, we have
    \begin{align}
        \mathbb{P}\bigg(\left|h\Big(T^B_n\big(\overline{1}^{R^B_m+1}\big)\Big)-c^*\log n\right|\geq\gamma_n\bigg)\longrightarrow0\,.\label{eq:TB1DisTight}
    \end{align}
    Let $(\gamma_n)_{n\geq0}$ be a sequence such that $\gamma_n=\omega\left(\sqrt{\log n}\right)$, and assume without loss of generality that $\gamma_n=o(\log n)$.
    
    Consider a sequence $(\beta_n)_{n\geq0}$ such that $\beta_n=\omega\left(\sqrt{\log n}\right)$, and that $\beta_n=o(\gamma_n)$; in other words, $\sqrt{\log n}\ll\beta_n\ll\gamma_n\ll\log n$. By applying Proposition~\ref{prop:BoundsOnN}, we know that
    \begin{align*}
        \mathbb{P}\Big(e^{-\beta_n}\leq\Big|T^B_n\big(\overline{1}^{R^B_m+1}\big)\Big|(1-q_n)\leq\beta_n\Big)\longrightarrow1\,.
    \end{align*}
    This implies that
    \begin{align*}
        &\mathbb{P}\bigg(\left|h\Big(T^B_n\big(\overline{1}^{R^B_m+1}\big)\Big)-c^*\log n\right|\geq\gamma_n\bigg)\\
        &\hspace{1cm}=\mathbb{P}\bigg(\left|h\Big(T^B_n\big(\overline{1}^{R^B_m+1}\big)\Big)-c^*\log n\right|\geq\gamma_n\,\,\bigg|\,\,e^{-\beta_n}\leq\big|T^B_n\big(\overline{1}^{R^B_m+1}\big)\big|(1-q_n)\leq\beta_n\bigg)+o(1)\,.
    \end{align*}
    
    Recall that $\log\big(n(1-q_n)\big)=O\left(\sqrt{\log n}\right)$. Since $\sqrt{\log n}\ll\beta_n\ll\gamma_n\ll\log n$, this implies that, for any sequence $(s_n)_{n\geq0}$ such that $e^{-\beta_n}\leq s_n(1-q_n)\leq\beta_n$, we have $\log s_n=\log n+O(\beta_n)\sim\log n$. It follows that
    \begin{align*}
        s_n(1-q_n)\leq\beta_n=o(\log s_n)
    \end{align*}
    and that
    \begin{align*}
        \frac{\gamma_n}{s_n(1-q_n)\vee\sqrt{\log n}}\geq\frac{\gamma_n}{\beta_n\vee\sqrt{\log n}}=\omega(1)\,.
    \end{align*}
    This corresponds to the assumptions of Proposition~\ref{prop:convProbSmall}, and we henceforth know that
    \begin{align*}
        \mathbb{P}\Big(\Big|h(T_{s_n,q_n})-c^*\log s_n\Big|\geq\gamma_n\,\Big)\longrightarrow0\,.
    \end{align*}
    Moreover, since $\log s_n=\log n+O(\beta_n)=\log n+o(\gamma_n)$, it follows that
    \begin{align*}
        \mathbb{P}\Big(\Big|h(T_{s_n,q_n})-c^*\log n\Big|\geq\gamma_n\,\Big)\longrightarrow0\,.
    \end{align*}
    Since, conditioned on having size $s_n$, $T^B_n\big(\overline{1}^{R^B_m+1}\big)$ is distributed as $T_{s_n,q_n}$, this implies that
    \begin{align*}
        \mathbb{P}\bigg(\left|h\Big(T^B_n\big(\overline{1}^{R^B_m+1}\big)\Big)-c^*\log n\right|\geq\gamma_n\,\,\bigg|\,\,e^{-\beta_n}\leq\big|T^B_n\big(\overline{1}^{R^B_m+1}\big)\big|(1-q_n)\leq\beta_n\bigg)\longrightarrow0\,,
    \end{align*}
    which proves that (\ref{eq:TB1DisTight}) holds, and concludes the proof of the proposition.
\end{proof}

With the previous results, we can now prove Theorem~\ref{thm:CLT}.

\begin{proof}[Proof of Theorem~\ref{thm:CLT}]
    We will prove that, for all $t\in\mathbb{R}$, we have
    \begin{align}
        \lim_{n\rightarrow\infty}\mathbb{P}\left(\frac{h(T^B_n)-n(1-q_n)-c^*\log\big((1-q_n)^{-1}\big)}{\sqrt{n(1-q_n)q_n}}\leq t\right)=\Phi(t)\,,\label{eq:CLT}
    \end{align}
    where $\Phi$ is the cumulative density function of the $\textsc{Normal}(0,1)$ distribution. By considering subsequences if necessary, we can assume either that $n(1-q_n)=\omega\big(\log^2n\big)$, or that $n(1-q_n)=O\big(\log^2n\big)$. We now fix $t\in\mathbb{R}$ and prove that (\ref{eq:CLT}) holds by dividing the proof into the two previous cases.
    
    Assume first that $n(1-q_n)=\omega(\log^2n)$. Since $nq_n=\omega(1)$, this implies that
    \begin{align*}
        \frac{\sqrt{n(1-q_n)q_n}}{\log\big((1-q_n)^{-1}\big)}\longrightarrow\infty\,.
    \end{align*}
    Let $(\gamma_n)_{n\geq0}$ be any sequence that converges to infinity such that $\gamma_n=\omega\big(\log\big((1-q_n)^{-1}\big)\big)$ and $\gamma_n=o\big(\sqrt{n(1-q_n)q_n)}\big)$, and define
    \begin{align*}
        E_n=\Big\{h(T^B_n)-R^B_n\geq\gamma_n\Big\}\,.
    \end{align*}
    Recall the upper bound from (\ref{eq:htnq_upper}):
    \begin{align*}
        h(T^B_n)-R^B_n&\leq1+\max_{k\geq0}\Big\{h\Big(T^B(\overline{1}^k\overline{0})\Big)-k\Big\}\,.
    \end{align*}
    Using this bound, we have
    \begin{align*}
        \mathbb{P}(E_n)&\leq\mathbb{P}\left(1+\max_{k\geq0}\Big\{h\Big(T^B(\overline{1}^k\overline{0})\Big)-k\Big\}\geq\gamma_n\right)\,,
    \end{align*}
    and the right hand side converges to $0$ when $n\rightarrow\infty$, thanks to Proposition~\ref{prop:boundsLeftSubtree} applied with $\xi_n=\gamma_n-1-c^*\log\big((1-q_n)^{-1}\big)-M\sqrt{\log(1-q_n)^{-1}}$; this tends to infinite since $\gamma_n=\omega\big(\log\big((1-q_n)^{-1}\big)\big)$. It follows that
    \begin{align}
        \mathbb{P}\left(\frac{h(T^B_n)-n(1-q_n)}{\sqrt{n(1-q_n)q_n}}\leq t\right)&=\mathbb{P}\left(\frac{h(T^B_n)-n(1-q_n)}{\sqrt{n(1-q_n)q_n}}\leq t,E_n^c\right)+o(1)\,.\label{eq:CLTFirstCase}
    \end{align}
    To bound the right hand side, on one hand, using that $R^B_n\leq h(T^B_n)$, we have
    \begin{align*}
        \mathbb{P}\left(\frac{h(T^B_n)-n(1-q_n)}{\sqrt{n(1-q_n)q_n}}\leq t,E_n^c\right)\leq\mathbb{P}\left(\frac{R^B_n-n(1-q_n)}{\sqrt{n(1-q_n)q_n}}\leq t\right)\,,
    \end{align*}
    and by applying Proposition~\ref{prop:CLTRD}, the upper bound converges to $\Phi(t)$; we are using here that $\sqrt{n(1-q_n)q_n}=\omega\big(\log\big((1-q_n)^{-1}\big)\big)$, so the $\log\big((1-q_n)^{-1}\big)$ term in the numerator of Proposition~\ref{prop:CLTRD} is asympotically negligible and can be dropped. On the other hand, by the definition of $E_n$, we have $h(T^B_n)<R^B_n+\gamma_n$ on $E^c_n$, and it follows that
    \begin{align*}
        \mathbb{P}\left(\frac{h(T^B_n)-n(1-q_n)}{\sqrt{n(1-q_n)q_n}}\leq t,E_n^c\right)&\geq\mathbb{P}\left(\frac{R^B_n+\gamma_n-n(1-q_n)}{\sqrt{n(1-q_n)q_n}}\leq t,E_n^c\right)\\
        &\geq\mathbb{P}\left(\frac{R^B_n-n(1-q_n)}{\sqrt{n(1-q_n)q_n}}\leq t-\frac{\gamma_n}{\sqrt{n(1-q_n)q_n}}\right)-\mathbb{P}(E_n)
    \end{align*}
    Since $\frac{\gamma_n}{\sqrt{n(1-q_n)q_n}}$ and $\mathbb{P}(E_n)$ converge to $0$, this lower bound also converges to $\Phi(t)$, again thanks to Proposition~\ref{prop:CLTRD}. Combing the last two results with (\ref{eq:CLTFirstCase}) and again using that $\log\big((1-q_n)^{-1}\big)=o\big(\sqrt{n(1-q_n)q_n}\big)$, it follows that
    \begin{align*}
        \mathbb{P}\left(\frac{h(T^B_n)-n(1-q_n)-c^*\log\big((1-q_n)^{-1}\big)}{\sqrt{n(1-q_n)q_n}}\leq t,E_n^c\right)&=\mathbb{P}\left(\frac{h(T^B_n)-n(1-q_n)}{\sqrt{n(1-q_n)q_n}}\leq t,E_n^c\right)+o(1)\\
        &=\Phi(t)+o(1)\,,
    \end{align*}
    which concludes the proof of Theorem~\ref{thm:CLT} in the case where $n(1-q_n)=\omega\big(\log^2n\big)$.
    
    Assume now that $n(1-q_n)=O(\log^2n)$ and that $n(1-q_n)=\omega(\log n)$. Note that this implies that $\log\big(n(1-q_n)\big)=O(\log\log n)=O\big(\sqrt{\log n}\big)$ and then $\log n=\log\big((1-q_n)^{-1}\big)+O\big(\sqrt{\log n}\big)$. Let $m=m(n)=\min\big\{\ell\geq0:\ell(1-q_n)+\log\ell\geq n(1-q_n)\big\}$ and $(\gamma_n)_{n\geq0}$ be a sequence such that $\gamma_n=\omega\left(\sqrt{\log n}\right)$ and $\gamma_n=o\left(\sqrt{n(1-q_n)q_n}\right)$, which is possible since $n(1-q_n)q_n=\omega(\log n)$. Define the event
    \begin{align*}
        F_n=\Big\{\big|h(T^B_n)-R^B_m-1-c^*\log\big((1-q_n)^{-1}\big)\big|\geq\gamma_n\Big\}\,.
    \end{align*}
    Using both bounds of (\ref{eq:third_height_bd}), we have that
    \begin{align*}
        \mathbb{P}(F_n)&=\mathbb{P}\Big(h(T^B_n)-R^B_m-1\geq c^*\log\big((1-q_n)^{-1}\big)+\gamma_n\Big)+\mathbb{P}\Big(h(T^B_n)-R^B_m-1\leq c^*\log\big((1-q_n)^{-1}\big)-\gamma_n\Big)\\
        &\leq\mathbb{P}\left(\max\left\{\sup_{k\geq0}\Big\{h\Big(T^B\big(\overline{1}^k\overline{0}\big)\Big)-k\Big\},h\Big(T^B_n\big(\overline{1}^{R^B_m+1}\big)\Big)\right\}\geq c^*\log\big((1-q_n)^{-1}\big)+\gamma_n\right)\\
        &\hspace{0.5cm}+\mathbb{P}\left(h\Big(T^B_n\big(\overline{1}^{R^B_m+1}\big)\Big)\leq c^*\log\big((1-q_n)^{-1}\big)-\gamma_n\right)\,.
    \end{align*}
    By applying Proposition~\ref{prop:convRenSubtree}, which states that $h\big(T^B_n\big(\overline{1}^{R^B_m+1}\big)\big)=c^*\log n+O_\mathbb{P}\left(\sqrt{\log n}\right)$, and Proposition~\ref{prop:boundsLeftSubtree}, which states that $\sup_{k\geq0}\big\{h\big(T^B\big(\overline{1}^k\overline{0}\big)\big)-k\big\}\leq c^*\log n+O_\mathbb{P}\left(\sqrt{\log n}\right)$, we obtain that $\mathbb{P}(F_n)\rightarrow0$.
    
    Separating (\ref{eq:CLT}) according to $F_n$ and $F_n^c$ as previously, we now obtain that
    \begin{align*}
        &\mathbb{P}\left(\frac{h(T^B_n)-n(1-q_n)-c^*\log\big((1-q_n)^{-1}\big)}{\sqrt{n(1-q_n)q_n}}\leq t\right)\\
        &\hspace{1cm}=\mathbb{P}\left(\frac{h(T^B_n)-n(1-q_n)-c^*\log\big((1-q_n)^{-1}\big)}{\sqrt{n(1-q_n)q_n}}\leq t,F_n^c\right)+o(1)\,.
    \end{align*}
    Note that, by definition of $m$, we have $m(1-q_n)+\log m=n(1-q_n)+O(1)$. Since $n(1-q_n)=\omega(\log n)$, it follows that $m\sim n$ and so $\log m=\log n+O(1)$. This implies that $m(1-q_n)=\omega(\log m)$ and that $mq_n=\omega(1)$, and by Proposition~\ref{prop:CLTRD}, we obtain
    \begin{align*}
        \mathbb{P}\left(\frac{R^B_m-m(1-q_n)-\log\big((1-q_n)^{-1}\big)}{\sqrt{m(1-q_n)q_n}}\leq t\right)\longrightarrow\Phi(t)\,.
    \end{align*}
    The previous identities also imply that $m(1-q_n)q_n\sim n(1-q_n)q_n$ and, since $\log\big(n(1-q_n)\big)=O(\log\log n)$, that
    \begin{align*}
        m(1-q_n)+\log\big((1-q_n)^{-1}\big)=n(1-q_n)+O(\log\log n)=n(1-q_n)+o\big(\sqrt{n(1-q_n)q_n}\big)\,.
    \end{align*}
    It follows that
    \begin{align*}
        \mathbb{P}\left(\frac{R^B_m-n(1-q_n)}{\sqrt{n(1-q_n)q_n}}\leq t\right)\longrightarrow\Phi(t)\,.
    \end{align*}
    
    Now, use the definition of $F_n$ and the previous asymptotic result to obtain that, on one hand
    \begin{align*}
        &\mathbb{P}\left(\frac{h(T^B_n)-n(1-q_n)-c^*\log\big((1-q_n)^{-1}\big)}{\sqrt{n(1-q_n)q_n}}\leq t,F_n^c\right)\\
        &\hspace{1cm}\leq\mathbb{P}\left(\frac{R^B_n+1-n(1-q_n)-\gamma_n}{\sqrt{n(1-q_n)q_n}}\leq t,F_n^c\right)\\
        &\hspace{1cm}\leq\mathbb{P}\left(\frac{R^B_n-n(1-q_n)}{\sqrt{n(1-q_n)q_n}}\leq t+\frac{\gamma_n-1}{\sqrt{n(1-q_n)q_n}}\right)\\
        &\hspace{1cm}=\Phi(t)+o(1)\,,
    \end{align*}
    and on the other hand
    \begin{align*}
        &\mathbb{P}\left(\frac{h(T^B_n)-n(1-q_n)-c^*\log\big((1-q_n)^{-1}\big)}{\sqrt{n(1-q_n)q_n}}\leq t,F_n^c\right)\\
        &\hspace{1cm}\geq\mathbb{P}\left(\frac{R^B_n+1-n(1-q_n)+\gamma_n}{\sqrt{n(1-q_n)q_n}}\leq t,F_n^c\right)\\
        &\hspace{1cm}\geq\mathbb{P}\left(\frac{R^B_n-n(1-q_n)}{\sqrt{n(1-q_n)q_n}}\leq t-\frac{\gamma_n+1}{\sqrt{n(1-q_n)q_n}}\right)-\mathbb{P}(F_n)\\
        &\hspace{1cm}=\Phi(t)+o(1)\,.
    \end{align*}
    This proves that
    \begin{align*}
        \mathbb{P}\left(\frac{h(T^B_n)-n(1-q_n)-c^*\log\big((1-q_n)^{-1}\big)}{\sqrt{n(1-q_n)q_n}}\leq t,F_n^c\right)&=\Phi(t)+o(1)\,,
    \end{align*}
    which concludes the proof of Theorem~\ref{thm:CLT} in the second and last case.
\end{proof}

\subsection{Poisson fluctuations for the height}\label{sec:Poisson}

We conclude this section with the proof of Theorem~\ref{thm:Poisson}.

\begin{proof}[Proof of Theorem~\ref{thm:Poisson}]
    We will prove that, when $nq_n\rightarrow\lambda\in[0,\infty)$, then
    \begin{align*}
        n-1-h(T^B_n)\overset{d}{\longrightarrow}\textsc{Poisson}(\lambda)\,.
    \end{align*}
    Start by considering the event
    \begin{align*}
        E_n=\big\{h(T^B_n)>R^B_n\big\}\,.
    \end{align*}
    Since $R^B_n\leq h(T^B_n)$, this means that $E_n^c=\big\{R^B_n=h(T^B_n)\big\}$. Now, for the height of the whole tree to be larger than the right depth, there must be a non-empty left subtree $T^B_n(\overline{1}^k\overline{0})$ for some $0\leq k\leq R^B_n$. Moreover, if this non-empty left subtree is not $T^B_n(\overline{1}^{R^B_n}\overline{0})$ or $T^B_n(\overline{1}^{R^B_n-1}\overline{0})$, then its size has to be larger than $2$. This implies that
    \begin{align*}
        E_n\subseteq\Big\{\big|T^B_n(\overline{1}^{R^B_n}\overline{0})\big|\geq1\Big\}\cup\Big\{\big|T^B_n(\overline{1}^{R^B_n-1}\overline{0})\big|\geq1\Big\}\cup\bigcup_{0\leq k\leq R^B_n-2}\Big\{\big|T^B_n(\overline{1}^k\overline{0})\big|\geq2\Big\}\,.
    \end{align*}
    Recall from Lemma~\ref{lem:LeftSubtreesDistribution} that the trees $\big(\big|T^B(\overline{1}^k\overline{0})\big|\big)_{k\geq0}$ are all independent, $\textsc{Geometric}(1-q_n)$ distributed random variable. Using that $R^B_n\leq n-1$ and that $T^B_n(\overline{1}^k\overline{0})\subseteq T^B(\overline{1}^k\overline{0})$, this implies that
    \begin{align*}
        \mathbb{P}(E_n)\leq2\mathbb{P}\Big(\big|T^B(\overline{0})\big|\geq1\Big)+(n-2)\mathbb{P}\Big(\big|T^B(\overline{0})\big|\geq2\Big)=2q_n+(n-2)q_n^2=o(1)\,.
    \end{align*}
    This proves that $\mathbb{P}\big(h(T^B_n)=R^B_n\big)=1-o(1)$, when $nq_n=O(1)$. We will now prove that
    \begin{align*}
        n-1-R^B_n\overset{d}{\longrightarrow}\textsc{Poisson}(\lambda)\,,
    \end{align*}
    by showing that the characteristic function of $n-1-R^B_n$ converges to that of a $\textsc{Poisson}(\lambda)$ random variable.
    
    Since $q_n=o(1)$, we have that $\frac{1-q_n}{1-q_n^k}=1-q_n+o(q_n)$ for all $k\geq2$, where the small-$o$ term can be chosen to be independent of $k$. Consider now the characteristic function for $R^B_n$ from Proposition~\ref{prop:BivariateGenFun} to obtain
    \begin{align*}
        \mathbb{E}\left[e^{it(n-1-R^B_n)}\right]&=e^{it(n-1)}\prod_{1<k\leq n}\left(1+(e^{-it}-1)\frac{1-q_n}{1-q_n^k}\right)\\
        &=e^{it(n-1)}\Big(1+(e^{-it}-1)\big(1-q_n+o(q_n)\big)\Big)^{n-1}\\
        &=\Big(1+\big(e^{it}-1+o(1)\big)q_n\Big)^{n-1}\,.
    \end{align*}
    Since $nq_n\rightarrow\lambda$, it follows that
    \begin{align*}
        \mathbb{E}\left[e^{it(n-1-R^B_n)}\right]\longrightarrow e^{\lambda(e^{it}-1)}\,;
    \end{align*}
    this concludes the proof of the theorem.
\end{proof}

\section{Further questions}

This paper studies Mallows trees and proved some of the properties of its height. However, several related questions remain open. We discuss some of the possible further studies below. 

\begin{itemize}
    \item Between Theorem~\ref{thm:CLT} and Theorem~\ref{thm:Poisson}, we have a good understand of the distributional limit of $h(T_{n,q_n})$ when $n(1-q_n)/\log n\rightarrow\infty$. Moreover, we know from the results of~\cite{drmota2003analytic,reed2003height} that when $q_n\equiv 1$, the central limit theorem does not hold anymore, and the variance of the height is  $\Theta(1)$. This means that there exists a transition between the regime $n(1-q_n)/\log n\rightarrow\infty$ and $q_n=1$ where the central limit theorem of Theorem~\ref{thm:CLT} stops holding and moves to a more concentrated process with finite variance. It is natural to ask where this transition occurs; it is not clear to us whether the condition $n(1-q_n)/\log n\to \infty$ is necessary in order for a Gaussian central limit theorem to hold, or what other distributional limits are possible for sequences $(q_n)_{n \ge 0}$ with $\limsup_{n \to \infty} n(1-q_n)/\log n<\infty$.
    \item In this paper, we studied the height of $T_{n,q_n}$ in large part by relating it to the length of the rightmost path in $T_{n,q_n}$; 
    we did so by bounding from above the height of the left subtrees $T_{n,q_n}(\overline{1}^k\overline{0})$, for $k \ge 0$. 
    The intrinsic properties of the left subtrees both for finite $n$ and in the $n \to \infty$ limit, deserve further exploration in our view, and we next list a couple of specific questions of interest.
    
We know that for any fixed $q \in [0,1]$ and $k \in \N$, the trees $(T_{n,q}(\overline{1}^k\overline{0}))_{n \ge 0}$ are stochastically increasing in $n$. Moreover, working in the infinite $b$-model, we have 
    $T^{B(q)}_{n}(\overline{1}^k\overline{0})=T^{B(q)}(\overline{1}^k\overline{0})$ for all $n$ sufficiently large (recall that, by definition, $T^{B(q)}=\lim_{n \to \infty} T^{B(q)}_n$). It would be interesting to understand this filling process, i.e., to study the behaviour of 
    \begin{align*}
        D^{B(q)}_{n}=\max\Big\{0\leq k\leq R^{B(q)}_n:T^{B(q)}_{n}(\overline{1}^k\overline{0})=T^{B(q)}(\overline{1}^k\overline{0})\Big\}\,,
    \end{align*}
    which corresponds to the depth until which all trees are filled, as both $n$ and $q$ vary, and of
    \begin{align*}
        S^{B(q)}_{n,k}=\frac{\big|T^{B(q)}_n(\overline{1}^k\overline{0})\big|}{\big|T^{B(q)}(\overline{1}^{k}\overline{0})\big|}\,,
    \end{align*}
    which corresponds to the proportion of the subtree $T^{B(q)}(\overline{1}^{k}\overline{0})$ already present at time $n$; here both $q$ and $k$ may depend on $n$.
    \item Another direction of studies regarding the left subtrees is to consider the structure of the tree $T^{B(q)}(\overline{0})$, especially as $q\rightarrow1$. With the results from this paper, it is fairly straightforward to verify that, in the case when $q\rightarrow1$, the height of $T^{B(q)}(\overline{0})$ is $\big(c^*+o_\mathbb{P}(1)\big)\log(1/(1-q))$. It would be interesting to understand the lower order corrections to this height. We expect the height of $T^{B(q)}(\overline{0})$ to have bounded variance as $q\rightarrow1$, and to converge in distribution after recentering, at least along subsequences. It could also be interesting to characterize the filling levels (as in \cite{devroye1986note}) or the total path length of this tree.
    \item Corollary~\ref{cor:TnqIncreasing} says that $T_{n,q}$ is stochastically increasing in $n$ when $q$ is fixed. Computations for small values of $n$ suggest that $T_{n,q}$ is also stochastically decreasing in $q$. This would be interesting if true and would also provide a useful comparison tool, which would simplify some of the arguments of the current work (in particular Proposition~\ref{prop:couplingMallowsRBST}, in which we could simply chose $m=n$).
\end{itemize}

\addtocontents{toc}{\SkipTocEntry}
\section*{Acknowledgements}
During the preparation of this research, LAB was supported by an NSERC Discovery Grant and an FRQNT Team Grant, and BC was supported by an ISM Graduate Scholarship. BC also wishes to thank Ms.~Legrand for supporting and encouraging his interest in mathematics.

\appendix

\section{Moments of the right depth}\label{app:MomentsRD}

In this appendix, we prove Proposition~\ref{prop:MomentsOfRB} and deduce Fact~\ref{fact:ExpRnmu} from it. Recall the definition of $\mu_\alpha$ from Section~\ref{sec:usefulFunction}:
\begin{align*}
    \mu_\alpha(n,q)=\sum_{1\leq k\leq n}\left(\frac{1-q}{1-q^k}\right)^\alpha\,.
\end{align*}
For any integer $1\leq\beta<n$, we define
\begin{align*}
    \nu_\beta(n,q)=\sum_{\{1<k_1\neq\cdots\neq k_\beta\leq n\}}\prod_{\{1\leq i\leq\beta\}}\frac{1-q}{1-q^{k_i}}\,.
\end{align*}
This function is useful in computing the moments of $R^B_n$ as stated in the following lemma.

\begin{lemma}\label{lem:momentRBandNu}
    Let $n\geq1$, $q\in[0,1)$, and $B=(B_{i,j})_{i,j\geq1}$ have independent $\textsc{Bernoulli}(1-q)$ entries. Then, for all integer $\alpha\geq1$ we have
    \begin{align*}
        \mathbb{E}\left[\big(R^B_n\big)^\alpha\right]=\sum_{1\leq\beta\leq\alpha\wedge(n-1)}\genfrac{\{}{\}}{0pt}{0}{\alpha}{\beta}\times\nu_\beta(n,q)\,,
    \end{align*}
    where $\big(\genfrac{\{}{\}}{0pt}{1}{\alpha}{\beta}\big)_{\alpha,\beta\geq1}$ are Stirling numbers of the second kind.
\end{lemma}

\begin{proof}
    For $n=1$, $R^B_n=0$ and the sum on the right is empty. Assume now that $n\geq2$. We will prove by induction on $\alpha\geq1$ the following slightly more general result: for all $t\in\mathbb{R}$, we have
    \begin{align}
        &\mathbb{E}\left[\big(R^B_n\big)^\alpha e^{tR^B_n}\right]\notag\\
        &\hspace{1cm}=\sum_{1\leq\beta\leq\alpha\wedge(n-1)}\genfrac{\{}{\}}{0pt}{0}{\alpha}{\beta}\sum_{1<k_1\neq\cdots\neq k_\beta\leq n}\prod_{1\leq i\leq\beta}\frac{(1-q)e^t}{1-q^{k_i}}\prod_{k\neq k_1,\ldots,k_\beta}\left(1+(e^t-1)\frac{1-q}{1-q^k}\right)\label{eq:formulaMomentRB}\,.
    \end{align}
    The formula of the lemma follows by taking $t=0$.
    
    For $\alpha=1$, by considering the derivative of the moment generating function of $R^B_n$ from Proposition~\ref{prop:BivariateGenFun} obtained by taking $x=e^t$ and $y=1$, we have
    \begin{align*}
        \mathbb{E}\left[R^B_ne^{tR^B_n}\right]=\sum_{1<k\leq n}\frac{(1-q)e^t}{1-q^k}\prod_{1<\ell\leq n:\ell\neq k}\left(1+(e^t-1)\frac{1-q}{1-q^\ell}\right)\,,
    \end{align*}
    which exactly corresponds to the right hand side of (\ref{eq:formulaMomentRB}) when $\alpha=1$.
    
    Assume now that (\ref{eq:formulaMomentRB}) holds for some $\alpha\geq1$. We take the derivative in $t$ on both sides of (\ref{eq:formulaMomentRB}); the resulting analysis then depends on whether $\alpha<n-1$ or $\alpha\geq n-1$.
    
    \begin{description}
        \item[First case] $\alpha<n-1$. In this case, the last product on the right hand side is never empty and the derivative is
        \begin{align*}
            &\mathbb{E}\left[\big(R^B_n\big)^{\alpha+1}e^{tR^B_n}\right]\\
            &\hspace{0.5cm}=\sum_{1\leq\beta\leq\alpha}\genfrac{\{}{\}}{0pt}{0}{\alpha}{\beta}\sum_{1<k_1\neq\cdots\neq k_\beta\leq n}\beta\prod_{1\leq i\leq\beta}\frac{(1-q)e^t}{1-q^{k_i}}\prod_{k\neq k_1,\ldots,k_\beta}\left(1+(e^t-1)\frac{1-q}{1-q^k}\right)\\
            &\hspace{1cm}+\sum_{1\leq\beta\leq\alpha}\genfrac{\{}{\}}{0pt}{0}{\alpha}{\beta}\sum_{1<k_1\neq\cdots\neq k_{\beta+1}\leq n}\prod_{1\leq i\leq\beta+1}\frac{(1-q)e^t}{1-q^{k_i}}\prod_{k\neq k_1,\ldots,k_{\beta+1}}\left(1+(e^t-1)\frac{1-q}{1-q^k}\right)\,.
        \end{align*}
        Performing the change of variables $\beta\mapsto\beta-1$ in the second sum, then regrouping the two sums over the terms $2\leq\beta\leq\alpha$, we obtain
        \begin{align*}
            &\mathbb{E}\left[\big(R^B_n\big)^{\alpha+1}e^{tR^B_n}\right]\\
            &\hspace{0.5cm}=\sum_{2\leq\beta\leq\alpha}\left(\beta\genfrac{\{}{\}}{0pt}{0}{\alpha}{\beta}+\genfrac{\{}{\}}{0pt}{0}{\alpha}{\beta-1}\right)\sum_{1<k_1\neq\cdots\neq k_\beta\leq n}\prod_{1\leq i\leq\beta}\frac{(1-q)e^t}{1-q^{k_i}}\prod_{k\neq k_1,\ldots,k_\beta}\left(1+(e^t-1)\frac{1-q}{1-q^k}\right)\\
            &\hspace{1cm}+\genfrac{\{}{\}}{0pt}{0}{\alpha}{1}\sum_{1<k_1\leq n}\frac{(1-q)e^t}{1-q^{k_1}}\prod_{k\neq k_1}\left(1+(e^t-1)\frac{1-q}{1-q^k}\right)\\
            &\hspace{1cm}+\genfrac{\{}{\}}{0pt}{0}{\alpha}{\alpha}\sum_{1<k_1\neq\cdots\neq k_{\alpha+1}\leq n}\prod_{1\leq i\leq\alpha+1}\frac{(1-q)e^t}{1-q^{k_i}}\prod_{k\neq k_1,\ldots,k_{\alpha+1}}\left(1+(e^t-1)\frac{1-q}{1-q^k}\right)\,,
        \end{align*}
        which proves the desired formula since, by definition, we have $\beta\genfrac{\{}{\}}{0pt}{1}{\alpha}{\beta}+\genfrac{\{}{\}}{0pt}{1}{\alpha}{\beta-1}=\genfrac{\{}{\}}{0pt}{1}{\alpha+1}{\beta}$, as well as $\genfrac{\{}{\}}{0pt}{1}{\alpha}{1}=\genfrac{\{}{\}}{0pt}{1}{\alpha+1}{1}=1$ and $\genfrac{\{}{\}}{0pt}{1}{\alpha}{\alpha}=\genfrac{\{}{\}}{0pt}{1}{\alpha+1}{\alpha+1}$.
        \item[Second case] $\alpha\geq n-1$. In this case, the last product is empty when $\beta=n-1$. From this observation, the derivative becomes
        \begin{align*}
            &\mathbb{E}\left[\big(R^B_n\big)^{\alpha+1}e^{tR^B_n}\right]\\
            &\hspace{0.5cm}=\sum_{1\leq\beta\leq n-1}\genfrac{\{}{\}}{0pt}{0}{\alpha}{\beta}\sum_{1<k_1\neq\cdots\neq k_\beta\leq n}\beta\prod_{1\leq i\leq\beta}\frac{(1-q)e^t}{1-q^{k_i}}\prod_{k\neq k_1,\ldots,k_\beta}\left(1+(e^t-1)\frac{1-q}{1-q^k}\right)\\
            &\hspace{1cm}+\sum_{1\leq\beta\leq n-2}\genfrac{\{}{\}}{0pt}{0}{\alpha}{\beta}\sum_{1<k_1\neq\cdots\neq k_{\beta+1}\leq n}\prod_{1\leq i\leq\beta+1}\frac{(1-q)e^t}{1-q^{k_i}}\prod_{k\neq k_1,\ldots,k_{\beta+1}}\left(1+(e^t-1)\frac{1-q}{1-q^k}\right)\\
            &\hspace{0.5cm}=\sum_{2\leq\beta\leq n-1}\left(\beta\genfrac{\{}{\}}{0pt}{0}{\alpha}{\beta}+\genfrac{\{}{\}}{0pt}{0}{\alpha}{\beta-1}\right)\sum_{1<k_1\neq\cdots\neq k_\beta\leq n}\prod_{1\leq i\leq\beta}\frac{(1-q)e^t}{1-q^{k_i}}\prod_{k\neq k_1,\ldots,k_\beta}\left(1+(e^t-1)\frac{1-q}{1-q^k}\right)\\
            &\hspace{1cm}+\genfrac{\{}{\}}{0pt}{0}{\alpha}{1}\sum_{1<k_1\leq n}\frac{(1-q)e^t}{1-q^{k_1}}\prod_{k\neq k_1}\left(1+(e^t-1)\frac{1-q}{1-q^k}\right)\,,
        \end{align*}
        where the second equality follows by the same argument as in the first case. Applying the same identities for $\genfrac{\{}{\}}{0pt}{1}{\alpha}{\beta}$ as before, the desired formula holds for $\alpha+1$; this completes the induction.
    \end{description}
\end{proof}

The previous lemma gives the relation between the moments of $R^B_n$ and the functions $\nu_\beta$. Before proving the relation between $\nu_\beta$ and $\mu_\alpha$, we state and prove a useful formula.

\begin{lemma}\label{lem:ForumlaX}
    Fix two positive integers $r,n\geq1$, and for each $1\leq k\leq r$, choose $x^{(k)}=\big(x^{(k)}_i\big)_{1\leq i\leq n}\in\R^n$. For $1\leq k\leq r$, write $P^r_k$ for the set of partitions of $[r]$ into $k$ non-empty subsets. Then
    \begin{align*}
        \sum_{i_1\neq\cdots\neq i_r}\prod_{1\leq k\leq r}x^{(k)}_{i_k}=\sum_{1\leq k\leq r}\sum_{(A_1,\ldots,A_k)\in P^r_k}(-1)^{r+k}\prod_{1\leq j\leq k}\left(\big(|A_j|-1\big)!\sum_{i\in[n]}\prod_{a\in A_j}x^{(a)}_i\right)\,.
    \end{align*}
\end{lemma}

\begin{proof}
    We prove the result by induction on $r\geq1$, the case $r=1$ being obvious. Assuming it is true for some $r\geq1$, to prove that the formula holds for $r+1$, rewrite the sum over $i_1\neq\cdots\neq i_{r+1}$ by adding and subtracting the sum over $i_{r+1}$ such that $i_{r+1}=i_m$ for some $1\leq m\leq r$ to obtain
    \begin{align}
        \sum_{i_1\neq\cdots\neq i_{r+1}}\prod_{1\leq k\leq r+1}x^{(k)}_{i_k}&=\left(\sum_{i_1\neq\cdots\neq i_r}\prod_{1\leq k\leq r}x^{(k)}_{i_k}\right)\sum_{i\in[n]}x^{(r+1)}_i-\sum_{1\leq m\leq r}\sum_{i_1\neq\cdots\neq i_r}x^{(r+1)}_{i_m}\prod_{1\leq k\leq r}x^{(k)}_{i_k}\,.\label{eq:formulaX}
    \end{align}
    For the first term on the right, using the induction hypothesis, we know that
    \begin{align*}
        \sum_{i_1\neq\cdots\neq i_n}\prod_{1\leq k\leq r}x^{(k)}_{i_k}=\sum_{1\leq k\leq r}\sum_{(A_1,\ldots,A_k)\in P^r_k}(-1)^{r+k}\prod_{1\leq j\leq k}\left(\big(|A_j|-1\big)!\sum_{i\in[n]}\prod_{a\in A_j}x^{(a)}_i\right)\,,
    \end{align*}
    and hence
    \begin{align}
        &\left(\sum_{i_1\neq\cdots\neq i_n}\prod_{1\leq k\leq r}x^{(k)}_{i_k}\right)\sum_{i\in[n]}x^{(r+1)}_i\notag\\
        &\hspace{1cm}=\left(\sum_{1\leq k\leq r}\sum_{(A_1,\ldots,A_k)\in P^r_k}(-1)^{r+k}\prod_{1\leq j\leq k}\left(\big(|A_j|-1\big)!\sum_{i\in[n]}\prod_{a\in A_j}x^{(a)}_i\right)\right)\sum_{i\in[n]}x^{(r+1)}_i\notag\\
        &\hspace{1cm}=\sum_{1\leq k\leq r+1}\sum_{\{(A_1,\ldots,A_k)\in P^{r+1}_k:A_k=\{r+1\}\}}(-1)^{(r+1)+k}\prod_{1\leq j\leq k}\left(\big(|A_j|-1\big)!\sum_{i\in[n]}\prod_{a\in A_j}x^{(a)}_i\right)\,.\label{eq:formulaX1}
    \end{align}
    On the other hand, for all $1\leq m\leq r$, by writing $y^{(k,m)}_i=x^{(k)}_i$ if $k\neq m$ and $y^{(k,m)}_i=x^{(k)}_ix^{(r+1)}_i$ if $k=m$, we have
    \begin{align*}
        \sum_{i_1\neq\cdots\neq i_r}x^{(r+1)}_{i_m}\prod_{1\leq k\leq r}x^{(k)}_{i_k}&=\sum_{i_1\neq\cdots\neq i_r}\prod_{1\leq k\leq r}y^{(k,m)}_{i_k}\\
        &=\sum_{1\leq k\leq r}\sum_{(A_1,\ldots,A_k)\in P^r_k}(-1)^{r+k}\prod_{1\leq j\leq k}\left(\big(|A_j|-1\big)!\sum_{i\in[n]}\prod_{a\in A_j}y^{(a,m)}_i\right)\,,
    \end{align*}
    where the second line follows from the induction hypothesis. Now, note that
    \begin{align*}
        &\sum_{1\leq m\leq r}\prod_{1\leq j\leq k}\left(\big(|A_j|-1\big)!\sum_{i\in[n]}\prod_{a\in A_j}y^{(a,m)}_i\right)\\
        &\hspace{1cm}=\sum_{1\leq m\leq r}\prod_{1\leq j\leq k}\left(\big(|A_j|-1\big)!\sum_{i\in[n]}\big(x^{(r+1)}_i\big)^{\mathbbm{1}_{m\in A_j}}\prod_{a\in A_j}x^{(a)}_i\right)\\
        &\hspace{1cm}=\sum_{1\leq\ell\leq k}|A_\ell|\prod_{1\leq j\leq k}\left(\big(|A_j|-1\big)!\sum_{i\in[n]}\big(x^{(r+1)}_i\big)^{\mathbbm{1}_{j=\ell}}\prod_{a\in A_j}x^{(a)}_i\right)\,,
    \end{align*}
    which implies that
    \begin{align}
        &\sum_{1\leq m\leq r}\sum_{i_1\neq\cdots\neq i_r}x^{(r+1)}_{i_m}\prod_{1\leq k\leq r}x^{(k)}_{i_k}\notag\\
        &\hspace{1cm}=\sum_{1\leq k\leq r}\sum_{(A_1,\ldots,A_k)\in P^r_k}(-1)^{r+k}\sum_{1\leq m\leq r}\prod_{1\leq j\leq k}\left(\big(|A_j|-1\big)!\sum_{i\in[n]}\prod_{a\in A_j}y^{(a,m)}_i\right)\notag\\
        &\hspace{1cm}=-\sum_{1\leq k\leq r+1}\sum_{1\leq\ell\leq k}\sum_{\underset{r+1\in A_\ell,|A_\ell|>1}{(A_1,\ldots,A_k)\in P^{r+1}_k}}(-1)^{(r+1)+k}\prod_{1\leq j\leq k}\left(\big(|A_j|-1\big)!\sum_{i\in[n]}\prod_{a\in A_j}x^{(a)}_i\right)\,,\label{eq:formulaX2}
    \end{align}
    the last equality holding since, if $k=r+1$, then for any partition in $P^{r+1}_k$ all parts have size $1$, so the inner sum is empty.
    
    Substituting (\ref{eq:formulaX1}) and (\ref{eq:formulaX2}) into the right-hand side of (\ref{eq:formulaX}) and combining them, the inner sum becomes over all $(A_j)\in P^{r+1}_k$. This completes the inductive step.
\end{proof}

We now apply the previous lemma to establish the relation between $\nu_\beta$ and $\mu_\alpha$.

\begin{lemma}\label{lem:RelationMuNu}
    Let $n\geq0$, $q\in[0,1)$, and $\beta\geq1$. Write $S_\beta=\{(s_1,\ldots,s_\beta):s_1+2s_2+\cdots+\beta s_\beta=\beta\}$ and for $s=(s_1,\ldots,s_\beta)\in S$, let $|s|=s_1+\cdots+s_\beta$. Then, we have
    \begin{align*}
        \nu_\beta(n,q)=\beta!\sum_{s\in S_\beta}(-1)^{\beta+|s|}\prod_{1\leq i\leq\beta}\frac{\mu_i(n,q)^{s_i}}{i^{s_i}s_i!}
    \end{align*}
\end{lemma}

\begin{proof}
    Using the notation of Lemma~\ref{lem:ForumlaX}, let $r=\beta$ and $x^{(k)}_i=\frac{1-q}{1-q^i}$. Then the formula can be rewriten as
    \begin{align*}
        \sum_{i_1\neq\cdots\neq i_\beta}\frac{1-q}{1-q^{i_1}}\cdots\frac{1-q}{1-q^{i_\beta}}=\sum_{1\leq k\leq\beta}\sum_{(A_j)\in P^\beta_k}(-1)^{\beta+k}\prod_{1\leq j\leq k}\left(\big(|A_j|-1\big)!\sum_{i\in[n]}\prod_{a\in A_j}\frac{1-q}{1-q^i}\right)\,,
    \end{align*}
    which implies that
    \begin{align*}
        \nu_\beta(n,q)&=\sum_{1\leq k\leq\beta}\sum_{(A_j)\in P^\beta_k}(-1)^{\beta+k}\prod_{1\leq j\leq k}\Big(\big(|A_j|-1\big)!\mu_{|A_j|}(n,q)\Big)\,.
    \end{align*}
    
    For a partition $A=(A_j)\in P^\beta_k$, write $s_i(A)=|\{j\in[k]:|A_j|=i\}|$. Then $s(A)=\big(s_1(A),...,s_\beta(A)\big)\in S_\beta$ and $|s(A)|=s_1(A)+\cdots+s_\beta(A)=k$. Regrouping the sum over $(A_j)$, we then obtain
    \begin{align*}
        \nu_\beta(n,q)&=\sum_{s\in S_\beta}\sum_{\left\{A\in P^\beta_{|s|}:s(A)=s\right\}}(-1)^{\beta+|s|}\prod_{1\leq j\leq |s|}\Big(\big(|A_j|-1\big)!\mu_{|A_j|}(n,q)\Big)\\
        &=\sum_{s\in S_\beta}(-1)^{\beta+|s|}\prod_{1\leq i\leq\beta}\mu_i(n,q)^{s_i}\left(\sum_{\left\{A\in P^\beta_{|s|}:s(A)=s\right\}}\prod_{1\leq j\leq |s|}\big(|A_j|-1\big)!\right)\,.
    \end{align*}
    To conclude the proof, note that, for a given $s\in S_\beta$, we have
    \begin{align*}
        \sum_{\left\{A\in P^\beta_{|s|}:s(A)=s\right\}}\prod_{1\leq j\leq|s|}\big(|A_j|-1\big)!&=\Big|\Big\{\sigma\in\mathcal{S}_\beta:\textrm{$\sigma$ has $s_i$ cycles of size $i$}\Big\}\Big|\\
        &=\frac{\beta!}{\prod_{1\leq i\leq\beta}i^{s_i}s_i!}\,.
    \end{align*}
    This proves the desired formula.
\end{proof}

The proof of Proposition~\ref{prop:MomentsOfRB} is now a direct consequence of Lemma~\ref{lem:momentRBandNu} combined with Lemma~\ref{lem:RelationMuNu}. From this proposition, or by direct computation using the moment generating function of $R^B_n$ from Proposition~\ref{prop:BivariateGenFun}, we obtain that
\begin{align*}
    \mathbb{E}\left[R^B_n\right]=\mu_1(n,q)
\end{align*}
and
\begin{align*}
    \mathbb{E}\left[(R^B_n)^2\right]=\mu_1(n,q)+\mu_1(n,q)^2-\mu_2(n,q)\,,
\end{align*}
which proves Fact~\ref{fact:ExpRnmu}.

\section{Asymptotics of $\mu_\alpha$}\label{app:asymptoticMu}

In this appendix, we prove Proposition~\ref{prop:ConvOfMu} and \ref{prop:ConvOfMuAlpha}. Both these proofs will be based on the following bounds for $\mu_\alpha$.

\begin{prop}\label{prop:BoundsOnMu}
    Let $\alpha\geq1$, $n\geq1$, and $q\in[0,1)$. Then, for all $m\in[n]$, we have
    \begin{align*}
        \mu_\alpha(n,q)&\geq(n-m)(1-q)^\alpha+\sum_{1<k\leq m}\frac{1}{k^\alpha}
    \end{align*}
    and
    \begin{align*}
        \mu_\alpha(n,q)&\leq n(1-q)^\alpha+\alpha m(1-q)^\alpha q^{m}\frac{1-q^n}{(1-q^{m})^{\alpha+1}}+\left(\frac{m(1-q)}{1-q^m}\right)^\alpha\sum_{1<k\leq m}\frac{1}{k^\alpha}\,.
    \end{align*}
\end{prop}

\begin{proof}
    First write
    \begin{align}
        \mu_\alpha(n,q)&=\sum_{1<k\leq m}\left(\frac{1-q}{1-q^k}\right)^\alpha+\sum_{m<k\leq n}\left(\frac{1-q}{1-q^k}\right)^\alpha\,.\label{eq:divisionMu}
    \end{align}
    
    For the lower bound use in the first sum that $1-q^k\leq k(1-q)$ and in the second sum that $1-q^k\leq1$ to obtain
    \begin{align*}
        \mu_\alpha(n,q)&\geq\sum_{1<k\leq m}\frac{1}{k^\alpha}+\sum_{m<k\leq n}(1-q)^\alpha=\sum_{1<k\leq m}\frac{1}{k^\alpha}+(n-m)(1-q)^\alpha\,,
    \end{align*}
    which is the desired bound.
    
    For the upper bound, we start with the first term in (\ref{eq:divisionMu}). Define the function $\phi(x)=\frac{x}{1-e^{-x}}$, so that
    \begin{align*}
        \sum_{1<k\leq m}\left(\frac{1-q}{1-q^k}\right)^\alpha&=\sum_{1<k\leq m}\left(\frac{1-q}{1-e^{-k|\log q|}}\right)^\alpha=\sum_{1<k\leq m}\left(\frac{1-q}{k|\log q|}\phi\big(k|\log q|\big)\right)^\alpha\,.
    \end{align*}
    Note that $\phi$ is increasing, from which we deduce the following bound:
    \begin{align*}
        \sum_{1<k\leq m}\left(\frac{1-q}{1-q^k}\right)^\alpha&\leq\sum_{1<k\leq m}\left(\frac{1-q}{k|\log q|}\phi\big(m|\log q|\big)\right)^\alpha\\
        &=\left(\frac{1-q}{|\log q|}\phi\big(m|\log q|\big)\right)^\alpha\sum_{1<k\leq m}\frac{1}{k^\alpha}\\
        &=\left(\frac{m(1-q)}{1-q^m}\right)^\alpha\sum_{1<k\leq m}\frac{1}{k^\alpha}\,.
    \end{align*}
    which is the last term in the desired upper bound.
    
    Consider now the second term of (\ref{eq:divisionMu}). Since $k\mapsto\frac{1-q}{1-q^k}$ is decreasing in $k$, we have $\frac{1-q}{1-q^k}\leq\frac{1-q}{1-q^{\ell m}}$ for $\ell m\leq k\leq(\ell+1)m$, so
    \begin{align*}
        \sum_{m<k\leq n}\left(\frac{1-q}{1-q^k}\right)^\alpha&\leq m\sum_{1\leq\ell\leq\frac{n}{m}}\left(\frac{1-q}{1-q^{\ell m}}\right)^\alpha
    \end{align*}
    Now, rewrite and bound the last sum as follows
    \begin{align*}
        \sum_{1\leq\ell\leq\frac{n}{m}}\left(\frac{1-q}{1-q^{\ell m}}\right)^\alpha&=(1-q)^\alpha\left\lfloor\frac{n}{m}\right\rfloor+\sum_{1\leq\ell\leq\frac{n}{m}}\left[\left(\frac{1-q}{1-q^{\ell m}}\right)^\alpha-(1-q)^\alpha\right]\\
        &\leq\frac{n(1-q)^\alpha}{m}+(1-q)^\alpha\sum_{1\leq\ell\leq\frac{n}{m}}\frac{1-\big(1-q^{\ell m}\big)^\alpha}{\big(1-q^{\ell m}\big)^\alpha}\,,
    \end{align*}
    to obtain that
    \begin{align*}
        \sum_{m<k\leq n}\left(\frac{1-q}{1-q^k}\right)^\alpha&\leq m\left[\frac{n(1-q)^\alpha}{m}+(1-q)^\alpha\sum_{1\leq\ell\leq\frac{n}{m}}\frac{1-\big(1-q^{\ell m}\big)^\alpha}{\big(1-q^{\ell m}\big)^\alpha}\right]\\
        &=n(1-q)^\alpha+m(1-q)^\alpha\sum_{1\leq\ell\leq\frac{n}{m}}\frac{1-\big(1-q^{\ell m}\big)^\alpha}{\big(1-q^{\ell m}\big)^\alpha}\,.
    \end{align*}
    Simplifying the sum on the right using that $\frac{1}{1-q^{\ell m}}\leq\frac{1}{1-q^{m}}$ when $\ell\geq1$, and that $1-\big(1-q^{\ell m}\big)^\alpha\leq\alpha\big(1-\big(1-q^{\ell m}\big)\big)=\alpha q^{\ell m}$, we have
    \begin{align*}
        \sum_{m<k\leq n}\left(\frac{1-q}{1-q^k}\right)^\alpha&\leq n(1-q)^\alpha+m(1-q)^\alpha\sum_{1\leq\ell\leq\frac{n}{m}}\frac{\alpha q^{\ell m}}{\big(1-q^{m}\big)^\alpha}\\
        &\leq n(1-q)^\alpha+\alpha m(1-q)^\alpha\frac{q^{m}}{(1-q^{m})^\alpha}\frac{1-q^n}{1-q^m}.
    \end{align*}
    This corresponds to the first two terms in the desired upper bound and the proposition follows.
\end{proof}

In order to apply the two bounds from this proposition, we now choose the right sequence $(m_n)_{n\geq0}$ corresponding to $(q_n)_{n\geq0}$, so that the first terms in the asymptotic behaviour of $\mu_\alpha(n,q_n)$ correspond to $n(1-q_n)^\alpha$ and $\sum_{1<k\leq m_n}\frac{1}{k^\alpha}$.

\begin{proof}[Proof of Proposition~\ref{prop:ConvOfMu}]
    We want to prove that
    \begin{align*}
        \mu_1(n,q_n)=n(1-q_n)+\log\left(n\wedge\frac{1}{1-q_n}\right)+O\left(\sqrt{\log\left(n\wedge\frac{1}{1-q_n}\right)}\right)\,.
    \end{align*}
    First, in the case where $(q_n)_{n\geq0}$ converges to $0$, since $\frac{1}{1-q_n^k}=1+O(q_n^k)$ uniformly over $k$, it follows from the definition of $\mu_1(n,q_n)$ that
    \begin{align*}
        \mu_1(n,q_n)&=\sum_{1<k\leq n}(1-q_n)\big(1+O(q_n^k)\big)=n(1-q_n)+O(q_n^2)\,.
    \end{align*}
    Since $\log\left(n\wedge\frac{1}{1-q_n}\right)=-\log(1-q_n)\sim q_n$, the desired asymptotic behaviour follows.
    
    We assume now that $(q_n)_{n\geq0}$ is bounded away from $0$. In this case, we prove the lower and the upper bounds separately.

    For the lower bound, define $m_n=\left\lfloor n\wedge\frac{1}{1-q_n}\right\rfloor$. Using the lower bound in Proposition~\ref{prop:BoundsOnMu} and since $m_n(1-q_n)\leq1$, we have
    \begin{align*}
        \mu_1(n,q_n)&\geq(n-m_n)(1-q_n)+\sum_{1<k\leq m_n}\frac{1}{k}=n(1-q_n)+\log m_n+O(1)\,.
    \end{align*}
    This result is actually stronger than what we aim to prove, nevertheless the desired lower bound follows.
    
    For the upper bound, define $m_n=\left\lfloor\frac{n\wedge\frac{1}{1-q_n}}{\sqrt{\log\left(n\wedge\frac{1}{1-q_n}\right)}}\right\rfloor$ and note that, since $q_n$ is bounded away from $0$, we have $m_n(1-q_n)=O(1)$. Using the upper bound of Proposition~\ref{prop:BoundsOnMu}, we have
    \begin{align*}
        \mu_1(n,q_n)\leq n(1-q_n)+m_n\frac{(1-q_n)q_n^{m_n}}{1-q_n^{m_n}}\frac{1-q_n^n}{1-q_n^{m_n}}+\frac{(1-q_n)m_n}{1-q_n^{m_n}}\sum_{1<k\leq m_n}\frac{1}{k}\,.
    \end{align*}
    By studying the variations of the function $\phi:x\mapsto\frac{(1-x)x^m}{1-x^m}$, one can see that it is increasing and smaller than $\frac{1}{m}$ for $x\in[0,1)$, and it follows that
    \begin{align*}
        m_n\frac{(1-q_n)q_n^{m_n}}{1-q_n^{m_n}}&\leq1\,.
    \end{align*}
    Moreover, since $m_n(1-q_n)=O(1)$, we have
    \begin{align*}
        q_n^{m_n}=e^{m_n\log q_n}=e^{-m_n(1-q_n)+O(m_n(1-q_n)^2)}=1-m_n(1-q_n)+O\big((m_n(1-q_n))^2\big)\,,
    \end{align*}
    and so
    \begin{align*}
        \frac{(1-q_n)m_n}{1-q_n^{m_n}}&=\frac{m_n(1-q_n)}{m_n(1-q_n)+O\big((m_n(1-q_n))^2\big)}=1+O\big(m_n(1-q_n)\big)\,.
    \end{align*}
    Combining the last results with the fact that $\sum_{1<k\leq m_n}\frac{1}{k}=\log m_n+O(1)$ and that $1-q_n^n\leq1\wedge\big(n(1-q_n)\big)$, we obtain
    \begin{align*}
        \mu_1(n,q_n)&\leq n(1-q_n)+O\left(\frac{1\wedge\big(n(1-q_n)\big)}{m_n(1-q_n)}\right)+\Big(1+O\big(m_n(1-q_n)\big)\Big)\Big(\log m_n+O(1)\Big)\\
        &=n(1-q_n)+\log m_n+O\left(\frac{1\wedge\big(n(1-q_n)\big)}{m_n(1-q_n)}\right)+O\big(m_n(1-q_n)\log m_n\big)+O(1)\,.
    \end{align*}
    The desired upper bound follows from this formula since $m_n=\left\lfloor\frac{n\wedge\frac{1}{1-q_n}}{\sqrt{\log\left(n\wedge\frac{1}{1-q_n}\right)}}\right\rfloor$, which implies that
    \begin{align*}
        \frac{1\wedge\big(n(1-q_n)\big)}{m_n(1-q_n)}=\Theta\left(\sqrt{\log\left(n\wedge\frac{1}{1-q_n}\right)}\right)\,,
    \end{align*}
    that
    \begin{align*}
        \log m_n=\log\left(n\wedge\frac{1}{1-q_n}\right)+O\left(\log\log\left(n\wedge\frac{1}{1-q_n}\right)\right)\,,
    \end{align*}
    and that
    \begin{align*}
        m_n(1-q_n)\log m_n&=O\left(\sqrt{\log\left(n\wedge\frac{1}{1-q_n}\right)}\right)\,.\qedhere
    \end{align*}
\end{proof}

\begin{proof}[Proof of Proposition~\ref{prop:ConvOfMuAlpha}]
    We use a similar technique, by proving an upper and a lower bound separately, to show that
    \begin{align*}
        \mu_\alpha(n,q_n)=n(1-q_n)^\alpha+\zeta(\alpha)-1+O\left(\left((1-q_n)\vee\frac{1}{n}\right)^\frac{\alpha-1}{\alpha+1}\right)\,
    \end{align*}
    holds for all $\alpha>1$.
    
    Let $m_n=\left\lfloor n\wedge\frac{1}{1-q_n}\right\rfloor$. Using the lower bound of Proposition~\ref{prop:BoundsOnMu}, we have
    \begin{align*}
        \mu_\alpha(n,q_n)&\geq(n-m_n)(1-q_n)^\alpha+\sum_{1<k\leq m_n}\frac{1}{k^\alpha}\,.
    \end{align*}
    Since $\sum_{1<k\leq m_n}\frac{1}{k^\alpha}=\zeta(\alpha)-1+O\left(\frac{1}{(m_n)^{\alpha-1}}\right)$, we obtain
    \begin{align*}
         \mu_\alpha(n,q_n)&\geq n(1-q_n)^\alpha+\zeta(\alpha)-1+O\Big(m_n(1-q_n)^\alpha\Big)+O\left(\frac{1}{(m_n)^{\alpha-1}}\right)
    \end{align*}
    and the lower bound now follows from the fact that $m_n(1-q_n)^\alpha\leq\frac{1}{(m_n)^{\alpha-1}}\leq\left(\frac{1}{m_n}\right)^\frac{\alpha-1}{\alpha+1}$.
    
    For the upper bound, let
    \begin{align*}
        m_n=\left\lfloor\left(\frac{1}{(1-q_n)^2}\wedge\frac{n}{1-q_n}\wedge n^{\alpha+1}\right)^\frac{1}{\alpha+1}\right\rfloor\,.
    \end{align*}
    This definition just encodes that there are three cases to study: $\frac{1}{1-q_n}\leq n$, $n<\frac{1}{1-q_n}\leq n^\alpha$, and $n^\alpha<\frac{1}{1-q_n}$. In all case, we have $1\leq m_n\leq n$. Applying the upper bound of Proposition~\ref{prop:BoundsOnMu}, it follows that
    \begin{align}
        \mu_\alpha(n,q_n)&\leq n(1-q_n)^\alpha+\alpha m_n\frac{(1-q_n)^\alpha q_n^{m_n}}{\big(1-q_n^{m_n}\big)^\alpha}\frac{1-q_n^n}{1-q_n^{m_n}}+\left(\frac{m_n(1-q_n)}{1-q_n^{m_n}}\right)^\alpha\sum_{1<k\leq m_n}\frac{1}{k^\alpha}\,.\label{eq:UBMu}
    \end{align}
    On one hand, by the definition of $m_n$, we have $m_n(1-q_n)\leq\frac{1-q_n}{(1-q_n)^\frac{2}{\alpha+1}}=(1-q_n)^{\frac{\alpha-1}{\alpha+1}}\leq1$. It follows that $1-q_n^{m_n}=\Omega\big(m_n(1-q_n)\big)$, and since $1-q_n^n\leq1\wedge n(1-q_n)$, we have
    \begin{align*}
        \alpha m_n\frac{(1-q_n)^\alpha q_n^{m_n}}{\big(1-q_n^{m_n}\big)^\alpha}\frac{1-q_n^n}{1-q_n^{m_n}}&=\alpha m_n(1-q_n)^\alpha q_n^{m_n}\big(1-q_n^n\big)\cdot O\left(\frac{1}{\big(m_n(1-q_n)\big)^{\alpha+1}}\right)\\
        &=O\left(\frac{1\wedge(n(1-q_n))}{m_n^\alpha(1-q_n)}\right)\,.
    \end{align*}
    On the other hand, 
    we claim that 
    \[
            \left(\frac{m_n(1-q_n)}{1-q_n^{m_n}}\right)^\alpha =1+O\big(m_n(1-q_n)\big). 
    \]
    This holds since if $q_n < 1-\sqrt{2}$ then $m_n=1$, and the claim simply asserts that $1=1+O(1-q_n)$; and if $q_n$ is bounded away from zero then, using the fact that $m_n(1-q_n)\leq1$ again, we have
    \begin{align*}
        \left(\frac{m_n(1-q_n)}{1-q_n^{m_n}}\right)^\alpha=\left(\frac{m_n(1-q_n)}{m_n(1-q_n)+O\big((m_n(1-q_n)^2)\big)}\right)^\alpha=1+O\big(m_n(1-q_n)\big)\,.
    \end{align*}
    
    We also know that $\sum_{1<k\leq m_n}\frac{1}{k^\alpha}=\zeta(\alpha)-1+O\left(\frac{1}{m_n^{\alpha-1}}\right)$, which yields the following bound:
    \begin{align*}
        \left(\frac{m_n(1-q_n)}{1-q_n^{m_n}}\right)^\alpha\sum_{1<k\leq m_n}\frac{1}{k^\alpha}&=\Big(1+O\big(m_n(1-q_n)\big)\Big)\left(\zeta(\alpha)-1+O\left(\frac{1}{m_n^{\alpha-1}}\right)\right)\\
        &=\zeta(\alpha)-1+O\left(\frac{1}{m_n^{\alpha-1}}\right)+O\big(m_n(1-q_n)\big)\,.
    \end{align*}
    Plugging these results back into (\ref{eq:UBMu}), we obtain
    \begin{align*}
        \mu_\alpha(n,q_n)\leq n(1-q_n)^\alpha+\zeta(\alpha)-1+O\left(\frac{1\wedge(n(1-q_n))}{m_n^\alpha(1-q_n)}\right)+O\left(\frac{1}{m_n^{\alpha-1}}\right)+O\big(m_n(1-q_n)\big)\,.
    \end{align*}
    It thus suffices to show that
    \begin{align}
        O\left(\frac{1\wedge(n(1-q_n))}{m_n^\alpha(1-q_n)}\right)+O\left(\frac{1}{m_n^{\alpha-1}}\right)+O\big(m_n(1-q_n)\big)=O\left(\left((1-q_n)\vee\frac{1}{n}\right)^\frac{\alpha-1}{\alpha+1}\right)\,.\label{eq:firstToSecondBound}
    \end{align}
    We divide the proof into three cases according to the value of $q_n$.
    \begin{description}
        \item[First case] $\frac{1}{1-q_n}\leq n$. In this case, we have
        \begin{align*}
            m_n=\left\lfloor\frac{1}{(1-q_n)^\frac{2}{\alpha+1}}\right\rfloor=\Theta\left(\frac{1}{(1-q_n)^\frac{2}{\alpha+1}}\right)\,,
        \end{align*}
        and (\ref{eq:firstToSecondBound}) holds since
        \begin{align*}
            \frac{1\wedge(n(1-q_n))}{(m_n)^\alpha(1-q_n)}=\Theta\left(\frac{(1-q_n)^\frac{2\alpha}{\alpha+1}}{(1-q_n)}\right)=\Theta\Big((1-q_n)^\frac{\alpha-1}{\alpha+1}\Big)\,,
        \end{align*}
        and
        \begin{align*}
            \frac{1}{(m_n)^{\alpha-1}}=\Theta\Big((1-q_n)^\frac{2(\alpha-1)}{\alpha+1}\Big)=O\Big((1-q_n)^\frac{\alpha-1}{\alpha+1}\Big)\,,
        \end{align*}
        and
        \begin{align*}
            m_n(1-q_n)=\Theta\left(\frac{1-q_n}{(1-q_n)^\frac{2}{\alpha+1}}\right)=\Theta\Big((1-q_n)^\frac{\alpha-1}{\alpha+1}\Big)\,.
        \end{align*}
        \item[Second case] $n<\frac{1}{1-q_n}\leq n^\alpha$. In this case, note that $(1-q_n)\leq\frac{1}{n}$. We also have
        \begin{align*}
            m_n=\left\lfloor\left(\frac{n}{1-q_n}\right)^\frac{1}{\alpha+1}\right\rfloor=\Theta\left(\left(\frac{n}{1-q_n}\right)^\frac{1}{\alpha+1}\right)\,,
        \end{align*}
        and (\ref{eq:firstToSecondBound}) holds since
        \begin{align*}
            \frac{1\wedge(n(1-q_n))}{(m_n)^\alpha(1-q_n)}=\Theta\left(n\left(\frac{1-q_n}{n}\right)^\frac{\alpha}{\alpha+1}\right)=\Theta\Big(n^\frac{1}{\alpha+1}(1-q_n)^\frac{\alpha}{\alpha+1}\Big)=O\left(\frac{1}{n^\frac{\alpha-1}{\alpha+1}}\right)\,,
        \end{align*}
        and
        \begin{align*}
            \frac{1}{(m_n)^{\alpha-1}}=\Theta\left(\left(\frac{(1-q_n)}{n}\right)^\frac{\alpha-1}{\alpha+1}\right)=O\left(\frac{1}{n^\frac{2(\alpha-1)}{\alpha+1}}\right)=O\left(\frac{1}{n^\frac{\alpha-1}{\alpha+1}}\right)\,,
        \end{align*}
        and
        \begin{align*}
            m_n(1-q_n)=\Theta\left(\left(\frac{n}{1-q_n}\right)^\frac{1}{\alpha+1}(1-q_n)\right)=\Theta\Big(n^\frac{1}{\alpha+1}(1-q_n)^\frac{\alpha}{\alpha+1}\Big)=O\left(\frac{1}{n^\frac{\alpha-1}{\alpha+1}}\right)\,.
        \end{align*}
        \item[Third case] $\frac{1}{1-q_n}>n^\alpha$. In this case, we have $m_n=n$ and (\ref{eq:firstToSecondBound}) holds since
        \begin{align*}
            \frac{1\wedge(n(1-q_n))}{(m_n)^\alpha(1-q_n)}=\frac{n}{n^\alpha}=O\left(\frac{1}{n^\frac{\alpha-1}{\alpha+1}}\right)\,,
        \end{align*}
        and
        \begin{align*}
            \frac{1}{(m_n)^{\alpha-1}}=\frac{1}{n^{\alpha-1}}=O\left(\frac{1}{n^\frac{\alpha-1}{\alpha+1}}\right)\,,
        \end{align*}
        and
        \begin{align*}
            m_n(1-q_n)=n(1-q_n)<n\cdot\frac{1}{n^\alpha}=O\left(\frac{1}{n^\frac{\alpha-1}{\alpha+1}}\right)\,.
        \end{align*}
    \end{description}
    This concludes the proof of the upper bound and of the proposition.
\end{proof}

\medskip
\bibliographystyle{siam}
\bibliography{2020-07-24_mallows}

\end{document}